\documentclass[a4paper,11pt]{article}
\usepackage{amsmath,amsthm,amsfonts,amssymb,bbm,stmaryrd,scalerel}
\usepackage{graphicx,psfrag,tikz}
\usepackage{caption}
\usepackage{subcaption}
\usepackage{cite}
\usepackage{hyperref}
\usepackage{bm}
\usepackage{enumitem}
\usepackage[hmargin=1in,vmargin=1in]{geometry}
\usepackage{stmaryrd}
\usepackage{dsfont} 
\usepackage{xcolor}
\hypersetup{
	colorlinks,
	linkcolor={red!50!black},
	citecolor={blue!50!black},
	urlcolor={blue!80!black}
}
\renewcommand{\P}{\mathbb{P}}

\newcommand{\N}{\mathbb N}
\newcommand{\R}{\mathbb R}
\newcommand{\Z}{\mathbb Z}
\newcommand{\Q}{\mathbb Q}

\newcommand{\inti}[1]{{\llbracket #1 \rrbracket}}
\newcommand{\Omq}{\Omega_{\text{\tiny $\mathbb{Q}$}}}

\newcommand{\cB}{\mathcal{B}}

\newcommand{\cG}{\mathcal{G}}

\newcommand{\cL}{\mathcal{L}}

\newcommand{\cU}{\mathcal{U}}

\newcommand{\cdir}{c_d}
\newcommand{\disjc}{\mathrm{Disj}}
\newcommand{\disjcg}{\overline{\mathrm{Disj}}}
\newcommand{\tdis}{\mathrm{3Close}}
\newcommand{\Mdisjc}{\mathrm{MaxDisjCls}}
\newcommand{\Mbox}{M}
\newcommand{\smbox}{\mathcal{R}}
\newcommand{\cone}{\mathcal{C}}
\newcommand{\coneu}{\mathcal{C}^{up}}
\newcommand{\coned}{\mathcal{C}^{dn}}

\newcommand{\Ind}{\mathcal{E}}
\newcommand{\brnc}[1]{\mathrm{Brnch}(#1)}
\newcommand{\trpz}{\mathcal{T}}
\newcommand{\bqtu}{\mathrm{Bqt}^{\mathrm{up}}}
\newcommand{\bqtd}{\mathrm{Bqt}^{\mathrm{dn}}}

\newcommand{\Rufk}{\mathfrak{X}_{\uparrow}}

\newcommand{\el}{\mathcal{B}}
\newcommand{\ele}{\mathfrak{A}} 
\newcommand{\gle}{\mathfrak{L}} 
\newcommand{\els}{\mathfrak{C}} 
\newcommand{\NT}{\mathrm{NoTouch}}
\newcommand{\Cl}{\mathrm{Close}}

\newcommand{\uni}{\cU}

\newcommand{\Mi}{U}

\newcommand{\sig}{{\scaleobj{0.8}{\boxempty}}} 
\newcommand{\sigg}{{\scaleobj{0.9}{\boxempty}}} 

\newcommand{\Thg}{\mathrm{3G}}
\newcommand{\nThg}{\mathrm{NO3G}}

\newcommand{\dl}{\cL}

\newcommand{\paths}[2]{\Pi_{[{#1},{#2}]}}
\newcommand{\ash}{\mathcal{K}}
\newcommand{\argmax}[1]{\underset{#1}{\mathrm{argmax}}}

\newtheorem{theorem}{Theorem}[section]
\newtheorem{lemma}[theorem]{Lemma}
\newtheorem{proposition}[theorem]{Proposition}
\newtheorem{corollary}[theorem]{Corollary}

\theoremstyle{definition}
\newtheorem{definition}[theorem]{Definition}
\newtheorem{remark}[theorem]{Remark}

\numberwithin{equation}{section}
\numberwithin{figure}{section}

\author{Ofer Busani\thanks{School of Mathematics, University of Edinburgh, James Clerk Maxwell Building, Peter Guthrie Tait Road,
		Edinburgh, EH9 3FD, UK. E-mail: {\tt obusani@ed.ac.uk}} 
}
\title{Non-existence of three non-coalescing infinite geodesics with the same direction in the directed landscape}

\date{\today}

\begin{document}

\sloppy

\maketitle
\begin{abstract}
	It is believed that for metric-like models in the KPZ class the following property holds: with probability one, starting from any point, there  are at most two semi-infinite geodesics  with the same direction that do not coalesce. Until now, such a result was only proved for one model - exponential LPP \cite{C11} using its inherent connection to the totally asymmetric exclusion process. We prove that the above property holds for the directed landscape, the universal scaling limit of models in the KPZ class. Our proof reduces the problem to one on line ensembles and therefore paves the way to show similar results for other metric-like models in the KPZ class. Finally, combining our result with the ones in \cite{BSS22, bhatia2023duality} we obtain the full qualitative geometric description of infinite geodesics in the directed landscape. 
\end{abstract}
\tableofcontents

\section{Introduction and main results}
\subsection{Random growth interfaces in the KPZ class}
In their seminal paper \cite{KPZ86} Kardar, Parisi and Zhang put forward a singular stochastic partial differential equation as a model for the random growth of many natural phenomena such as accumulation of snow, fire fronts, bacterial colonies etc.\ 
Based on this, it was conjectured in \cite{KPZ86} that universal scaling exponents of $1/3$ and $2/3$ should govern the fluctuation of a large class of random growth interface models satisfying some natural properties such as smoothness, lateral growth and short range correlation of the dynamics. 
Their work was ensued with intense investigation in the physics and mathematics literature that resulted in a large set of models that are believed to exhibit universal long time fluctuations. This class of models is referred to as the KPZ class. 

Our understanding of the exact nature of this universality of fluctuations has evolved over the past quarter of a century (see surveys  \cite{Corwin-survey,Zygouras22}).  A major milestone came in \cite{KPZfixed}, where the authors showed that under proper scaling, the dynamics of the fluctuations of models in the KPZ class converge to the \textit{KPZ fixed-point} - a Markov process living in the space of functions. 


Starting from two initial conditions, one may couple the dynamics of the two  interfaces by applying the same driving noise to both initial datum. One may then ask, is there a scaling limit for the joint distribution of two interfaces? The answer to that was obtained  in \cite{DOV22}, where the authors constructed the directed landscape (DL) - a continuous function $\dl$ on $\R^4_{\uparrow}$ that allows one to apply the KPZ fixed-point dynamics via a variational formula involving the initial datum and the DL. In that sense, the DL is the most general scaling limit in the KPZ class. 
\subsection{Geodesics in the KPZ class}
Variational formulas are not unique to the DL, in fact many prelimiting models in the KPZ class satisfy some form of such formulas, which in turn endow the model (and its associated interface) with a dynamics. Just as in the theory of parabolic PDEs, maximizers of variational formulas give rise to geodesics - continuous paths that evolve in time and essentially trace the dynamics of the initial condition in the space-time plane. Geodesics are instrumental in the study of many aspects of KPZ models  such as the regularity of  interfaces \cite{BBS20,busa-ferr-20},  space time correlation \cite{FO19,BG21,FO22,BGZ21},  fractal dimensions of the difference profile \cite{BGH21,BGH22}, a solution to the slow bond problem \cite{BSS16} and mixing times \cite{S23,SS22}.

Being such an important tool, the geodesics themselves have been the focus of a vast study including questions about  their fluctuations \cite{Joansson-03,BSS16,BSS19}, modulus of continuity \cite{H19,HS20}, coalescence \cite{P16,Z20, BSS19, SS20,DEP23} and recently their occupation time \cite{Ganguly-Zhang-2022,GZ22}. 

In models such as first and last passage percolation geodesics appear naturally as the minimizers/maximizers of the passage time between any two points. Thus, first and last passage percolation models can  (and should) be viewed as random-metric models. As these models are not necessarily bona fide metrics, we shall refer to them colloquially as metric-like models. One such model is the DL itself, what makes it unique, however, is  that it is conjectured to be the scaling limit of metric-like models in the KPZ class in a strong sense i.e.\ geodesics of metric-like models should converge to the geodesics of the DL under the KPZ scaling. This was proved for a handful of integrable models in \cite{DOV22,Dauvergne-Virag-21,Wu23}. 

One direction in the study of geodesics in the KPZ class is that of infinite geodesics. These are infinite continuous (with respect to the natural topology on the graph $\mathcal{G}$ of the model at hand) paths such that the restriction  between any two points on the path is again a geodesic. The study of such objects can be traced back to the 90's \cite{HN97,HN01,LN96,N95} in the context of first passage percolation as well as \cite{AH19}, and continued to their counterparts in Last passage percolation \cite{FP05,FMP09,CP11,GRAS17a,SS23a}.  One of the main motivation behind studying such objects is that infinite geodesics allow one to construct  a family of random processes $\{\cB_\theta\}_{\theta\in\Theta}$ that are stationary (up to a constant) with respect to the dynamics coming from the variational formula \cite{H08,CP12,DH14,GRS17b}. The set $\Theta$ is that of all possible directions and depends on the model. One can then show that, with probability one, for any $\theta\in\Theta$ the \textit{Busemann function} $\cB_\theta$ lends itself to a construction of infinite geodesics starting from any point on the graph and have a direction  $\theta$ \cite{JRAS19,SS23b,BSS22}. Such geodesics are called \textit{Busemann gedoesics}.

\subsubsection{The behaviour of infinite geodesics at directions of non-uniqueness and the N3G problem}

Let us, in broad strokes, paint the picture of our knowledge of infinite geodesics in the KPZ class. For the sake of brevity and clarity, we restrict the discussion here  to (mostly) exponential LPP, arguably the prelimiting model for which most is known about infinite geodesics. In \cite{FP05} it was shown  that  with probability one, all geodesics are directed, every direction has a geodesic directed at it,  and there exists a set of full Lebesgue measure for which uniqueness holds. Using the connection between second class particle and the competition interface, it was shown in  \cite{FMP09} that there are random directions at which at least two geodesics  emanating from a shared point are directed.  In \cite{C11}  the author utilized a deep connection between exponential LPP and the TASEP speed process to show that from any point on the lattice emanate \textbf{at most} two infinite geodesics.  Non-existence of bi-infinite geodesics was resolved in \cite{BHS22} (see also \cite{BBS20b} for a different proof).  In \cite{JRAS19}, the authors identified the set of random directions $\Xi\subset \Theta$ of non-uniqueness of geodesics with the jump points of the Busemann process and obtained the full geometric description of infinite geodesics in Exponential LPP. In particular, the results in \cite{JRAS19} strengthen the results in  \cite{C11} by showing that there exist no three non-coalescing infinite geodesics (not necessarily emanating from the same point) in the same direction. 

The picture described above is believed to hold in general for metric-like models in the KPZ class. In particular, the results in \cite{C11} suggest that  for any metric-like model in the KPZ class, with probability one, for any $x\in\cG$, there exists no three geodesics starting from $x$, sharing the same direction and do not coalesce at infinity - a property we shall refer to as the \textit{ N3G property}. 

 In  \cite{GRAS17a} infinite geodesics for lattice  LPP with quite general weights were considered. The work therein  implies that at a direction of non-uniqueness $\xi\in\Xi$,  Busemann geodesics (mentioned above) make up two of the geodesics directed at $\xi$. Thus, the N3G property is equivalent to showing that all semi-infinite geodesics  are Busemann geodesics. The latter, are far better understood as they were originally constructed via the Busemann process which is, as least for some models, fairly understood \cite{Fan-Seppalainen-20}. Thus, the N3G problem is not merely a quantitative question about geodesic but in fact lies at the heart of our understanding of infinite geodesics in the KPZ class.

 In  \cite{JRAS19} the authors obtained the  full global behaviour of infinite geodesics of the exponential LPP by resolving the coalescence  of Busemann geodesics in that model. Until now this  has been  the only instance of a complete geometric description of infinite geodesics for a model in the KPZ class. Although much of the machinery and ideas in  \cite{JRAS19} can be generalized for other models, a key input in proving  their results is the N3G property proved in \cite{C11}.  Consequently, the N3G problem  has been a major hurdle in obtaining a complete description of the global behaviour of geodesics in these models \cite{SS23b,BSS22}.

Although being such a central question in the study of infinite geodesics, the N3G property has not been proved for any other model but the exponential LPP. This is partly due to that the  proof  in \cite{C11} relies on the inherent connection between the  exponential LPP and the exclusion process and therefore does not generalize to other metric-like models. 
 
\subsection{Solving the N3G for the directed landscape via line ensembles} 
 The main result in this paper is Theorem \ref{thm:4}, where we resolve the N3G problem in the DL. As mentioned earlier, a similar result has only been proven for one model - the Exponential LPP in \cite{C11}. Perhaps as important, we believe our proof can be generalized to  both zero and (possibly) positive temperature integrable models in the KPZ class. To illustrate this, let us outline the main ideas behind the proof.  
 
 {\it{Sketch of the proof of Theorem \ref{thm:4}.}} Let us suppose that there were three non-coalescing infinite geodesics with some random direction $\xi\in\R$, emanating from the point $(0,0)$. 
 \begin{enumerate}[label = Step \arabic*]
 	\item \label{st1} Figure \ref{fig:5}. Bounds on the fluctuations of infinite geodesics suggest that at time $T$ these geodesics should not be too far from one another; these geodesics should be found within a spatial interval of length  $O(T^{2/3})$. For a large time $T$,  the three geodesics must have branched off which implies that one can find three disjoint geodesics in a (random) trapezoid with lower base of order $O(1)$ and upper base of order $O(T^{2/3})$. Since we would like to cover all possible directions, we must cover a spatial interval of order $T$, which translates into  $O(T^{1/3})=O(T/T^{2/3})$ possible such trapezoids. If we would like to use a union bound on all possible trapezoids, we must show that the probability of three disjoint geodesics in a trapezoid is $o(T^{-1/3})$.
 	\item \label{st2} Figure \ref{fig:6}. Scaling properties of the DL and some observations regarding disjoint geodesics show that the previous step is equivalent to showing that the probability of two pairs of disjoint geodesics with $\epsilon$-close endpoints in a trapezoid of lower and upper bases of length $\epsilon$ and (lets say) $\log(\epsilon^{-1})$ respectively,  is $o(\epsilon^{1/2})$. 
 	\item \label{st3} Figure \ref{fig:7}. By appending disjoint paths to their endpoints, one can turn each of the pairs of geodesics in the previous step into  two disjoint paths with shared endpoints. 
 	\item \label{st4} Figure \ref{fig:8}. Using a version of the Greene's Theorem for the DL obtained in \cite{Dauvergne-Zhang-2021}, one can show that the probability of the event in the previous step can be bounded by the probability that the first two lines of the Airy lines ensemble are $\epsilon^{1/2}$ close to each other at any two points in an interval of order $\log(\epsilon^{-1})$, and that the latter probability  is $o(\epsilon^{1/2})$. This bound in turn, can be obtained using various techniques such as the ones  in  \cite{Hammond1,dauvergne2023wiener}, which rely  on the Brownian Gibbs property of the Airy line ensemble. 
 \end{enumerate}
  Both positive and zero temperature integrable metric-like models give rise to Gibbsian line ensembles. As such, we believe these models are amenable to an analysis similar to that described in Steps 1--4 above. 
  
  The connection between finite geodesics and the closeness of line ensembles has appeared earlier in the literature. In \cite{Hammond1,Hammond2}, the author obtained an upper bound on the probability of $k$  finite geodesics  having close endpoints. Similar ideas were used in \cite{Dauvergne23} to characterize all possible networks in the DL. However, in the context of infinite geodesics the use of this ideas seems to be  new, especially in the context of the N3G problem.    
 
 Using the results in \cite{BSS22,bhatia2023duality} and Theorem \ref{thm:4} we are able to complete the global description of infinite geodesics in the DL in Theorem \ref{thm:2}. As mentioned, a similar result has only been available for one other model in the KPZ class - the exponential LPP in \cite{JRAS19}, via the results in \cite{C11}. \vspace{5pt}\\

 To conclude, the main contribution of this paper is threefold
 \begin{enumerate}
 	\item We resolve the N3G problem for the directed landscape - the scaling limit of all models in the KPZ class.
 	\item Our proof  relies on the connection between random growth models and line ensembles, and therefore can be generalized to other metric-like models (and possibly to positive temperature models). We peruse this in a forthcoming work \cite{BS23}.
 	\item We complete the global geometric description of infinite geodesics in the directed landscape.
 \end{enumerate}

The rest of the paper is organized as follows. In the rest of this section we give  precise statements of our results. Section \ref{sec:DL} is dedicated to a brief account of the DL, the extended DL and its properties. In Section \ref{sec:fgeo} we essentially prove  \ref{st1}, i.e.\ showing that the N3G problem can be reduced to  finding a tight enough bound on the probability of  three disjoint geodesics in a trapezoid. Then, assuming such a bound in Proposition \ref{prop:2}, we prove Theorems \ref{thm:4} and \ref{thm:2}. The rest of the paper is devoted to  proving Proposition \ref{prop:2}. In Section \ref{sec:trp} we obtain \ref{st2}. In Section \ref{sec:1} we explain the main ideas behind the connection between geodesics with close endpoints and closeness of line ensembles. In Section \ref{sec:ubt} we  apply the strategy from Section \ref{sec:1} to   show how to go from  \ref{st2} via \ref{st3} to \ref{st4}, and prove Proposition \ref{prop:2}. Finally, in Section \ref{sec:Alb} we obtain the necessary  bound in \ref{st4}, i.e.\ we show that the probability that the first two lines of the Airy line ensemble are $\epsilon^{1/2}$ close to one another at any two spatial points in a large interval is of order $o(\epsilon^{1/2})$. This result is used as an input for Section \ref{sec:ubt}.

 \subsection{Acknowledgements}
 The author would like to thank Evan Sorensen for  reading earlier versions of this paper, and whose comments helped to improve the paper significantly. The author would also like to thank  Duncan Dauvergne for much helpful discussions about Section  \ref{sec:Alb} of the paper that resulted in a considerably simplified proof, and  Xuan Wu for an illuminating discussion at the workshop on 'The asymmetric exclusion process' at the Simons Center for Geometry and Physics. 
 
 The work of O.B. was partially supported by the Deutsche Forschungsgemeinschaft (DFG, German Research Foundation) under Germany’s Excellence Strategy--GZ 2047/1, projekt-id 390685813.

\subsection{Main results}
The directed landscape (DL) $\dl$ is a continuous function on $\R^4_\uparrow:=\{\bm{u}=(x,s;y,t)\in\R^4: s\leq t \}$, originally constructed in \cite{DOV22} as the scaling limit of  Brownian last passage percolation. It is believed to be the universal scaling limit of models in the KPZ universality class, which  has been verified for several integrable models in \cite{Dauvergne-Virag-21,Wu23}. The random function $\dl$ assigns weight on continuous paths $\pi:[s,t]\rightarrow \R$ via
\begin{equation}\label{eq136}
	\dl(\pi)=\inf_{k\in \N}\,\,\inf_{s=t_0<t_1<\ldots<t_k=t} \sum_{i=1}^k\dl\big([\pi(t_{i-1})],[\pi(t_i)]\big),
\end{equation}
where $[\pi(r)]=\big(\pi(r),r\big)\in\R^2$. From the triangle inequality 
\begin{equation}
	\dl(p,q)\geq \dl(p,u)+\dl(u,q) \qquad \forall (p,q),(p,u),(u,q)\in\R^4_\uparrow,
\end{equation}
it follows that 
\begin{equation}
\dl(\pi)\leq \dl\big([\pi(s)],[\pi(t)]\big).
\end{equation}
A path $\pi$ is called a geodesic if the last display holds with equality. It is not hard to verify that if $\pi:[s,t]\rightarrow \R$ is a geodesic then any restriction of $\pi$ to a subinterval of $[s,t]$ is again a geodesic. Between any two points $u,v\in\R_\uparrow^4$, with probability 1,  there exists a unique geodesic between the two points i.e.\
\begin{equation}\label{eq122}
	\P(\text{there exists a unique geodesic between the points $u$ and $v$})=1\qquad \forall u,v\in \R^4_{\uparrow}.
\end{equation}
 There are, however, exceptional points on the plane where uniqueness of geodesics fails, and give rise to interesting fractal phenomena \cite{BGH21, BGH22, Dauvergne23}. 
 
 An infinite geodesic is an infinite path $\pi:[s,\infty)\rightarrow \R$ such that the restriction of $\pi$ on $[r,t]$ ($r\geq s$) -  $\pi|_{[r,t]}$ is a geodesic. We say an infinite geodesic $\pi:[s,\infty)\rightarrow \R$ is $\xi$-directed for some $\xi\in\R$, if $\pi(r)/r\rightarrow\xi$ as $r\rightarrow \infty$, in particular infinite geodesics in this work evolve upward. In the DL, infinite geodesics were first constructed and studied in \cite{Rahman-Virag-21} when either the point from which the geodesics emanate or the direction in which they are directed is fixed. These results were then refined in \cite{BSS22} and obtained independently in \cite{Ganguly-Zhang-2022}. We will dive into the intricate properties of  infinite geodesics in the DL in a moment, but for now let us just mention that (\cite{Rahman-Virag-21,Ganguly-Zhang-2022})
 \begin{equation}\label{eq135}
 	\begin{aligned}
 		&\text{from any \textbf{fixed}  point $p\in\R^2$ and $\xi\in\R$, with probability one,  there exists a} \\
 		&\text{ unique geodesic $\pi^\xi_p$ starting from $p$ in the direction $\xi$.  }
 	\end{aligned}
 \end{equation}  
\subsubsection*{The Busemann process and Busemann geodesics}
One of the main tools in studying infinite geodesics in the DL and other metric-like models is  the \textit{Busemann process}. Fix $\xi\in\R$, then the  Busemann function associated with $\xi$
\begin{equation}\label{eq114}
	\{W_{\xi}(p,q):p,q\in\R^2\}
\end{equation}
retains much information on infinite geodesics in the DL going in direction $\xi$. In the DL, the Busemann function was first constructed in \cite{Rahman-Virag-21} and independently in \cite{Ganguly-Zhang-2022}. For any two  points on the plane $p$ and $q$, the value $W_\xi(p,q)$ can be obtained as the difference of the $\dl$-distance between the points $p$ and $q$ to the point of coalescence of the two infinite geodesics emanating from $p$ and $q$ i.e.\
\begin{equation}
	W_{\xi}(p,q)=\dl(p,z_c)-\dl(q,z_c)
\end{equation}
where the point $z_c:=[\pi_p^\xi(t_c)]$, with $t_c=\inf\{r\in\R:\pi^\xi_p(r)\cap\pi^\xi_q(r)\neq \emptyset\}$, is the point of coalescence of $\pi_p^\xi$ and $\pi_q^\xi$. 
 More precisely,   for a fixed direction $\xi\in\R$ it was shown in \cite{Rahman-Virag-21} that for any sequence $(z^n_1,z^n_2)\in\R^2$ such that $z^n_1/z^n_2\rightarrow \xi$, there exists $N\in\N$ such that for any $n\geq N$ it holds that  
\begin{equation}\label{eq115}
	W_\xi(p,q)=\dl(p,z_n)-\dl(q,z_n).
\end{equation}
 For $s<t$ and $x,y\in\R$ it holds 
 \begin{equation}\label{eq112}
 	W_\xi(x,s;y,t)=\sup_{z\in\R}\big\{ \dl(x,s;z,t)+W_\xi(z,t;y,t)\big \}.
 \end{equation}
The variational identity above gives rise to one direct connection between the Busemann process and infinite geodesics -   the supremum in \eqref{eq112} is attained exactly at points $z$ such that $(z,t)$ lies on an  infinite geodesic starting from the point $(x,s)$. Thus, by varying the variable $t$ form $s$ to infinity, \eqref{eq112} allows us to etch the paths of all possible geodesics going in direction $\xi$ from the point $(x,s)$. It was shown in \cite{Rahman-Virag-21} that  $x\mapsto W_\xi(0,t;x,t)$ is distributed as a Brownian motion with diffusivity $\sqrt{2}$ and drift $2\xi$. 

From the properties of the DL and \eqref{eq112} it is then possible to refine the picture in \eqref{eq135} - it was shown in \cite{Rahman-Virag-21}  that given a fixed direction $\xi$, all geodesics in direction $\xi$ starting from any point on the plane  must coalesce. 
 
 In order to understand the behaviour of infinite geodesics going in a \textit{random direction}, one has to consider a coupling of all Busemann functions on one event of full probability. This was obtained in \cite{BSS22}, where the Busemann process
 \begin{equation}\label{eq113}
 	\{W_{\xi\sig}(p,q):\xi\in \R,\sigg\in\{+,-\},p,q\in\R^2\}
 \end{equation}
  was constructed. Note the difference between \eqref{eq114} and \eqref{eq113}, where in the latter we added the parameterization by $\xi\sigg$. The  function  $\xi\mapsto W_{\xi\sigg}(p,q)$ with $\sigg=+$ ($\sigg=-$) is a right-continuous (left-continuous) function obtained as a right(left) continuation of the Busemann function from \eqref{eq115} evaluated at rational directions. The Busemann function $W_{\xi\sig}(\cdot,\cdot)$ with $\sig=+$ ($\sig=-$) holds much information about the behavior of infinite geodesics just to the right (left) of the geodesics going in the direction $\xi$. These Busemann functions give rise to two families of $\xi$-directed infinite geodesics called Busemann geodesics. We now  define this rigorously. For $(x,s)\in\R^2$, $\xi\in\R$ and $\sigg\in\{-,+\}$ define
  \begin{equation}
  	A_{(x,s)}^{\xi\sig}=\{(y,t)\in \R\times\Q:t>s,y\in\argmax{z}\,\{\dl(x,s;z,t)+W_{\xi\sig}(z,t;0,t)\},
  \end{equation}
  and the set
  \begin{equation}
  	\mathcal{A}_{(x,s)}^{\xi\sig}=\text{Cls}(A_{(x,s)}^{\xi\sig}),
  \end{equation}   
  where for a set $A\subseteq \R^2$, we let $\text{Cls}(A)$ denote its closure in $\R^2$. For a geodesic $\pi:[s,t]\rightarrow\R$, we let $[\pi]=[\pi(r)]_{r\in[s,t]}$ denote the graph of $\pi$ in $\R^2$. Similarly we define $[\pi]$ for infinite geodesics. 
  \begin{definition}\label{def:1}
  	Let $(x,s)\in \R^2,\xi\in \R$ and $\sigg\in \{-,+\}$. A $\xi\sigg$ \textit{Busemann geodesic} starting from $(x,s)$  is a continuous path $\pi:[s,\infty)\rightarrow \R$ such that $\pi(s)=x$ and $[\pi]\subseteq \mathcal{A}_{(x,s)}^{\xi\sig}$. 
  \end{definition}
  \begin{remark}
  	Definition \ref{def:1} is different than \cite[Definition 5.10]{BSS22} but equivalent to it. Arguably, the advantage of Definition \ref{def:1} is that it clearly shows how Busemann functions give rise to Busemann geodesics.   
  \end{remark}
  \begin{definition}
  	The leftmost Busemann geodesic $\pi_{(x,s)}^{\xi\sig,L}$  is a Busemann geodesic starting from $(x,s)$ in direction $\xi\sig$ such that for any other Busemann geodesic $\eta$ starting from  $(x,s)$ in direction $\xi\sig$ it holds 
  	\begin{equation}
  		\pi_{(x,s)}^{\xi\sig,L}(t)\leq \eta(t).
  	\end{equation}
  	 Similarly we define the rightmost Busemann geodesic $\pi_{(x,s)}^{\xi\sig,R}$ by reversing the inequality in the last display.
  \end{definition}
  Next we describe the random set $\Xi$ of non-uniqueness of geodesics. 
   For $p,q\in\R^2$ we define the sets 
  \begin{equation}
  	\Xi(p;q)=\{\xi\in\R:W_{\xi-}(p;q)\neq W_{\xi+}(p,q)\} \quad \text{ and }\quad  \Xi=\bigcup_{p,q\in\R^2}\Xi(p;q).
  \end{equation}
  It was shown in \cite[Theorem 5.5 (iii)]{BSS22} that the set $\Xi$ is almost surely dense and countable. Moreover, for a fixed $\xi\in\R$ it holds that $\P(\xi\in\Xi)=0$. The next result gives  more information about the behavior of infinite geodesics via the relation between the Busemann geodesics and the set $\Xi$.
  \begin{theorem}\label{thm:3}
  	The following hold on a single event of probability one.
  	\begin{enumerate}[label=\normalfont(\roman*)]
  		\item {\cite[Proposition 33]{bhatia2023duality}}\label{ita} All infinite geodesics in the directed landscape are semi-infinite i.e.\ there are no bi-infinite geodesics in the directed landscape.
  		\item \cite[Theorem 6.5]{BSS22} For all $(x,s)\in \R^2$ and $\xi\in\R$, $\pi_{(x,s)}^{\xi-,L}(t)$ is the leftmost and $\pi_{(x,s)}^{\xi+,R}(t)$ the rightmost infinite geodesic  in direction $\xi$, i.e.\ for any infinite geodesic starting from $(x,s)$ in direction $\xi$
  		\begin{equation}
  			\pi_{(x,s)}^{\xi-,L}(t)\leq \eta(t) \leq \pi_{(x,s)}^{\xi+,R}(t) \quad \text{ for all $t\geq s$}.
  		\end{equation}
  		\item \cite[Theorem 7.1 (i)]{BSS22} For all $p,q\in\R^2$, if $\pi_1$ and $\pi_2$ are $\xi\sigg$ Busemann geodesics from $p$ and $q$ respectively, then $\pi_1$ and $\pi_2$ coalesce. If the first point of intersection of the two geodesics is not $p$ or $q$, then the first point of intersection is the coalescence point of the two geodesics.
  		\item \cite[Theorem 7.3]{BSS22} The following are equivalent
  		\begin{itemize}
  			\item $\xi\notin \Xi$
  			\item All semi-infinite geodesics in direction $\xi$ coalesce.
  		\end{itemize}
  	\end{enumerate}
  \end{theorem}
  To sum up in words, the picture of infinite geodesics in the DL is as follows. All infinite geodesics are semi-infinite. Typically, for a fixed point $p\in\R^2$ and a direction $\xi\in\R$ there exists a unique geodesic starting from $p$ and going in direction $\xi$.  There  exists however,  a random dense and countable set of exceptional directions $\Xi$ for which the following holds with probability one; if  $\xi\notin\Xi$, there are exceptional points in $\R^2$ for which at least two geodesics sharing a starting point split immediately and later coalesce.  If $\xi\in \Xi$ then form every point $p\in\R^2$ emanate two non-coalescing $\xi$-directed geodesics  sandwiching all geodesics emanating from $p$, these are the Busemann geodesics $\pi_{p}^{\xi-,L}$ (leftmost) and $\pi_{p}^{\xi+,R}$(rightmost).  
  
  What is therefore left open  is what happens between  the geodesics $\pi_{p}^{\xi-,L}$ and $\pi_{p}^{\xi+,R}$; Are there other infinite geodesics  directed at $\xi$ that do not coalesce with these two geodesics? Our main result answers this question in somewhat more generality.
   \begin{theorem}[Solution to the N3G problem in the DL] \label{thm:4}
  	With probability one, for any direction $\xi\in\R$ there are no three infinite non-coalescing  $\xi$-directed geodesics.
  \end{theorem}
  The following result is an immediate corollary. 
  \begin{corollary}\label{thm:5}
  	With probability one, for any point on the plane $p\in \R^2$ and a direction $\xi\in\R$ there exists at most two infinite non-coalescing  $\xi$-directed geodesics starting from the point $p$.  
  \end{corollary}
  Combining Corollary \ref{thm:5} with the results in \cite{BSS22,bhatia2023duality} (Theorem \ref{thm:3}) we obtain a characterization of  geodesics in the DL.
  \begin{theorem}\label{thm:2}
  	With probability one, all infinite geodesics in the directed landscape are Busemann. 
  \end{theorem}
  Combining our results with Theorem \ref{thm:3} we obtain the full qualitative picture of infinite geodesics in the directed landscape.
  \begin{theorem}\label{thm:1}
  	The following statements hold on a single event of full probability. There exists a random countably infinite dense subset $\Xi$ of $\R$ such that the following hold
  	\begin{enumerate}[label=\normalfont(\roman*)]
  		\item {\cite[Proposition 33]{bhatia2023duality}} All infinite geodesics are semi-infinite.
  		\item \cite[Theorem 6.3]{BSS22} Every semi-infinite geodesic has a direction $\xi\in \R$. From each initial point $p\in\R^2$ and in each direction $\xi\in\R$, there exists at least one semi-infinite geodesic from $p$ in direction $\xi$.
  		\item \cite[Theorem 7.1 (i)]{BSS22} When $\xi\notin\Xi$, all semi-infinite geodesics in direction $\xi$ coalesce. There exists a random set of inital points, of zero planar Lebesgue measure, outside of which the semi-infinite geodesic in each direction $\xi\in\Xi$ is unique. 
  		\item \label{it:ngr} When $\xi\in\Xi$, there exists \textbf{exactly} two families of semi-infinite geodesics in direction $\xi$, called the $\xi-$ and $\xi+$ geodesics. From every initial point $p\in\R^2$ there exists both  $\xi-$ geodesics and $\xi+$ geodesics which eventually  separate and never come back together. All $\xi-$ geodesics coalesce, and all $\xi+$ geodesics coalesce. 
  	\end{enumerate}
  	\begin{proof}
  		Item \ref{it:ngr} was proven in \cite[Theorem 2.5]{BSS22} with the word 'at least' instead of  'exactly'. Our Corollary \ref{thm:4} allows us to make this replacement in the statement of the result. 
  	\end{proof}
  \end{theorem}
  \subsection*{Notation}
  We use $w,x,y,z$ to denote values in $\R$. We use $p,q,u,v$ to denote points in $\R^2$.  For $a<b$, we denote by $\llbracket a,b \rrbracket $ all integers between $a$ and $b$ i.e.\ $\llbracket a,b \rrbracket =\Z\cap [a,b]$. If $p\in \R^2$, we denote by $p_1$ and $p_2$ the elements of $p$, i.e.\ $p=(p_1,p_2)$. We let $\R^k_{>}$ denote the set of ordered $k$-tuples, i.e.\ $(x_1,\ldots,x_k)$ where $x_1>x_2\ldots>x_k$.
   \begin{figure}[!h]
  	\centering
  	\begin{tabular}[c]{cc}
  		\begin{subfigure}[c]{0.5\textwidth}
  			\includegraphics[width=8cm]{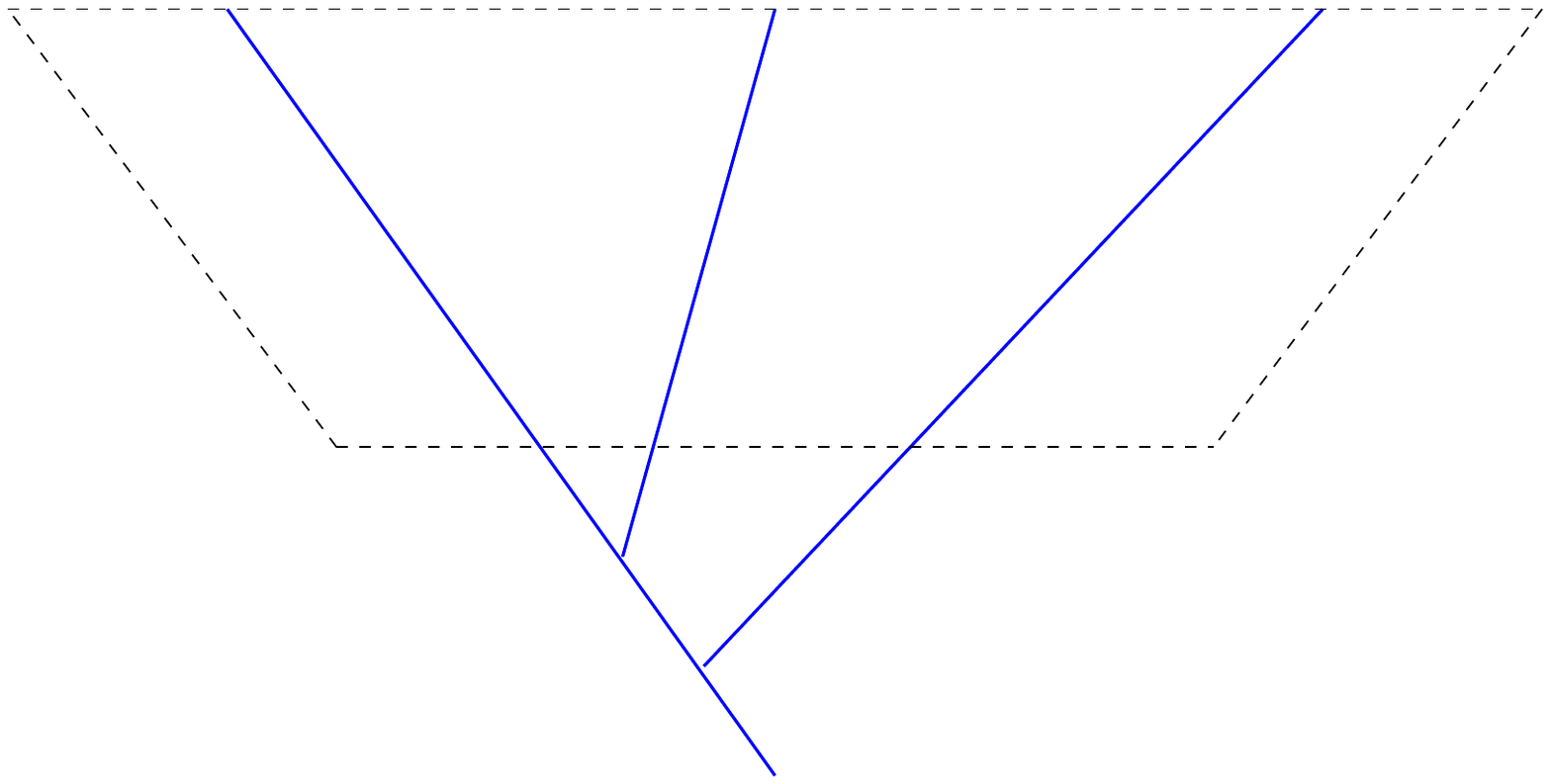}
  			\caption{The existence of a point from which at least three geodesics with the same direction emanate implies that one of $T^{1/3}$ trapezoids of lower and upper bases of order $O(1)$ and $O(T^{2/3})$ respectively, has three disjoint geodesics.}
  			\label{fig:5}
  		\end{subfigure}&
  		\begin{subfigure}[c]{0.5\textwidth}
  			\includegraphics[width=8cm]{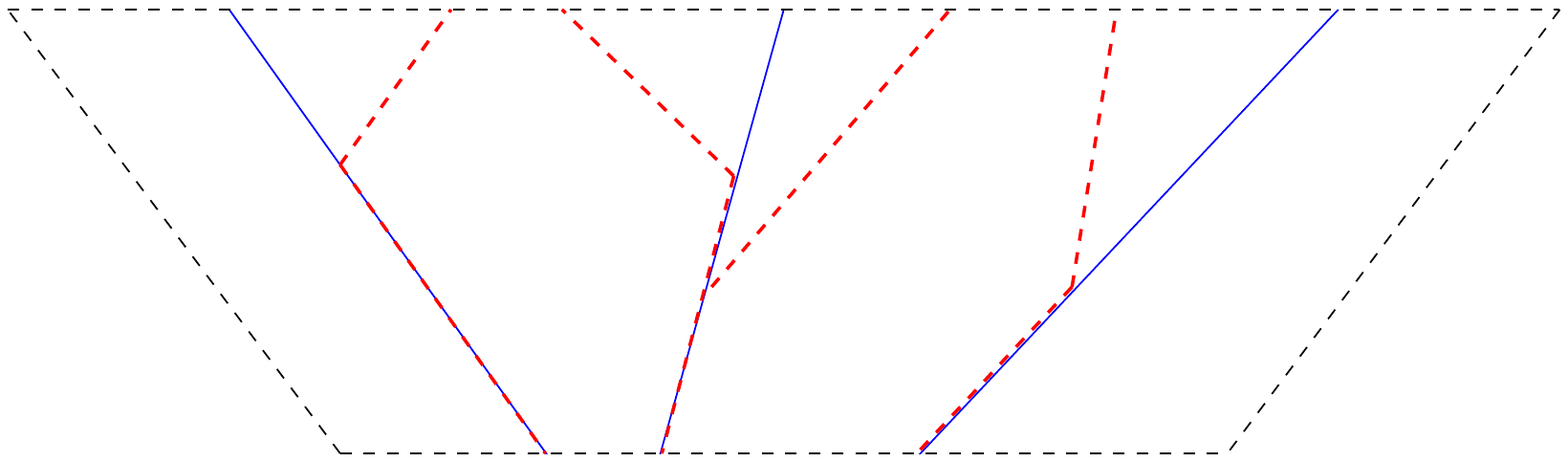}
  			\caption{The previous picture can be shown to imply that there exist two pairs of disjoint geodesics with $\epsilon$-close endpoints in a trapezoid with a lower and upper bases of order $\epsilon$ and $\log(\epsilon^{-1})$ respectively.}
  			\label{fig:6}
  		\end{subfigure}\\
  		\begin{subfigure}[c]{0.5\textwidth}
  			\includegraphics[width=8cm]{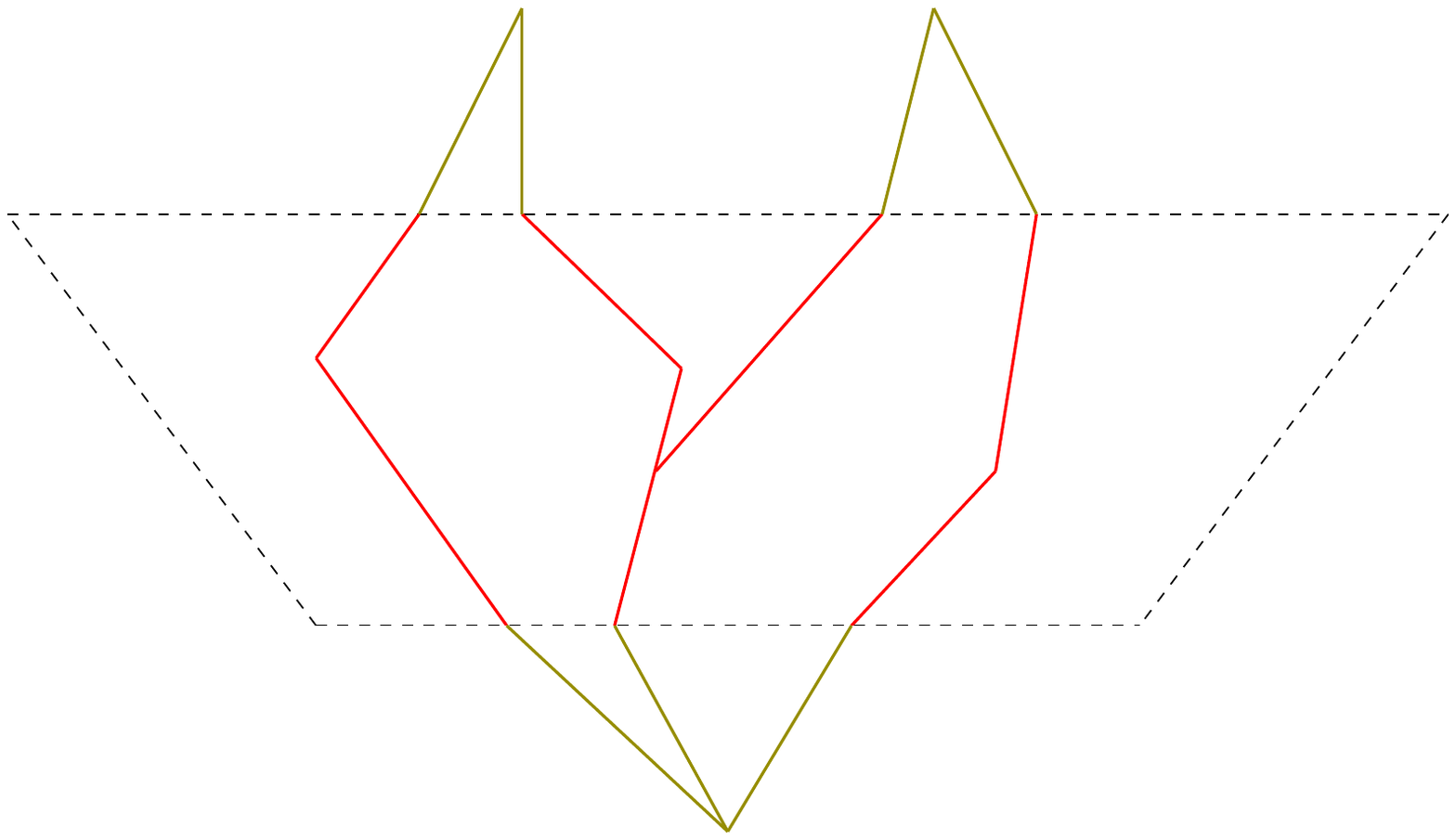}
  			\caption{From the previous picture one can construct two pairs of disjoint paths with shared endpoints.}
  			\label{fig:7}
  		\end{subfigure}&
  		\begin{subfigure}[c]{0.5\textwidth}
  			\includegraphics[width=8cm]{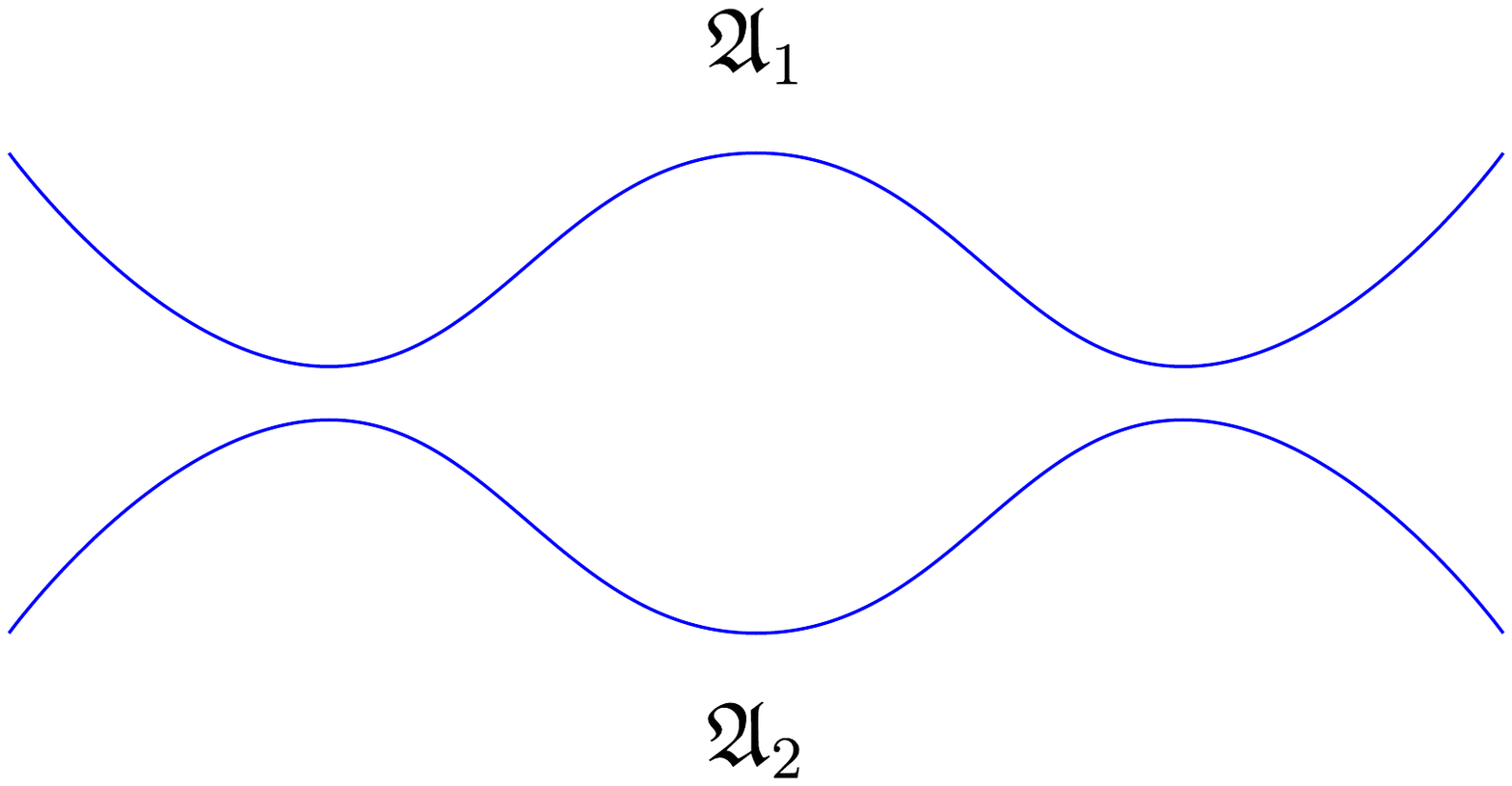}
  			\caption{One can  show that the probability of the event in the previous picture is bounded by the probability that the first two lines of the Airy line ensemble are close to each other at two points. }
  			\label{fig:8}
  		\end{subfigure}\\
  	\end{tabular}    
  	\caption{The main steps behind the proof of Theorem \ref{thm:5}\label{fig:ceo}}
  \end{figure}

  \section{The extended landscape}\label{sec:DL}
  The extended landscape was introduced in \cite{Dauvergne-Zhang-2021}. The authors show that disjoint non-intersecting paths in the Brownian LPP converge to the length of disjoint non-intersecting paths in the directed landscape. Moreover, and very importantly, they show a connection between the parabolic Airy line ensemble and the extended landscape in the flavour of  Greene's Theorem. 
 
  	Let $\Rufk$ be the space of all points $(\bm{x},s;\bm{y},t)$, where $s<t\in\R$ and $\bm{x},\bm{y}$ lie in the same space $\R^k_\leq=\{\bm{x}\in\R^k:x_1\leq \cdots\leq x_k\}$ for some $k\in\N$. For $\bm{u}=(\bm{x},s;\bm{y},t)\in\Rufk$, define
  	\begin{equation}\label{eq8}
  		\dl(\bm{u})=\sup_{\pi_1,\ldots\pi_k}\sum_{i=1}^k\dl(\pi_i)
  	\end{equation}
  	where $k$ is such that $\bm{x},\bm{y}\in \R^k_\leq$. The supremum is over all $k$-tuples of paths $\pi=(\pi_1,\ldots,\pi_k)$ where each $\pi_i$ is a path from $(x_i,s)$ to $(y_i,t)$, and the paths $\pi_i$ are disjoint. Such a collection $\pi$ is called a disjoint $k$-tuple from $(\bm{x},s)$ to $(\bm{y},t)$. This extension of the DL to  $\Rufk\rightarrow \R\cup\{\infty\}$ is called the the extended landscape. 
 
  \begin{theorem}[{\cite[Theorem 1.7]{Dauvergne-Zhang-2021}}]\label{thm:do}
  	Almost surely, the supremum in \eqref{eq8} is attained for  every $\bm{u}=(\bm{x},s;\bm{y},t)\in \Rufk$ by some disjoint k-tuple $\pi$, called a \textbf{disjoint optimizer} for $\bm{u}$ in $\dl$. Moreover, for any fixed $\bm{u}\in \Rufk$, almost surely there is a unique disjoint optimizer $\pi_{\bm{u}}$ for $\bm{u}$ in $\dl$. 
  \end{theorem}
  The Airy line ensemble $\{\hat{\ele}_i: \R\rightarrow \R,i\in\N\}$ was introduced in \cite{Prahofer-Spohn-02} as the scaling limit of the PNG model started from the droplet initial condition. Since then, it has been shown to be the universal  scaling limit of many other models. The parabolic Airy line ensemble $\{\ele_i\}_{i\in\N}$ is defined through $\ele_i(x)=\hat{\ele}_i(x)-x^2$. The following result is a limiting version of the celebrated Greene's theorem and is fundamental to our proof. It connects the extended DL with the parabolic Airy line ensemble.  
  \begin{theorem}[{\cite[Corollary 1.9]{Dauvergne-Zhang-2021}}]\label{thm:DZ}
  	Let $\{\ele_i\}_{i\in\N}$ be the parabolic Airy line ensemble. Then there exists a coupling between $\ele$ and $\dl$ such that  
  	\begin{equation}\label{eq42}
  		\sum_{i=1}^k \ele_i(y)=\dl(0^k,0;y^k,1) \qquad \forall y\in \R,
  	\end{equation}
  	where $y^k=(y,\ldots,y)$ is the vector of length $k$ whose all elements equal $y$.
  \end{theorem}
   The extended landscape can be decomposed into a deterministic part plus a stationary stochastic process $\ash$ i.e.\
  \begin{equation}\label{eq108}
  	\ash(\bm{x},s;\bm{y},t)=\dl(\bm{x},s;\bm{y},t)+\sum_{i=1}^k\frac{(x_i-y_i)^2}{t-s}.
  \end{equation}
  Explicitly, the stationarity of $\ash$ means that for any $c\in\R$ and $\bm{c}=(c,\ldots,c)$ it holds that  $\ash((\bm{x},s;\bm{y},t))\sim \ash((\bm{x},s;\bm{y}+\bm{c},t))$. We shall need the following bounds involving $\ash$. 
  \begin{lemma}[{\cite[Lemma 6.7]{Dauvergne-Zhang-2021}}]\label{lm:ub}
  	For any $k\in\N$, there is a random constant $R>1$, such that for any $\bm{x},\bm{y}\in\R^k_\leq$ and $s<t$, we have
  	\begin{equation}\label{eq40}
  		|\ash(\bm{x},s;\bm{y},t)|<RG(\bm{x},\bm{y},s,t)(t-s)^{1/3}
  	\end{equation}
  	where
  	\begin{equation}\label{eq124}
  		G(\bm{x},\bm{y},s,t)=\Big(1+\frac{\|\bm{x}\|_1+|\bm{y}\|_1}{(t-s)^{2/3}}\Big)\Big(1+\frac{|s|}{t-s}\Big)(1+|\log(t-s)|).
  	\end{equation}
  	Also $\P(R>a)<ce^{-da}$ for any $a>0$. Here $c,d>0$ are constants depending on $k$.
  \end{lemma}
  \begin{lemma}[{\cite[Proposition 10.5]{DOV22}}]\label{lm:ub4}
  	Fix $b>0$. For any $|x'|,|x|,|y'|,|y|\leq b$ the following holds
  	\begin{equation}\label{eq41}
  		\big| \ash(x',0;y',1)-\ash(x,0;y,1)\big|\leq C\xi^{1/2}\log^{1/2}(4b\xi^{-1}),
  	\end{equation}
  	where $\xi=|x-x'|\vee|y-y'|$ and  $C$ is a random constant satisfying $\P(C>m)\leq cb^{10}e^{-dm^{3/2}}$ for some universal  constants $c,d>0$.
  \end{lemma}
  The extended landscape satisfies the following useful symmetry and scaling properties.
  \begin{lemma}[{\cite[Lemma 6.10]{Dauvergne-Zhang-2021}}]\label{lm:prp}
  	Take $q>0,r,c\in\R$, and let $T_c\bm{x}$ denote the shifted vector $(x_1+c,\ldots,x_k+c)$. We have the following equalities in distribution for $\dl$ as functions in $\Rufk$
  	\begin{enumerate}
  		\item Stationarity: $\dl(\bm{x},s;\bm{y},t)\stackrel{d}{=}\dl(T_c\bm{x},s;T_c\bm{y},t+r)$.
  		\item Flip symmetry: $\dl(\bm{x},s;\bm{y},t)\stackrel{d}{=}\dl(-\bm{x},-s;-\bm{y},-t)$.
  		\item Rescaling: $\dl(\bm{x},s;\bm{y},t)\stackrel{d}{=}q\dl(q^{-2}\bm{x},q^{-3}s;q^{-2}\bm{y},q^{-3}t)$.
  		\item Skew symmetry:
  		\begin{equation}
  			\dl(\bm{x},s;\bm{y},t)+(t-s)^{-1}\Vert\bm{x}-\bm{y}\Vert_2^2\stackrel{d}{=} 	\dl(\bm{x},s;T_c\bm{y},t)+(t-s)^{-1}\Vert\bm{x}-T_c\bm{y}\Vert_2^2.
  		\end{equation}
  	\end{enumerate}
  \end{lemma}
 \subsection{Properties of geodesics in the directed landscape}
 In this section we collect results on geodesics in the DL that we shall use in this work. We begin with a bound on the fluctuations of finite geodesics.
  \begin{lemma}[{\cite[Lemma 3.11]{Ganguly-Zhang-2022}}]\label{lm:12}
 	There is a random number $R$ such that the following is true. First, for any $M>0$ we have $\P(R>M)\leq Ce^{-cM^{9/4}\log(M)^{-4}}$ for some constants $c,C>0$. Second, for any $u=(x,s;y,t)\in \R^4_{\uparrow}$, any geodesic $\pi_u$, and $(s+t)/2\leq r<t$, we have 
 	\begin{equation}
 		\Big|\pi_u(r)-\frac{x(t-r)+y(r-s)}{t-s}\Big|<R(t-r)^{2/3}\log^3\big(1+\lVert u\rVert/(t-r)\big)
 	\end{equation}
 	Similar bounds holds when $s<r<(s+t)/2$ by symmetry.
 \end{lemma}
 Next is the analogous result for infinite geodesics.
 \begin{lemma}[ {\cite[Theorem 3.2]{Rahman-Virag-21}}]\label{lm:9}
 	Given a starting point $p=(\chi,s)\in\R^2$ and direction $d\in\R$, there is almost surely an infinite geodesic $g$ of $\mathcal{L}$ from $p$ with direction $d$. More precisely, there is $a>0$ and a random constant $C$ with $\mathbb{E}(a^{C^3})=c<\infty$ so that 
 	\begin{equation}\label{eq111}
 		|g(s+t)-\chi-dt|\leq C t^{2/3}\Big[1\vee\big(\log\log t\big)^{1/3}\Big] \qquad \text{ for $t\geq 1$}.
 	\end{equation}
 	Moreover, the constant $c$ is independent of the direction $d$.
 \end{lemma}
 \begin{proof}
 	The statement in \cite[Theorem 3.2]{Rahman-Virag-21} is slightly different, however, an inspection of the proof therein \cite[Eq. 3.6]{Rahman-Virag-21} shows that one can rewrite the statement as in \eqref{eq111}. Lastly, to see that the constant $c$ is independent of the direction $d$, we note that from the shear invariance of the DL, a bound analogous to   \cite[Eq. 3.6]{Rahman-Virag-21} holds for general direction $d$ i.e.\
 	\begin{equation}
 		\P\Big(\frac{\sup_{s\in[0,t]}|g_n(s)-ds|}{t^{2/3}}>\lambda\Big)<ce^{-a\lambda^3},
 	\end{equation}
 	where $c$ and $a$ are independent of $d$. The rest of the proof continues as in \cite{Rahman-Virag-21}
 \end{proof}
 A priori, two infinite geodesics emanating from the same point may form a `bubble' by diverging at some point in time and merging back again later. The next result states that bubbles can only appear in geodesics close to their endpoints.
 \begin{lemma}[{\cite[Theorem 1]{bhatia2023duality}}]\label{lm:MB}
 	There exists an event $\Omega_{\rm{NoBubble}}$ of full probability, on which  there exist no point $u=(x,s;y,t)\in\R_{\uparrow}^4$ such that there are two distinct geodesics $\eta^1,\eta^2$ from $(x,s)$ to $(y,t)$ with the property that for some small enough $\delta>0$, $\eta^1(r)=\eta^2(r)$ for all $r\in(s,s+\delta)\cup(t-\delta,t)$. As a consequence, almost surely, for any geodesic $\gamma:[a,b]\rightarrow \R$ and any $(a_1,b_1)\subseteq [a,b]$, $\gamma|_{(a_1,b_1)}$ is the unique geodesic between its endpoints.
 \end{lemma}
 Recall that for a geodesic $\pi:[s,t]\rightarrow\R$, we let $[\pi]=[\pi(r)]_{r\in[s,t]}$ denote the graph of $\pi$ in $\R^2$. The following result says that geodesics with close endpoints tend to overlap. The result originally appeared in \cite[Lemma 3.3]{Dauvergne-Sarkar-Virag-2020}, the following version is from \cite[Lemma B.12]{BSS22}.
 \begin{lemma}[{\cite[Lemma 3.3]{Dauvergne-Sarkar-Virag-2020}}]\label{lm:14}
 	The following holds on a single event of full probability. Let $(p_n;q_n)=(x_n,s_n,;y_n,u_n)\in\R_\uparrow\rightarrow(p;q)=(x,s;y,u)\in\R_\uparrow$, and let $g_n$ be any sequence of geodesics from $p_n$ to $q_n$. Suppose that either
 		\begin{enumerate}
 			\item For all $n$, $g_n$ is the unique geodesic from $(x_n,s_n)$ to $(y_n,u_n)$ and $[g_n]\rightarrow [g]$ for some geodesic $g$ from $p$ to $q$, or
 			\item There is a unique geodesic $g$ from $p$ to $q$.
 		\end{enumerate}
 		Then, the \textbf{overlap} 
 		\begin{equation}
 			O(g_n,g):=\{t\in[s_n,u_n]\cap [s,u]:g_n(t)=g(t)\}
 		\end{equation}
 		is an interval for all $n$ whose endpoints converge to $s$ and $u$.
 \end{lemma}

\section{Reducing the problem to a bound on the existence of disjoint geodesics in a trapezoid}\label{sec:fgeo}
 In this section we reduce the proof of Theorem \ref{thm:4} to a statement about geodesics with finite time span. Throughout this section we fix arbitrary $\cdir>1$ show that there are no three non-coalescing infinite geodesics going in a direction $\xi\in [-\cdir,\cdir]$. 
Define  the box 
\begin{equation}\label{eq140}
	\smbox:=\{u=(u_1,u_2):0\leq u_2\leq 1,-1/2\leq u_1\leq 1/2\}.
\end{equation}
For $M>1$, consider a small box of dimension $ 6\cdir \Mbox \times\Mbox$, with its lower edge on the real line (see Figure \ref{fig:1}). For $\Mi>2M$, we divide the set $[-\cdir \Mi,\cdir \Mi]\times\{\Mi\}$ into intervals of size $l>0$ i.e.\ define 
\begin{equation}
	I_i:=[-\frac l2 +il,\frac l2+il)\times \{\Mi\}, \qquad \forall i\in\Ind
\end{equation}
where
\begin{equation}\label{eq89}
	\Ind:=\Big\llbracket -\frac{\cdir \Mi}l-2,\frac{\cdir \Mi}l+2\Big\rrbracket 
\end{equation}
The intervals $\{I_i\}$ are disjoint and 
\begin{equation}
	\Big[-\cdir \Mi-\tfrac52l,\cdir \Mi+\tfrac52l\Big]\times\{\Mi\}\subseteq\bigsqcup _{i\in\Ind}I_i.
\end{equation}
Let $\hat{p}^i$ be the center of the interval $I_i$ i.e.\ 
\begin{equation}\label{eq147}
	\hat{p}^i=(il,\Mi), \qquad \forall\,\, i\in\Ind.
\end{equation}
We also denote by $y^i$ the first component of the left endpoint of the interval $I_i$ i.e.\
\begin{equation}
	y^i:=-\frac l2 +il, \qquad \forall\,\, i\in\Ind. 
\end{equation}
This way one has that $(y^i,\Mi)$ and $(y^{i+1},\Mi)$ are the left and right points of the line segment $I_i\subseteq\R^2$.   For a point $p\in\R^2$ and a given $\xi\in\R$ we denote by $\pi^\xi_p$ the almost sure unique geodesic directed at $\xi$. For $\xi\in \R$, we define 
 \begin{equation}
 \begin{aligned}
 		\pi^\xi_L&:=\pi^\xi_{(-\cdir \Mbox,0)}\\
 		\pi^\xi_R&:=\pi^\xi_{(\cdir \Mbox,0)}.
 \end{aligned}
 \end{equation}
 We define the function 
 \begin{equation}
 	\xi:\R^2 \rightarrow \R,
 \end{equation}
that takes points and sends them  to the direction associated with the their position with  respect to the origin, i.e.\
\begin{equation}
	\xi(p)=\frac{p_1}{p_2}.
\end{equation}
Set 
\begin{equation}
	\xi^i=\xi(\hat{p}^i), \qquad  \forall i\in \Ind,
\end{equation}
  so that we associate a direction with each interval $I_i$. Denote $i_L=\inf\Ind$ and $i_R=\sup\Ind$. Define the following events
  \begin{equation}\label{eq90}
  	\begin{aligned}
  		A^{\Mi,\Mbox,l}_1&:=\Big\{ y^{i-1}< \pi^{\xi^{i-1}}_L(\Mi),\pi_R^{\xi^{i+1}}(\Mi)<y^{i+2}, \quad \forall i\in\llbracket i_L+1,i_R-1 \rrbracket\Big\}\\
  		A^{\Mi,\Mbox,l}_2&:=\Big\{\sup_{0\leq s\leq 1}\pi_L^{\xi^{i-1}}(s)<-1/2,\inf_{0\leq s\leq 1}\pi_R^{\xi^{i+1}}(s)>1/2, \quad \forall i\in\llbracket i_L+1,i_R-1 \rrbracket\Big\}\\
  		A^{\Mi,\Mbox,l}_3&:=\Big\{\pi_L^{\xi^{i-1}}(M)>-3\cdir\Mbox,\pi_R^{\xi^{i+1}}(M)<3\cdir\Mbox, \quad\forall i\in\llbracket i_L+1,i_R-1 \rrbracket\Big\}\\
  		A^{\Mi,\Mbox,l}&:=A^{\Mi,\Mbox,l}_1\cap A^{\Mi,\Mbox,l}_2\cap A^{\Mi,\Mbox,l}_3.
  	\end{aligned}
  \end{equation} 
A typical  realization of the event $A^{\Mi,\Mbox,l}$ is illustrated in Figure \ref{fig:1}. The event $A^{\Mi,\Mbox,l}$  allows us to control infinite geodesics on a finite time horizon. The following result tells us that it occurs with high probability.  
\begin{lemma}\label{lm:10}
	For  $\Mi>2\Mbox>32$,  there exist constants $c,C>0$, independent of  $\cdir,\Mbox,l$ and $\Mi$ such that 
	\begin{equation}
		\P(A^{\Mi,\Mbox,l})>1-C\exp\Big[{-c\Big(\frac{l}{\Mi^{2/3}[\log\log \Mi]^{1/3}}\Big)^3}\Big]\Big(\frac{2\cdir \Mi}l\Big)-Ce^{-c\cdir^3\Mbox[\log\log M]^{-1}}.
	\end{equation} 
\end{lemma}
\begin{proof}
	 Our hypothesis on $\Mi$ and $\Mbox$ ensures that $\log\log \Mi> \log\log \Mbox>1$. From Lemma \ref{lm:9} with $\chi=(\pm\cdir\Mbox,0),s=0$ and $t=\Mi$, there exist constants $c,C>0$, independent of $\cdir,\Mbox,l$ and $\Mi$ s.t.\
	\begin{equation}
	\begin{aligned}
			&\P(\pi_L^{\xi^{i-1}}(\Mi)<y^{i-1} \text{ or } \pi_R^{\xi^{i+1}}(\Mi)>y^{i+2}) \\
			&\leq\P(|\pi_L^{\xi^{i-1}}(\Mi)-\hat{p}^{i-1}_1|>l/2)+\P(|\pi_R^{\xi^{i+1}}(\Mi)-\hat{p}^{i+1}_1|>l/2)\\
		&\leq C\exp\Big[{-c\Big(\frac{l}{\Mi^{2/3}[\log\log \Mi]^{1/3}}\Big)^3}\Big], \qquad \forall i\in \Ind.
	\end{aligned}
	\end{equation}
A union bound on  $i\in\Ind$  gives
\begin{equation}\label{eq1}
	\P\Big(A^{\Mi,\Mbox,l}_1\Big)  > 1-C\exp\Big[{-c\Big(\frac{l}{\Mi^{2/3}[\log\log \Mi]^{1/3}}\Big)^3}\Big]\Big(\frac{2\cdir \Mi}l+5\Big),
\end{equation}
where we used \eqref{eq89} to conclude that $|\Ind|\leq \Big(\frac{2\cdir \Mi}l+5\Big)$. Next note that by order of geodesics
\begin{equation}\label{eq49}
	A^{\Mi,\Mbox,l}_2\supseteq \Big\{\sup_{0\leq s\leq 1}\pi_L^{\xi^{i_R}}(s)<-1/2,\inf_{0\leq s\leq 1}\pi_R^{\xi^{i_L}}(s)>1/2\Big\}.
\end{equation}
From Lemma \ref{lm:9} with $\chi=(\pm\cdir\Mbox,0),s=0$ and $t=1$
\begin{equation}\label{eq50}
	\P(	A^{\Mi,\Mbox,l}_2)\geq \P\Big(\sup_{0\leq s\leq 1}\pi_L^{\xi^{i_R}}(s)< -1/2 \,\text{ or }\inf_{0\leq s\leq 1}\pi_R^{\xi^{i_L}}(s)> 1/2\Big)>1- Ce^{-c(\cdir\Mbox)^3}.
\end{equation}
Next, by order of geodesics
\begin{equation}
	A^{\Mi,\Mbox,l}_3\supseteq \Big\{\pi_L^{\xi^{i_L}}(M)>-3\cdir \Mbox, \pi_R^{\xi^{i_R}}(M)<3\cdir \Mbox\Big\}.
\end{equation}
Again, using Lemma \ref{lm:9}  with $\chi=(\pm\cdir\Mbox,0),s=0$ and $t=\Mbox$ we can bound the probability of the event on the right hand side of the last display to obtain 
\begin{equation}\label{eq51}
	\P\big(A^{\Mi,\Mbox,l}_3\big)>1- Ce^{-c(\cdir\Mbox^{1/3}[\log\log M]^{-1/3})^3}.
\end{equation}
Using  \eqref{eq1}, \eqref{eq50} and \eqref{eq51} we obtain the result.
\end{proof}
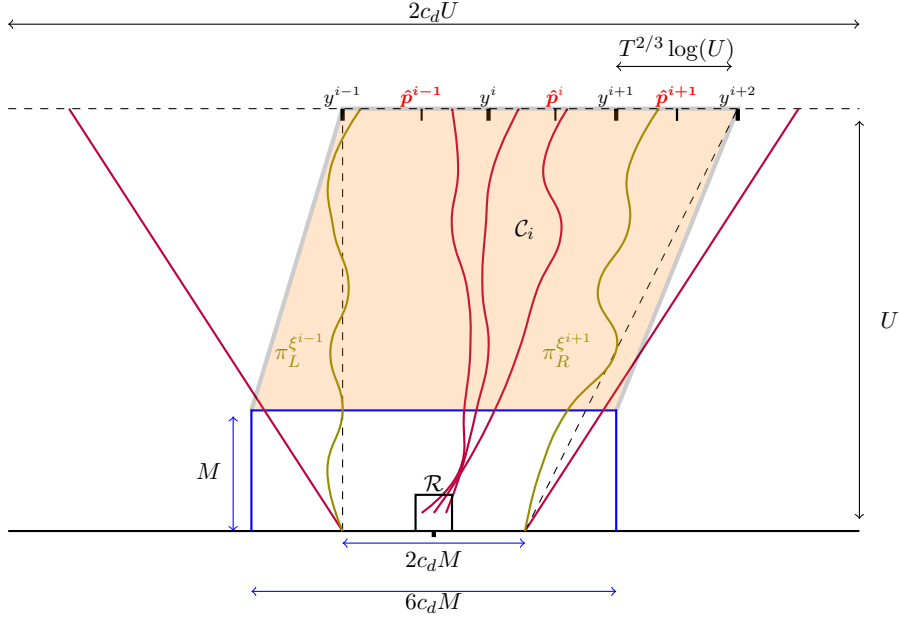
\begin{figure}[]
	\centering
	\begin{tikzpicture}[scale=0.8, every node/.style={scale=0.8}]
		\draw[ thick] (-7,0) -- (7,0);
		\draw[blue,thick] (-3,0) -- (-3,2)--(3,2)--(3,0);
		\draw[dashed] (-7,7) -- (7,7);
		\draw[<->] (-7,8.4) -- (7,8.4);
		\draw[<->] (7,0.2) -- (7,6.8);
		\draw[thick,purple] (1.5,0) -- (6,7);
		\draw[thick,purple] (-1.5,0) -- (-6,7);
		\draw[dashed] (1.5,0) -- (5,7);
		\draw[dashed] (-1.5,0) -- (-1.5,7);
		
		%
		%
		\draw[blue, <->] (-3,-1) -- (3,-1);
		\draw[blue, <->] (-3.3,0) -- (-3.3,1.9);
		\node [scale=1][] at (7.5,3.5) {$\Mi$};
		\node [scale=1][] at (0,-1.2) {$6\cdir \Mbox$};
		
		\draw[blue, <->] (-1.5,-0.2) -- (1.5,-0.2);
		\node [scale=1][] at (0,-0.5) {$2\cdir \Mbox$};

		\node [scale=1][] at (-3.7,1) {$ \Mbox$};
		
		\node [scale=1,thick] at (0,8.6) {$2\cdir \Mi$};
		
		\draw[thick] (2,7) -- (2,6.8);
		\node [scale=1,thick] at (2,7.2) {\footnotesize \color{red}$\bm{\hat{p}}^i$};
		
		\draw[<->] (3,7.7) -- (4.9,7.7);
		\node [scale=1,thick] at (4,8){$T^{2/3}\log (\Mi) $};
		
		\draw [purple, thick] plot [smooth, tension=0.8] coordinates { (-0.2,0.3) (0.42,1) (0.5,2) (0.6,3) (0.55,4) (0.3,5) (0.4,6) (0.3,7)}; 
		\draw [purple, thick] plot [smooth, tension=0.8] coordinates { (0.2,0.3) (0.45,1) (0.7,2) (0.9,3) (0.8,4) (0.85,5) (0.97,6) (1.4,7)};
		\draw [purple, thick] plot [smooth, tension=0.8] coordinates { (0,0.3) (0.47,1) (1,2) (1.4,3) (1.7,4) (2.1,5) (1.8,6) (2.2,7)};
		
		\draw [olive, thick] plot [smooth, tension=0.8] coordinates { (-1.5,0) (-1.75,1) (-1.5,2) (-1.7,3) (-1.4,4) (-1.65,5) (-1.7,6) (-1.2,7)}; 
		\draw [olive, thick] plot [smooth, tension=0.8] coordinates { (1.5,0) (1.75,1) (2.2,2) (3,3) (2.7,4) (3.2,5) (3.2,6) (3.7,7)}; 
		\node [scale=1,thick] at (0.9,7.2) {\footnotesize$y^i$};
		\draw[ultra thick] (0.9,7) -- (0.9,6.8);
		
		\node [scale=1,thick] at (3,7.2)
		{\footnotesize$y^{i+1}$};
		\draw[ultra thick] (3,7) -- (3,6.8);
		
		\node [scale=1,thick] at (5,7.2)
		{\footnotesize$y^{i+2}$};
		\draw[ultra thick] (5,7) -- (5,6.8);
		
		\node [scale=1,thick] at (-1.5,7.2)
		{\footnotesize$y^{i-1}$};
		\draw[ultra thick] (-1.5,7) -- (-1.5,6.8);
		
		\node [scale=1,thick] at (-0.2,7.2)
		{\footnotesize\color{red}$\bm{\hat{p}^{i-1}}$};
		\draw[thick] (-0.2,7) -- (-0.2,6.8);
		
		\node [scale=1,thick] at (4,7.2)
		{\footnotesize\color{red}$\bm{\hat{p}^{i+1}}$};
		\draw[thick] (4,7) -- (4,6.8);		
		
		\draw [ultra thick, draw=black, fill=orange, opacity=0.2]
		(3,2) -- (5,7)--(-1.5,7)--(-3,2); -- cycle;
		\node [scale=1,thick] at (1.5,5) {$\cone_i$};
		\draw[ultra thick]	(0,0) -- (0,-0.1);
		\draw[thick]	(-0.3,0) -- (-0.3,0.6)--(0.3,0.6)--(0.3,0);
		\node [scale=1,thick] at (0,0.8) {$\smbox$};
		
		\node [scale=1,thick] at (-2.2,3)
		{\color{olive}$\pi_{L}^{\xi^{i-1}}$};
		
		\node [scale=1,thick] at (2.2,3)
		{\color{olive}$\pi_{R}^{\xi^{i+1}}$};
		\draw[thick] (4,7) -- (4,6.8);
			
	\end{tikzpicture}
	\caption{An illustration of a typical realization of the event $A^{\Mi,\Mbox,l}$. By order of geodesics, the infinite geodesics $\pi_{L}^{\xi^{i-1}}$ and $\pi_{R}^{\xi^{i+1}}$ sandwich all infinite geodesics starting from $\smbox$ and going in direction $\xi^i$. Our main result will show that contrary to the illustration, with high probability, there are no more than two distinct geodesics going in direction $\xi^i$ and that are disjoint between time $M$ and time $\Mi$.}
	\label{fig:1}
\end{figure}
 For $p,q\in\R^2$ such that $p_2<q_2$ and $a,b>0$ we define
  \begin{equation}
 	\cone=\trpz(p,a;q,b):=\cone^{up}\cup \cone^{dn},
 \end{equation}  
where
 \begin{equation}\label{updn}
 	\begin{aligned}
 		\cone^{dn}:=&\{p+(x,0):x\in(-a/2,a/2)\}\\
 		\cone^{up}:=&\{q+(x,0):x\in(-b/2,b/2)\}.
 	\end{aligned}
 \end{equation}
Somewhat abusing the term, we refer to  $\cone$ as a trapezoid with upper and lower bases $\cone^{up}$ and $\cone^{dn}$ respectively. If $\cone=\trpz(p,a;q,b)$, we define $\cone^{dn}$ and $\cone^{up}$ as in \eqref{updn}.
 For a finite set of geodesics $\{\pi_{p^1}^1,...,\pi_{p^k}^k\}$ emanating from the points  $\bm{p}=(p^1,\ldots,p^k)\in(\R^2)^k$, we define the branching time $\brnc{\pi_{p^1}^1,...,\pi_{p^k}^k}$ to be the time from which onward the geodesics do not meet i.e.\ if $s_{\bm{p}}=\max\{p^i_2:i\in[k]\}$
 \begin{equation}\label{eq24}
 	\brnc{\pi_{p^1}^1,...,\pi_{p^k}^k}=\sup\{t>s_{\bm{p}}:\text{the sets $[ \pi_{p^i}^i(r)]_{t\leq r \leq \infty}$ for $i\in[k]$ are not mutually disjoint}\}.
 \end{equation}
   In case the set on the RHS of \eqref{eq24} is empty we set $\brnc{\pi_{p^1}^1,...,\pi_{p^k}^k}=s_{\bm{p}}$. We note that the geodesic $\pi^{i}_{p^i}$ has two  annotations of  $i$ as it is possible that $p^i=p^j$ for $i,j\in\{1,2,3\}$.
   Thus, $\brnc{\pi_{p^1}^1,...,\pi_{p^k}^k}$ is the latest point in time where any pair of the  geodesics $\pi_{p^1}^1,...,\pi_{p^k}^k$ intersects. When $\brnc{\pi_p^1,\pi_q^2}<\infty$ ($\brnc{\pi_p^1,\pi_q^2}=\infty$) for some $p,q\in\R^2$, we  say the geodesics $\pi_p^1,\pi_q^2$ do not meet (meet) at infinity.  Define the events
\begin{equation}
	\begin{aligned}
		\Thg^{\cdir}:=&\{\exists\xi\in[-\cdir ,\cdir ], \bm{p}=(p^1,p^2,p^3)\in\smbox^3: \small\text{   there exist three geodesics $\{\pi_{p^1}^{\xi,1},\pi_{p^2}^{\xi,2},\pi_{p^3}^{\xi,3}\}$}\\
		&\text{leaving from $\bm{p}$ in direction $\xi$ s.t\ $\brnc{\pi_{p^1}^{\xi,1},\pi_{p^2}^{\xi,2},\pi_{p^3}^{\xi,3}}<\infty$}\}
			 \\
		\Thg^{\cdir,M}:=&\{\exists\xi\in[-\cdir ,\cdir ], \bm{p}=(p^1,p^2,p^3)\in\smbox^3: \small\text{  \ there exist three geodesics $\{\pi_{p^1}^{\xi,1},\pi_{p^2}^{\xi,2},\pi_{p^3}^{\xi,3}\}$} \\
		&\text{leaving from $\bm{p}$ in direction $\xi$ s.t.\  $\brnc{\pi_{p^1}^{\xi,1},\pi_{p^2}^{\xi,2},\pi_{p^3}^{\xi,3}}<M$}\}.
	\end{aligned}
\end{equation}
 Although the events above depend on $\cdir$, we shall  use $	\Thg$ and $	\Thg^{M}$ to ease the notation.
Next we define 
\begin{equation}
	\uni=\big\{\gamma: \text{$\gamma$ is the unique geodesic between its endpoints}\big\}.
\end{equation}
In words, $\uni$ is a random set that consists of all geodesics in DL that are the unique geodesic between their endpoints. Let $\cone= \trpz\big((x,s),a;(y,t),b\big)$ be a trapezoid. We say a path $\gamma$ is \textit{a geodesic in $\cone$} if $\gamma$ is a geodesic and $\gamma(s)\in\coned$ and $\gamma(t)\in\coneu$. In other words, $\gamma$ is a geodesic in $\cone$ if its endpoints lie in the bases of the trapezoid $\cone$. We define the event where there are $k$ disjoint geodesics going from $\cone^{dn}$ to $\cone^{up}$ i.e. 
\begin{equation}\label{dijs}
		\disjc^k_{\cone}:=\{\text{there exist $k$ disjoint geodesics  $\pi^1,\ldots,\pi^k\in \uni$ in $\cone$}\}.
\end{equation}
 Whenever the event $\disjc^k_{\cone}$ occurs we say that \textit{there exist $k$ disjoint geodesics in the trapezoid $\cone$}. Next we associate a trapezoid with each $\hat{p}^i$ (recall \eqref{eq147})
\begin{equation}\label{eq119}
	\begin{aligned}
		\cone_i:=\trpz\Big((0,\Mbox),6\cdir\Mbox;\hat{p}^i, 3l\Big) \qquad i\in\Ind.
	\end{aligned}
\end{equation}
\begin{lemma}\label{lem:ub}
	For fixed $\cdir,M>16$ and $\Mi>2M$ the following bound holds
	\begin{equation}
		\P\big(\Thg^{M}\cap A^{\Mi,\Mbox,l}\big)\leq \Big(\frac{2\cdir \Mi}l+5\Big)\sup_{i\in\Ind}\P\big(\disjc^{\,3}_{\cone_i}\big).
	\end{equation}
\end{lemma}
\begin{proof}
	Consider the event of full measure $\Omega_{\rm{NoBubble}}$ from Lemma \ref{lm:MB}. For every realization $\omega\in \Thg^{M}\cap A^{\Mi,\Mbox,l}\cap \Omega_{\rm{NoBubble}}$ there exists a direction $\xi_0\in[-\cdir , \cdir ]$ and  points $(p^1,p^2,p^3)\in\smbox^3$ s.t.\ there exist  geodesics $\{\pi_{p^i}^{\xi_0,i}\}_{i\in \{1,2,3\}}$ directed towards $\xi_0$ such that the planar paths
	\begin{equation}
		\{[\pi^{\xi_0,1}_{p^1}](s)\}_{s\geq M}, 	\{[\pi^{\xi_0,2}_{p^2}](s)\}_{s\geq M}, 	\{[\pi^{\xi_0,3}_{p^3}](s)\}_{s\geq M}, \quad \text{  are mutually disjoint}.
	\end{equation}
Let  $i_0\in\inti{i_L+1,i_R-1}$ be such that  $\xi_0\in I_{i_0}$. By order of geodesics
\begin{equation}\label{eq145}
	\pi_L^{\xi^{i_0-1}}(s)\leq \pi_{p^j}^{\xi_{0},j}(s)\leq \pi_R^{\xi^{i_0+1}}(s) \qquad \forall j\in\{1,2,3\},\forall s\geq 0.
\end{equation}
Recalling the event $A^{\Mi,\Mbox,l}$ from \eqref{eq90}, the last display implies that  
\begin{equation}\label{eq53}
	\omega\in\Big\{\pi_{p^j}^{\xi_0,j}(M)\in \coned_{i_0} \text{ and } \pi_{p^j}^{\xi_0,j}(\Mi)\in \coneu_{i_0},  \qquad\forall j\in\{1,2,3\}\Big \}.
\end{equation}
To see why the last statement is true, consider the trapezoid \eqref{eq119}, and note that the event $A^{\Mi,\Mbox,l}$ implies that
\begin{equation}
	\pi_L^{\xi^{i_0-1}}(M),\pi_R^{\xi^{i_0+1}}(M)\in  \coned_{i_0} \quad \text{ and } \quad \pi_L^{\xi^{i_0-1}}(U),\pi_R^{\xi^{i_0+1}}(U)\in  \coneu_{i_0},
\end{equation}   
which combined with \eqref{eq145} implies \eqref{eq53}.  As $\omega\in \Thg^{M}$ it holds that
\begin{equation}\label{eq52}
	\omega\in \{\brnc{\pi_{p^1}^{\xi_0,1},\pi_{p^2}^{\xi_0,2},\pi_{p^3}^{\xi_0,3}}<M\}.
\end{equation}
  We now claim that   $\omega\in \disjc^3_{\cone_{i_0}}$. Indeed, from \eqref{eq52}, \eqref{eq53}  it follows that on $\omega$, the geodesics $\{\pi^{\xi_{0},i}|_{[M,\Mi]}\}_{i\in\{1,2,3\}}$ are disjoint in $\cone_{i_0}$ and since $\omega\in \Omega_{\rm{NoBubble}}$ we conclude that $\{\pi^{\xi_{0},i}|_{[M,\Mi]}\}_{i\in\{1,2,3\}}\subseteq \uni$.  More generally, 
\begin{equation}
	\Thg^{M}\cap A^{\Mi,\Mbox,l}\cap \Omega_{\rm{NoBubble}}\subseteq \bigcup_{i\in\Ind}\disjc^3_{\cone_i}.
\end{equation}
We conclude the result using a union bound on the different trapezoids. 
\end{proof}
\begin{corollary}
	There exists $C_1>0$ such that for
	\begin{equation}
		l= C_1\Mi^{2/3}[\log\log \Mi]^{1/3}[\log(\Mi)]^{1/3},
	\end{equation}
it holds that 
	\begin{equation}\label{eq2}
		\P\big(\Thg\big)\leq \limsup_{M\rightarrow \infty}\limsup_{\Mi\rightarrow \infty}\Big(\frac{2\cdir \Mi}l+5\Big)\P\big(\disjc^3_{\cone_0}\big).
	\end{equation}
\end{corollary}
\begin{proof}
	The following convergence holds a.s.\
	\begin{equation}\label{eq91}
		1_{\Thg^{M}}\uparrow  1_\Thg, \quad \text{ as $M\rightarrow \infty$.}
	\end{equation}
 Lemma \ref{lm:10} and  our hypothesis on $l$ dictate that 
\begin{equation}\label{eq92}
	\P\big(A^{\Mi,\Mbox,l}\big)\rightarrow 1,
\end{equation}
as first $\Mi \rightarrow\infty$ and then $\Mbox\rightarrow \infty$. Using the  Lebesgue convergence theorem in \eqref{eq91}, as well as  \eqref{eq92}  implies that 
	\begin{equation}
		 \lim_{\Mbox\rightarrow \infty}\lim_{\Mi\rightarrow \infty}\P\Big(\Thg^{M}\cap A^{\Mi,\Mbox,l}\Big)=\lim_{\Mbox\rightarrow \infty}\lim_{\Mi\rightarrow \infty}\P\Big(\Thg^{M}\Big)= \P\big(\Thg\big).
	\end{equation}
From Lemma \ref{lem:ub} and the last display
	\begin{equation}
	\P\big(\Thg\big)\leq \limsup_{M\rightarrow \infty}\limsup_{\Mi\rightarrow \infty}\Big(\frac{2\cdir \Mi}l+5\Big)\sup_{i\in\Ind}\P\big(\disjc^3_{\cone_i}\big).
\end{equation}
Finally, from the skew-stationarity of the DL (Lemma \ref{lm:prp}, we provide a proof for the display below in Lemma \ref{lm:8})
\begin{equation}\label{eq138}
	\P\big(\disjc^3_{\cone_i}\big)=\P\big(\disjc^3_{\cone_j}\big), \qquad \text{for all $i,j\in\Ind$}.
\end{equation}
   This concludes the result.
\end{proof}

	Equation \eqref{eq2} suggests that in order to prove our main result we should bound $\P\big(\disjc^3_{\cone_0}\big)$.
However, it would be more convenient to bound an event that is measurable only with respect to the data coming from the DL on a \textbf{fixed} finite time span. We therefore use translation invariance and scaling time by $\Mi(1-M/\Mi)$ (Lemma \ref{lm:prp}) to obtain
	\begin{equation}\label{eq23}
	\P\big(\disjc^3_{\cone'}\big)=\P\big(\disjc^3_{\cone_0}\big),
\end{equation}
where
\begin{equation}
		\cone':=\trpz\Big((0,0),6\cdir\Mbox \Mi^{-2/3}(1-M/\Mi)^{-2/3};(0,1), 3l\Mi^{-2/3}(1-M/\Mi)^{-2/3} \Big).
\end{equation}
Note that   since $\Mi>2M$, $(1-M/U)^{-2/3}<2$ and so
\begin{equation}\label{eq137}
	\cone'\subseteq\cone_{\text{scl}}:=\trpz\Big((0,0),12\cdir\Mbox \Mi^{-2/3};(0,1), 6l\Mi^{-2/3} \Big)=\trpz\Big((0,0),\epsilon;(0,1), L(\epsilon^{-1}) \Big),
\end{equation}
where we used the change of variables
\begin{equation}\label{eq120}
	\begin{aligned}
		\epsilon&:=12\cdir\Mbox \Mi^{-2/3}\\
		L(\epsilon^{-1})&:=6l\Mi^{-2/3}=6l(12\cdir\Mbox)^{-1}\epsilon.
	\end{aligned}
\end{equation}
Note that although $\epsilon$ depends on $M$ and $c_d$ we omit this dependence in the notation. From \eqref{eq137}
\begin{equation}
	\P\Big(\disjc^3_{\cone'_0}\Big)\leq \P\Big(\disjc^3_{\cone_{\text{scl}}}\Big).
\end{equation}
 Of what order should the bound on  $\P\big(\disjc^3_{\cone_{\text{scl}}}\big)$ be?  From \eqref{eq120}
 \begin{equation}
 	\frac{U}l=6U^{1/3}[L(\epsilon^{-1})]^{-1}=6\big(12\cdir\Mbox\big)^{1/2}\epsilon^{-1/2}[L(\epsilon^{-1})]^{-1},
 \end{equation}
  which, along with inequality \eqref{eq2}, suggests that we should aim for the following bound 
 \begin{equation}
 		\P\big(\disjc^3_{\cone_{\text{scl}}}\big)=o(\epsilon^{1/2}L(\epsilon^{-1})).
 \end{equation}
In other words, we have shown the following. 
 \begin{corollary}\label{cor:3}
 	Under the change of variables in \eqref{eq120}, for every $M>16$, there exists $C_2(M),\epsilon_0>0$ such that for 
 	\begin{equation}
 		L(\epsilon^{-1}):= C_2[\log\log(\epsilon^{-1})]^{1/3}[\log(\epsilon^{-1})]^{1/3}, \quad 0<\epsilon<\epsilon_0,
 	\end{equation}
 	it holds that 
 	\begin{equation}\label{eq121}
 		\P\big(\Thg\big)\leq \limsup_{M\rightarrow \infty}\limsup_{\epsilon\rightarrow 0}\Big(6\big(12\cdir\Mbox\big)^{1/2}\epsilon^{-1/2}[L(\epsilon^{-1})]^{-1}+5\Big)\P\big(\disjc^3_{\textrm{scl}}\big).
 	\end{equation}
 \end{corollary}
 In light of Corollary \ref{cor:3} the following result goes a long way to prove Theorem \ref{thm:2}
 \begin{proposition}\label{prop:2}
 	For every $M>16$, there exists $C_2(M),\epsilon_0>0$ such that for  
 	\begin{equation}
 		L(\epsilon^{-1}):= C_2[\log\log(\epsilon^{-1})]^{1/3}[\log(\epsilon^{-1})]^{1/3}, \quad 0<\epsilon<\epsilon_0,
 	\end{equation}
 	the following bound holds 
 	 \begin{equation}\label{eq93}
 		\P\big(\disjc^3_{\cone_{\text{scl}}}\big)\leq CL(2\epsilon^{-1})\epsilon^{\,8/15}.
 	\end{equation}
 \end{proposition}
 The rest of the paper is devoted to proving Proposition \ref{prop:2}, a task we conclude in Section \ref{sec:ubt}.    
	\begin{proof}[Proof of Theorem \ref{thm:4}]
	From Theorem \ref{thm:3} (i) it is enough to  look at semi-infinite geodesics. Define the event
\begin{equation}
		\begin{aligned}
		\nThg:=&\big \{\forall\xi\in\R, \bm{p}=\{p^1,p^2,p^3\}\subset \R^2: \small\text{$\brnc{\pi_{p^1}^{\xi,1},\pi_{p^2}^{\xi,2},\pi_{p^3}^{\xi,3}}=\infty$ for any three geodesics }\\
		&\text{$\{\pi_{p^1}^{\xi,1},\pi_{p^2}^{\xi,2},\pi_{p^3}^{\xi,3}\}$ leaving from $\bm{p}$ in direction $\xi$}\big \}.
	\end{aligned}
\end{equation}
		On the event $\nThg$, at least two out of any three distinct geodesics starting from  points $p^1,p^2,p^3\in\R^2$ in direction $\xi\in \R$ will coalesce. Thus, the proof of Theorem \ref{thm:4} boils down to showing that 
	\begin{equation}\label{eq118}
		\P(\nThg)=1.
	\end{equation}
	Plugging \eqref{eq93} on the right hand side of \eqref{eq121}
	\begin{equation}\label{eq141}
		\P\big(\Thg^{\cdir}\big)\leq \limsup_{M\rightarrow \infty}\limsup_{\epsilon\rightarrow 0}C\big(\cdir\Mbox\big)^{1/2}[L(2\epsilon^{-1})]\epsilon^{\,1/30}=0.
	\end{equation}
	  Next we define the event
	  	\begin{equation}
	  		\begin{aligned}
	  		\Thg^\infty:=&\{\exists\xi\in\R, \bm{p}=(p^1,p^2,p^3)\in\smbox^3: \small\text{   there exist three geodesics $\{\pi_{p^1}^{\xi,1},\pi_{p^2}^{\xi,2},\pi_{p^3}^{\xi,3}\}$}\\
	  		&\text{leaving from $\bm{p}$ in direction $\xi$ s.t\ $\brnc{\pi_{p^1}^{\xi,1},\pi_{p^2}^{\xi,2},\pi_{p^3}^{\xi,3}}<\infty$}\}.
	  	\end{aligned}
	  	\end{equation}
	  The event $\Thg^\infty$ is similar to $\Thg^{\cdir}$ but with no restriction on the set of  directions. As $\cdir$ can be chosen arbitrarily large we conclude from $\eqref{eq141}$ that 
	  \begin{equation}
	  	\P(\Thg^\infty)=0.
	  \end{equation}
 Next we would like to extend the result for geodesics starting from $\bm{p}\in(\R^2)^3$.  For $\alpha\geq 1$, define the set
	\begin{equation}
		\smbox_\alpha:=\{u=(u_1,u_2):-\alpha/2\leq u_2\leq \alpha/2,-\alpha/2\leq u_1\leq \alpha/2\}.
	\end{equation}
	Thus, the set $\smbox_\alpha\subseteq\R^2$ is simply a translated and scaled version of the set $\smbox$ defined earlier in \eqref{eq140}. Consider the event
	\begin{align}
		\Thg^\infty_\alpha:=&\{\exists\xi\in\R, \bm{p}=(p^1,p^2,p^3)\in\smbox_\alpha^3: \small\text{   there exist three geodesics $\{\pi_{p^1}^{\xi,1},\pi_{p^2}^{\xi,2},\pi_{p^3}^{\xi,3}\}$}\\
		&\text{leaving from $\bm{p}$ in direction $\xi$ s.t\ $\brnc{\pi_{p^1}^{\xi,1},\pi_{p^2}^{\xi,2},\pi_{p^3}^{\xi,3}}<\infty$}\},
	\end{align}
	and note that from the stationarity and scale invariance of the DL it follows that
	\begin{equation}\label{eq117}
		\P(\Thg^\infty_\alpha)=0.
	\end{equation}
		 Clearly,
		\begin{equation}
			\nThg^c\subseteq \bigcup_{\alpha\in\Z_+} \Thg^\infty_\alpha.
		\end{equation}
		Taking probability on both sides of the last display  and using a union bound along with \eqref{eq117} gives \eqref{eq118}. 
	\end{proof}
	\begin{remark}
		The use of stationarity and scale invariance of the DL is not crucial in obtaining \eqref{eq117}. Indeed, one can take $\smbox$ in \eqref{eq140} to be arbitrarily large. 
	\end{remark}
		\begin{proof}[proof of Theorem \ref{thm:2}]
		From Corollary \ref{thm:5} and Theorem \ref{thm:3} we see that the following holds with probability one: for any $p\in\R^2$ and $\xi\in\R$, if $\eta$ is an infinite $\xi$-directed geodesic starting  from $p$, then $\eta$ must coalesce with either $\pi_p^{\xi-,L}$ or $\pi_p^{\xi+,R}$ (or both if $\xi\notin \Xi$). But from the definition of Busemann geodesics $\eta$ must be a Busemann geodesic.
	\end{proof}
\begin{remark}
	In contrast to \eqref{eq93},  the result in \cite{busa-ferr-20} suggests that the probability of two disjoint geodesics in $\cone_{\text{scl}}$ is 
	\begin{equation}
		\P\Big(\disjc^2_{\cone_{\text{scl}}}\Big)\sim C\epsilon^{1/2}L(\epsilon^{-1}).
	\end{equation}
	which shows where the argument would break down if one were to show that there exists no two infinite geodesics that do not meet at infinity using an argument similar to \eqref{eq121}. 
\end{remark}
 \section{From $k$ disjoint geodesics to $k-1$ pairs of disjoint geodesics with close endpoints}\label{sec:trp}
 In what follows we use the following notation. For $(p,q)\in\R^4_\uparrow$ we let $\pi_p^q$ be any geodesic from $p$ to $q$. When there is no room for confusion as to the time horizons we shall often omit them from the geodesic notation e.g.\ for $p=(x,0)$ and $q=(y,1)$ we write $\pi_x^y:=\pi_p^q$.   For $a<b$ reals we define
 \begin{equation}
 	[a,b]^{\mathbb{Q}}=[a,b]\cap \mathbb{Q}.
 \end{equation}
In order not to have to refer to left and right geodesics between endpoints, we introduce the event 
\begin{equation}
	\Omq=\{\text{For any $(x,s,y,t)\in\Q^4\cap\R^4_{\uparrow}$ it holds that $\pi_{(x,s)}^{(y,t)}\in\uni$}\}.
\end{equation}
We remark  that from \eqref{eq122} it holds that  $\P(\Omq)=1$. Let us define a partial order relation on the set of geodesics. Consider two geodesics $\{\pi(r)\}_{r\in[s_1,s_2]}$ and $\{\eta(r)\}_{r\in[t_1,t_2]}$ we write
\begin{equation}
	\pi \leq \eta \iff \pi(r)\leq \eta(r) \quad \forall r\in [s_1\vee t_1,s_2\wedge t_2].
\end{equation}
We say the $k$-pairs of geodesics $\{(\pi^{1,i},\pi^{2,i})\}_{i\in [k]}$ are \textbf{ordered} if 
\begin{equation}
	\begin{aligned}
			&\pi^{j,i}\leq \pi^{k,i+1}, \qquad j,k\in\{1,2\},i\in[k-1]\\
			&\pi^{1,i}\leq \pi^{2,i}, \qquad i\in[k].
	\end{aligned}
\end{equation}
i.e.\ $\pi^{2,i}$ lies to the right of $\pi^{1,i}$ and both $\pi^{1,i+1}$ and $\pi^{2,i+1}$ lie to the right of $\pi^{1,i}$ and $\pi^{2,i}$.
Let $\cone$ be a trapezoid and let $\phi>0$. We define
\begin{equation}\label{eq123}
	\begin{aligned}
		\Mdisjc_{\cone}^{\phi}:=&\sup\{k:\text{There exist  $k$ ordered pairs of disjoint geodesics}\\\text{ $\{(\pi^i,\eta^i)\}_{i\in [k]}\subseteq \uni^2$} &\text{ in $\cone$  such that $|\pi^i(\cone^{up})-\eta^i(\cone^{up})|<\phi$}\},
	\end{aligned}
\end{equation}
where for $\coneu=[a,b]\times\{t\}$ we denote $\pi(\coneu)=\pi(t)$. In words, the random variable $	\Mdisjc_{\cone}^{\phi}$ is the maximal number of ordered, disjoint pairs of geodesics in $\cone$, such that the endpoints of  each pair on the set $\coneu$ is at most $\phi$ from each other. We remark that  the word 'disjoint' here refers to each pair of geodesics and not the relation between the different pairs i.e.\ with the notation in \eqref{eq123} it is always the case that $\pi^i$ and $\eta^i$ are disjoint but it is possible that $\eta^i\cap\pi^{i+1}\neq \emptyset$. Our next result relates the random variable $\Mdisjc_{\cone}^{\phi}$ with the event of two disjoint geodesics in $\cone$.
\begin{lemma}\label{lem:2dj}
	With probability one, for any trapezoid  $\cone$ and $\phi>0$
	\begin{equation}
		\disjc^2_{\cone}=\{\Mdisjc_{\cone}^{\phi}\geq 1\}
	\end{equation}
\end{lemma}
\begin{proof}
	 Clearly, $\disjc^2_{\cone}\supseteq\{\Mdisjc_{\cone}^{\phi}\geq 1\}$, so we only prove the reversed containment. As the proof is geometric in nature, we suggest consulting Figure \ref{fig:2}.  For simplicity, we assume
		\begin{equation}
		\begin{aligned}
			\coneu&=[b'_1,b'_2]\times \{1\}\\
			\coned&=[a'_1,a'_2]\times \{0\},
		\end{aligned}
	\end{equation}
	 It will be clear from the proof below that there is no loss in generality. As the time levels of the $\coned$ and $\coneu$ are fixed to be $0$ and $1$ respectively, to ease the notation of the proof, locally we abbreviate $\pi_{a}^{b}=\pi_{(a,0)}^{(b,1)}$ whenever $(a,0)\in\coned,(b,1)\in\coneu$. Suppose $\disjc^2_{\cone}$ holds, then by definition there exist $a'_1 \leq a_1\leq a_2\leq a'_2$  and $b'_1\leq b_1\leq b_2\leq b'_2$, so that $\pi_{a_1}^{b_1},\pi_{a_2}^{b_2}\in\uni$ are disjoint in $\cone$.  Consider the smaller trapezoid $\hat{\cone}$
	\begin{equation}
		\begin{aligned}
			\hat{\cone}^{up}&=[b_1,b_2]\times \{1\}\\
			\hat{\cone}^{dn}&=[a_1,a_2]\times \{0\}.
		\end{aligned}
	\end{equation}
   Define $u$ as the leftmost point in $[b_1,b_2]\times \{1\}$ such that the geodesics from $(a_2,0)$ terminating at $u$ do not intersect    $\pi_{a_1}^{b_1}$ i.e.\ $u=(b,1)$ where 
 \begin{equation}
 	b:=\inf\big\{y'\in[b_1,b_2]^\mathbb{Q}:\pi_{(a_1,0)}^{(b_1,1)}\cap\pi_{(a_2,0)}^{(y',1)}=\emptyset\big\}.
 \end{equation}
  We first claim that we may assume that 
 \begin{equation}\label{eq139}
 	b_1<b<b_2.
 \end{equation}
 Suppose $b=b_1$, then there exists a sequence of points $c_n\in\Q$ such that $c_n\rightarrow b_1$ and $\pi_{a_2}^{c_n}\cap\pi_{a_1}^{b_1}=\emptyset$. Taking $n_0\in\N$ such that $c_{n_0}-b_1<\phi$ implies the event $\{\Mdisjc_{\cone}^{\phi}\geq 1\}$ which concludes the proof. To complete \eqref{eq139}, we now claim that it must be that $b<b_2$. Assume, by way of contradiction, that $b=b_2$, then there exists a sequence $d_n\in\Q$ such that $d_n\rightarrow b_2$ and $\pi_{a_2}^{d_n}\cap\pi_{a_1}^{b_1}\neq\emptyset$. As $\pi_{a_2}^{b_2}\in\uni$, it follows  from  Lemma \ref{lm:14} that $\pi_{a_2}^{d_n}\rightarrow \pi_{a_2}^{b_2}$ in the overlap sense, which leads to a contradiction. \\
We now pick $u_l\in[b-\delta,b)^{\Q}$  and $u_r\in(b,b+\delta]^{\Q}$  where
\begin{equation}
	\delta:=\tfrac12\min\{\phi,b_2-b,b-b_1\},
\end{equation}
By definition of $u_l$ and $u_r$
\begin{equation}\label{eq6}
	\pi_{a_1}^{b_1}\cap \pi_{a_2}^{u_r}=\emptyset \quad \text{ and } \quad \pi_{a_1}^{b_1}\cap \pi_{a_2}^{u_l}\neq\emptyset.
\end{equation}
This implies that there exists $0\leq s<t<1$ such that (see Figure \ref{fig:2})
\begin{align}					 \pi_{a_1}^{b_1}(r)<&\pi_{a_2}^{u_l}(r)= \pi_{a_2}^{u_r}(r) \quad  0\leq r\leq s \label{eq3}\\
		\pi_{a_1}^{b_1}(r)\leq& \pi_{a_2}^{u_l}(r)<\pi_{a_2}^{u_r}(r) \quad  s<r\leq 1\label{eq4}\\
		\pi_{a_1}^{b_1}(t)&=\pi_{a_2}^{u_l}(t)\label{eq5}
\end{align}
By order of geodesics and the fact that all geodesics in question are  unique it follows that 
	\begin{align}
	\eqref{eq6}+\eqref{eq5} \implies	&\pi_{a_1}^{b_1}(r)=\pi_{a_1}^{u_l}(r)< \pi_{a_2}^{u_r}(r),\quad  0\leq r\leq t\\
		\eqref{eq4}\implies &\pi_{a_1}^{b_1}(r)\leq \pi_{a_1}^{u_l}(r)<  \pi_{a_2}^{u_r}(r) \quad  t\leq r\leq 1.
	\end{align}
which implies $	\pi_{a_1}^{u_l}\cap 	\pi_{a_2}^{u_r}=\emptyset$. By our choice of $u_l$ and $u_r$ it holds that  $u_r-u_l< \phi$. This concludes the result.
\end{proof}
We can now generalize \ref{lem:2dj} to when  the number of pairs is greater than $1$. In this paper we will use the result only for $k=3$, but we provide the proof for general $k\geq 3$ for any need of such  reference in the future.  

\begin{lemma}\label{lem:dis}
	With probability one, for any trapezoid  $\cone$, $\phi>0$ and integer $k\geq 2$
	\begin{equation}
		\disjc^k_{\cone}=\{\Mdisjc_{\cone}^{\phi}\geq k-1\}.
	\end{equation}
\end{lemma}
\begin{proof}
	For $k=2$ the result follows from Lemma \ref{lem:2dj}, so we assume $k\geq 3$.\vspace{5pt}\\
	$\disjc^k_{\cone}\supseteq \{\Mdisjc_{\cone}^{\phi}\geq k-1\}$: Let $\{(\pi^i,\eta^i)\}_{i\in [k-1]}$ be $k-1$ ordered pairs of disjoint geodesics in $\uni$. Then, the $k$ geodesics in the  set   
	\begin{equation}
		\{\pi^{i}\}_{i\in [k-1]}\cup \{\eta^{k-1}\}\subseteq \uni,
	\end{equation}
	  can be verified to be disjoint.\vspace{5pt}\\
	    $\disjc^k_{\cone}\subseteq \{\Mdisjc_{\cone}^{\phi}\geq k-1\}$: Without loss of generality, we assume 
	    \begin{equation}
	    	\begin{aligned}
	    		\cone^{up}&=[b_0,b_{k+1}]\times \{1\}\\
	    		\cone^{dn}&=[a_0,a_{k+1}]\times \{0\}.
	    	\end{aligned}
	    \end{equation}
	Let $a_0<a_1<\ldots <a_k<a_{k+1}$ and $b_0<b_1<\ldots <b_k<b_{k+1}$ be reals such that 
	\begin{equation}
		\{\pi_{a_i}^{b^i}\}_{i\in[k]}\subseteq\uni \text{ are mutually disjoint}.
	\end{equation} 
 
Consider the following $k-1$ trapezoids   
	\begin{equation}
		\cone_i=\cone^{up}_i\cup\cone^{dn}_i\qquad i\in[k-1],
	\end{equation}
where
	\begin{equation}
	\begin{aligned}
		\cone^{up}_i&=[b_i,b_{i+1}]\times \{1\}\\
		\cone^{dn}_i&=[a_i,a_{i+1}]\times \{0\}.
	\end{aligned}
\end{equation}
We now use Lemma \ref{lem:2dj} to find $k-1$ ordered pairs of disjoint geodesics, a pair for each trapezoid $\cone_i$. Precisely, for each $i\in[k-1]$, we use Lemma \ref{lem:2dj} with $\cone_i$ and the geodesics $\pi_{a_i}^{b_i}$ and $\pi_{a_{i+1}}^{b_{i+1}}$ to find two disjoint geodesics $\pi^i$ and $\eta^i$ such that $\pi^i(0),\eta^i(0)\in \cone^{dn}_i$, $\pi^i(1),\eta^i(1)\in \cone^{up}_i$ and $|\pi^i(1)-\eta^i(1)|<\phi$. This concludes the proof.
\end{proof}
\begin{figure}[t]	
	\centering
	\begin{tikzpicture}[scale=0.6, every node/.style={scale=0.8}]
		\draw[thick] (-3,0) -- (3,0);
		\draw[thick] (-7,7) -- (7,7);
		\draw[dashed,black] (3,0) -- (7,7);
		\draw[dashed,black] (-3,0) -- (-7,7);
		\draw[thick,black] (0,7) -- (0,6.8);
		
		\node [scale=1][] at (-3,-0.2) {$a_1$};	
		
		\node [scale=1][] at (3,-0.2) {$a_2$};
		
		\node [scale=1][] at (-7,7.3) {$b_1$};	
		
		\node [scale=1][] at (7,7.3) {$b_2$};
		
		\draw[thick,black] (0,7) -- (0,6.8);
		\node [scale=1][] at (0,7.3) {$u$};	
		
		\draw[thick,black] (-1,7) -- (-1,6.8);
		\node [scale=1][] at (-1,7.3) {$u_l$};
		
		\draw[thick,black] (1,7) -- (1,6.8);
		\node [scale=1][] at (1,7.3) {$u_r$};
		
		\draw[thick,dashed,blue] (-3,0) -- (-5.7,4.7)--(-1,7);
		\draw[thick,dashed,red] (3,0) -- (5,3.5)--(1,7);
		\draw[thick,dashed,olive] (3,0)--(3.5,7/8)--(-4.7,3)--(-5.2,3.8)--(-1,7);
		
		\node [scale=1][] at (-6.5,5) {$\pi_{a_1}^{b_1}$};
		\node [scale=1][] at (6.5,5) {$\pi_{a_2}^{b_2}$};
		
		\node [scale=1,red][] at (5,2.5) {$\pi_{a_2}^{u_r}$};
		
		\node [scale=1,blue][] at (-5,2.5) {$\pi_{a_2}^{u_l}$};
		
		\node [scale=1,olive][] at (1.3,2.5) {$\pi_{a_2}^{u_l}$};
	\end{tikzpicture}

	\caption{Given the two disjoint geodesics $\pi_{a_1}^{b_1}$ and $\pi_{a_2}^{b_2}$, we construct two geodesics, $\pi_{a_1}^{u_l}$ and $\pi_{a_2}^{u_r}$, whose upper endpoints are $\phi$ close to one another.}
	\label{fig:2}
\end{figure}
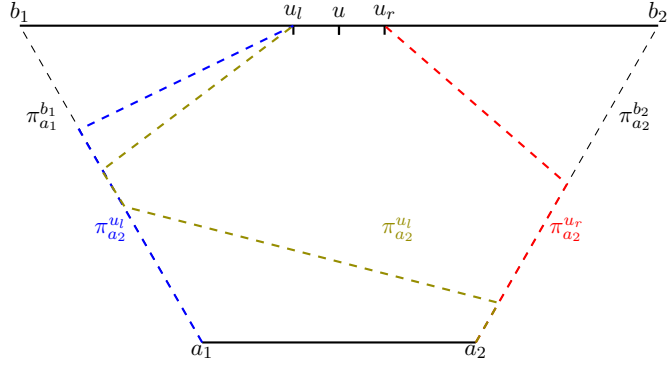
\section{From geodesics in a trapezoid to line ensembles}\label{sec:1}
In this section we describe the main ideas behind the upper bound on $\P(\Mdisjc_{\cone_{\text{scl}}}^{\epsilon}\geq 2)$.  These ideas originated in the work of Hammond in \cite{Hammond1,Hammond2}, where the author developed methods to relate the existence of $k\in\N$ disjoint geodesics in Brownian LPP and the behavior of Brownian Gibbs ensembles related to that model. These ideas were then introduced into the setup of the directed landscape through the work of Dauvergne and Zhang in \cite{Dauvergne-Zhang-2021} and used in \cite{Dauvergne23} to study the different networks of geodesics in the directed landscape. We shall refer to these suite of techniques as the HDZ technique. Thus, the main idea behind the  HDZ technique is to study  disjoint geodesics in random geometries in the KPZ class via the study of events of line-closeness in line ensembles.
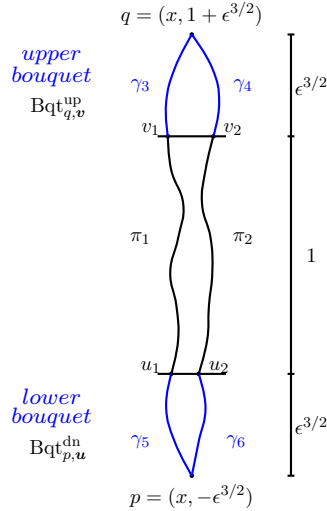
\begin{figure}[t]	
	\centering
	\begin{tikzpicture}[scale=0.45, every node/.style={scale=0.7}]
		\draw[thick] (-1,0) -- (1,0);
		\draw[thick] (-1,7) -- (1,7);
		
		\draw[fill] (0,10) circle (.3ex);
		\node [scale=1][] at (0,10.6)
		{$q=(x,1+\epsilon^{3/2})$};
		
		\node [blue,scale=1][] at (-1.5,8.5)
		{$\gamma_3$};
		\node [blue,scale=1][] at (1.5,8.5)
		{$\gamma_4$};
		
		\node [blue,scale=1][] at (-1.5,-2)
		{$\gamma_5$};
		\node [blue,scale=1][] at (1.3,-2)
		{$\gamma_6$};
		
		\node [scale=1][] at (-1.5,4)
		{$\pi_1$};
		\node [scale=1][] at (1.5,4)
		{$\pi_2$};
		
		\draw[fill] (0,-3) circle (.3ex);
		\node [scale=1][] at (0,-3.6)
		{$p=(x,-\epsilon^{3/2})$};
		
		\draw[fill] (0.63,7) circle (.3ex);
		\node [scale=1][] at (1.2,7.2)
		{$v_2$};
		\draw[fill] (-0.7,7) circle (.3ex);
		\node [scale=1][] at (-1.2,7.2)
		{$v_1$};
		
		\draw[fill] (0.2,0) circle (.3ex);
		\node [scale=1][] at (0.8,0.2)
		{$u_2$};
		\draw[fill] (-0.6,0) circle (.3ex);
		\node [scale=1][] at (-1.2,0.2)
		{$u_1$};
		
		\draw [thick] plot [smooth, tension=0.8] coordinates { (0.2,0) (0.45,1) (0.5,2) (0.6,3) (0.55,4) (0.3,5) (0.4,6) (0.63,7)};
		\draw [thick] plot [smooth, tension=0.8] coordinates { (-0.6,0) (-0.4,1) (-0.45,2) (-0.63,3) (-0.4,4) (-0.25,5) (-0.6,6) (-0.7,7)};
		
		\draw [blue, thick] plot [smooth, tension=0.8] coordinates { (-0.7,7) (-0.75,8) (-0.5,9) (0,10)};	
		\draw [blue, thick] plot [smooth, tension=0.8] coordinates { (0.63,7) (0.8,8) (0.6,9) (0,10)};
		
		\draw [blue, thick] plot [smooth, tension=0.8] coordinates { (-0.6,0) (-0.75,-1) (-0.5,-2) (0,-3)};	
		\draw [blue, thick] plot [smooth, tension=0.8] coordinates { (0.2,0) (0.4,-1) (0.2,-2) (0,-3)};
		
		\begin{scope}[shift={(0.9,0)}]
			\draw[thick] (2,0) -- (2,-3);
			\draw[thick] (1.9,0) -- (2.1,0);		
			\draw[thick] (1.9,-3) -- (2.1,-3);
			
			\node [scale=1][] at (2.6,-1.5)
			{$\epsilon^{3/2}$};
			
			\draw[thick] (2,7) -- (2,10);
			\draw[thick] (1.9,7) -- (2.1,7);		
			\draw[thick] (1.9,10) -- (2.1,10);
			\node [scale=1][] at (2.6,8.5)
			{$\epsilon^{3/2}$};
			
			\draw[thick] (2,7) -- (2,0);
			\node [scale=1][] at (2.6,3.5)
			{$1$};
		\end{scope}
		\begin{scope}[shift={(- 0.4,0)}]
			\node [scale=1.5][] at (-3.7,9)
			{$\textcolor{blue}{\substack{upper\\ bouquet}}$};
			\node [scale=1][] at (-3.5,7.8)
			{$\bqtu_{q,\bm{v}}$};
			
			\node [scale=1.5][] at (-3.7,-1)
			{$\textcolor{blue}{\substack{lower\\ bouquet}}$};		
			\node [scale=1][] at (-3.5,-2.2)
			{$\bqtd_{p,\bm{u}}$};
		\end{scope}
	\end{tikzpicture}
	\caption{The Bouquet construction.}\label{fig:4}
\end{figure}
We shall now outline the general ideas behind the HDZ. Consider the trapezoid $\cone_\epsilon=\trpz\big((0,0),\epsilon;(0,1),\epsilon\big)$. Suppose we would like to find an upper bound for  $\P\big(\disjc^2_{\cone_\epsilon}\big)$, i.e.\ we would like to bound the probability of two  disjoint geodesics of life time $1$ whose endpoints are within $\epsilon$ close to one another.\vspace{7pt}\\
{ \textit{HDZ technique}}
\begin{enumerate}\label{HDZ}
	\item \label{it1} On the event $\disjc^2_{\cone_\epsilon}$, it is possible (although generally non-trivial) to find two disjoint paths $\gamma^1$ and $\gamma^2$ (not necessarily geodesics) from $p=(0,-\epsilon^{3/2})$ to $q=(0,1+\epsilon^{3/2})$, such that 
	\begin{equation}\label{eq7}
		2\dl(\pi_p^q)-\big[\dl(\gamma^1)+\dl(\gamma^2)\big]= O\big(\epsilon^{1/2}\big).
	\end{equation} 
\item \label{it:gt}From Theorem \ref{thm:DZ}, which is a version of  Greene's Theorem for the DL (see also \cite{Greene74}), 
	\begin{equation}
		\begin{aligned}
			\dl(\pi_p^q)&=\ele_1(0)\\
			\dl(\gamma^1)+\dl(\gamma^2)&\leq \ele_1(0)+\ele_2(0),
		\end{aligned}
	\end{equation} 
	where $\ele_i$ is the $i$'th line in the parabolic Airy line ensemble. The last display along with \eqref{eq7} imply that 
\begin{equation}
	2\dl(\pi_p^q)-\big[\dl(\gamma^1)+\dl(\gamma^2)\big]= O\big(\epsilon^{1/2}\big)\implies \ele_1(0)-\ele_2(0)=O(\epsilon^{1/2}).
\end{equation} 
\item \label{it:lc}Using the Brownian Gibbs property of the Airy line ensemble \cite{CorwinHammond,Hammond1,dauvergne2023wiener},  one can obtain a bound on the closeness of the first two lines at $x=0$ 
\begin{equation}\label{eq26}
	\P\Big(\ele_1(0)-\ele_2(0)=O(\epsilon^{1/2})\Big)\leq  C\epsilon^{\frac32}.
\end{equation}
\end{enumerate} 
 	Next we explain how we implement the three different steps above in this paper. 
 	
     Step \ref{it1}:   The event $\disjc^2_{\cone_\epsilon}$ implies the existence of two disjoint geodesics $\pi^1$ and $\pi^2$ in $\cone_\epsilon$, with endpoints
    \begin{equation}
    	\begin{aligned}
    		u^1&=(\pi^1(0),0) \qquad v^1=(\pi^1(1),1)\\
    		u^2&=(\pi^2(0),0) \qquad v^2=(\pi^2(1),1).
    	\end{aligned}
    \end{equation}
    An \textit{upper} (lower) \textit{bouquet} consists of two disjoint paths which share their north(south)-most endpoint $q$ ($p$). We find (we shall shortly explain how) two disjoint paths $\gamma^3$ and $\gamma^4$ going from $v_1$ and $v_2$ (respectively) to the point $q=(0,1+\epsilon^{3/2})$. We call $\bqtu_{q,\bm{v}}=\gamma^3\cup \gamma^4$ the upper bouquet going from the vector $\bm{v}=(v_1,v_2)\in \R^4_{\uparrow}$  to the point $q$. Similarly, we construct the lower bouquet $\bqtd_{p,\bm{u}}=\gamma^5\cup \gamma^6$ that consists of two disjoint paths connecting the points $p$ and $\bm{u}$. See Figure \ref{fig:4}. We now define
    \begin{equation}
    	\begin{aligned}
    		\gamma^1=&\gamma^5\cup\pi^1\cup\gamma^3\\
    		\gamma^2=&\gamma^6\cup\pi^2\cup\gamma^4.
    	\end{aligned}
    \end{equation}
    Notice that from construction, $	\gamma^1$ and $\gamma^2$ are disjoint paths (not necessarily geodesics), which concludes the first part of item \ref{it1}. How does one construct a bouquet in the first place? In other words, how does one find two disjoint  paths between two points to one shared point?  In \cite{Hammond2}, the  model  was Brownian LPP, there the author utilized a natural order of paths in two melons sharing one endpoint (see \cite[Appendix B]{Hammond2}). In \cite{Dauvergne23}, and as we shall do in this work, the  bouquet is taken to be the disjoint optimizer from Theorem \ref{thm:do} i.e.\  $\bqtd_{p,\bm{u}}= \bm{\pi}$, where $\bm{\pi}$ is the disjoint optimizer between the point $p$ and the points $(u^1;u^2)$. From here, to obtain \eqref{eq7}, we must show that a) the weights of the bouquets $\bqtu_{q,\bm{v}}$ and $\bqtd_{p,\bm{u}}$ are of order $\epsilon^{1/2}$ b) control on the modulus of continuity of the Airy$_2$ process. Both $a)$ and $b)$ are then treated by Lemmas \ref{lm:ub} and \ref{lm:ub4}. 
   
	Step \ref{it:gt} : In both the Brownian LPP setup as well as the DL, we use a version of the celebrated Greene's Theorem which says that the maximal weight along $k$ disjoint geodesics equals the sum of the first $k$ lines of the RSK output. The version we need here is \cite[Corollary 1.9]{Dauvergne-Zhang-2021} which we state here as Theorem \ref{thm:DZ}.
	
	Step \ref{it:lc}: The idea here is to use the Brownian Gibbs property of the Airy line ensemble in order to show that the probability of the top two lines in the ensemble being $\epsilon^{1/2}$ close to each other is small enough for our applications. However, here it gets a bit more complicated. As we are ultimately interested in bounding the probability  $\P(\Mdisjc_{\cone_{\text{scl}}}^{\epsilon}\geq 2)$, we actually have \textit{two pairs} of disjoint geodesics to take into account. The end points of the first and second pairs can be as far as $O(L(\epsilon^{-1}))$ (the length of the top basis of the trapezoid) from each other. This translates to a  bound on the probability of the top two lines of the Airy line ensemble being $\epsilon^{1/2}$ close to one another at two points $u_1\leq u_2$ in the interval $[-L(\epsilon^{-1})/2,L(\epsilon^{-1})/2]$ (Figure \ref{fig:3}). To accomplish this, we use  the recent results in \cite{dauvergne2023wiener}.

   \begin{figure}[h!]
   	\begin{subfigure}{0.5\textwidth}	
   		\centering
   		\begin{tikzpicture}[scale=0.45, every node/.style={scale=0.7}]
   			\draw[thick] (-1,0) -- (1,0);
%
%
   			\draw[dashed] (-7,7) -- (7,7);

   			\begin{scope}[shift={(5,-1.3)}]
   			\draw[thick] (-6,7.5) -- (-4,7.5);
   			\draw[thick] (-6,7.4) -- (-6,7.6);
   			\draw[thick] (-4,7.4) -- (-4,7.6);
   			\node [scale=1][] at (-5,7)
   			{$\epsilon$};
   			\end{scope}
   			
   			\draw[red, line width=0.5mm] (-1,6.8) -- (-1,7.2);
   			\draw[red, line width=0.5mm] (1,6.8) -- (1,7.2);
   			
   			\draw[red,line width=0.5mm] (7,6.8) -- (7,7.2);
   			\draw[red, line width=0.5mm] (5,6.8) -- (5,7.2);
   			
   			\draw [thick] plot [smooth, tension=0.8] coordinates { (0.2,0) (0.45,1) (0.5,2) (0.6,3) (0.55,4) (0.3,5) (0.4,6) (0.63,7)};
   			\draw [thick] plot [smooth, tension=0.8] coordinates { (-0.6,0) (-0.4,1) (-0.45,2) (-0.63,3) (-0.4,4) (-0.25,5) (-0.6,6) (-0.7,7)};
   			
   			\draw [blue, thick] plot [smooth, tension=0.8] coordinates { (-0.7,7) (-0.75,8) (-0.5,9) (0,10)};	
   			\draw [blue, thick] plot [smooth, tension=0.8] coordinates { (0.63,7) (0.8,8) (0.6,9) (0,10)};
   			
   			\draw [blue, thick] plot [smooth, tension=0.8] coordinates { (-0.6,0) (-0.75,-1) (-0.5,-2) (0,-3)};	
   			\draw [blue, thick] plot [smooth, tension=0.8] coordinates { (0.2,0) (0.4,-1) (0.2,-2) (0,-3)};
   			
   			\draw [ thick] plot [smooth, tension=0.8] coordinates { (0.45,1) (1.2,2) (2.2,3) (3,4) (4.3,5) (5,6) (5.6,7)};	
   			\draw [ thick] plot [smooth, tension=0.8] coordinates {(0.9,0) (1.2,1) (1.8,2) (2.7,3) (3.6,4) (4.7,5) (6,6) (6.5,7)};
   			
   			\draw [olive, thick] plot [smooth, tension=0.8] coordinates { (5.6,7) (5.7,8) (6.2,9) (7,10)};	
   			\draw [olive, thick] plot [smooth, tension=0.8] coordinates { (6.5,7) (7,8) (7.1,9) (7,10)};
   			
   			\draw [olive, thick] plot [smooth, tension=0.8] coordinates { (0.9,0) (1,-1) (0.6,-2) (0,-3)};
   			
   			\draw[thick] (2,0) -- (2,-3);
   			\draw[thick] (1.9,0) -- (2.1,0);		
   			\draw[thick] (1.9,-3) -- (2.1,-3);
   			
   			\node [scale=1][] at (2.6,-1.5)
   			{$\epsilon^{3/2}$};
  			\draw[thick] (2,7) -- (2,10);
   			\draw[thick] (1.9,7) -- (2.1,7);		
   			\draw[thick] (1.9,10) -- (2.1,10);
   			
   			\node [scale=1][] at (2.6,8.5)
   			{$\epsilon^{3/2}$};
   			\node [scale=1][] at (-2.2,9.5)
   			{$\textcolor{blue}{\substack{upper\\ bouquet}}$};
   			\draw[fill] (0,9.95) circle (.4ex) ;
   			\node [scale=1][] at (0.1,10.4)
   			{$q^1$};
   			\node [scale=1][] at (-2.2,-2)
   			{$\textcolor{blue}{\substack{lower\\ bouquet}}$};
   			\node [scale=1][] at (7.1,10.4)
   			{$q^2$};
   			\draw[fill] (7,9.95) circle (.4ex) ;
   			\node [scale=1][] at (0,-3.3)
   			{$p$};
   			\draw[fill] (0,-2.95) circle (.4ex) ;
   			
   				\node [scale=1][] at (0,7.3)
   			{$u_1$};
   				\node [scale=1][] at (6.1,7.3)
   			{$u_2$};
   			\draw[fill] (0,7) circle (.4ex) ;
   			\draw[fill] (6.1,7) circle (.4ex) ;
   		\end{tikzpicture}
   		\caption{The event $\{\Mdisjc_{\cone_{\text{scl}}}^{\epsilon}\geq 2\}$}\label{fig:bqt}
   	\end{subfigure}%
   	\hspace*{\fill}
   	\begin{subfigure}{0.5\textwidth}	
   		\centering
   		\begin{tikzpicture}[scale=0.5, every node/.style={scale=0.8}]
   			\draw[thick] (-7,7) -- (7,7);
   			\draw[red, line width=0.5mm] (-0.5,6.8) -- (-0.5,7.2);
   			\draw[red,line width=0.5mm] (0.5,6.8) -- (0.5,7.2);
   			
   			\draw[red,line width=0.5mm] (6.5,6.8) -- (6.5,7.2);
   			\draw[red,line width=0.5mm] (5.5,6.8) -- (5.5,7.2);
   			
   			\draw [thick] plot [smooth, tension=0.8] coordinates {(-7,8) (-6,8.5) (-5,8.7) (-4,9.5) (-3,9.1) (-2,8.7) (-1,8.8) (0,8.5) (1,9.3) (2,9.6) (3,9.4) (5,9.2) (6,8.7) (7,9.2)};
   			\draw [blue,thick] plot [smooth, tension=0.8] coordinates {(-7,6) (-6,5.8) (-5,6.4) (-4,6.2) (-3,6.4) (-2,6.1) (-1,7.8) (0,8.1) (1,7.2) (2,7.5) (3,6.8) (5,7.2) (6,8.4) (7,7.5)};
   			\draw[thick] (1,8) -- (1,8.7);
   			\draw[thick] (0.9,8) -- (1.1,8);
   			\draw[thick] (0.9,8.7) -- (1.1,8.7);
   			\node [scale=1][] at (1.7,8.5)
   			{$\epsilon^{1/2}$};
   			\node [scale=1] at (1.7,10.5)
   			{$\ele_1$};
   			\node [blue,scale=1] at (1.7,6)
   			{$\ele_2$};
   				\node [scale=1][] at (0,7.3)
   			{$u_1$};
   			\node [scale=1][] at (6.1,7.3)
   			{$u_2$};
   			\draw[fill] (0,7) circle (.4ex) ;
   			\draw[fill] (6,7) circle (.4ex) ;
   		\end{tikzpicture}
   		\caption{The Airy lines $\ele_1$ and $\ele_2$ are close where the disjoint geodesics can be found}
   	\end{subfigure}
   	\caption{Two aspects of the event $\disjc^3_{\cone_\text{scl}}$}\label{fig:3}
   \end{figure}
    \section{Upper bound on the probability of three disjoint geodesics in a trapezoid}\label{sec:ubt}
       In this section we prove Proposition \ref{prop:2}, i.e.\ we obtain a bound on the probability of three disjoint geodesics in the trapezoid $\cone_{\text{scl}}:=\trpz\big((0,0),\epsilon;(0,1), L(\epsilon^{-1}) \big)$. The proof is somewhat dense with notation and we suggest the reader to consult Figure \ref{fig:cons} while reading. 
   
    We begin by meshing up the interval $[-L(\epsilon^{-1})/2,L(\epsilon^{-1})/2]$ into intervals of size $\epsilon$ with overlap.  Define the intervals
   \begin{equation}\label{eq83}
   	\begin{aligned}
   		J_i:=\Big[\frac\epsilon2(i-1),\frac\epsilon2(i+1)\Big)\times \{1\}, \qquad \forall i\in \Ind^\epsilon:=\Big\llbracket -\epsilon^{-1}L(\epsilon^{-1}),\epsilon^{-1}L(\epsilon^{-1})\Big\rrbracket.
   	\end{aligned}
   \end{equation}
as well as the interval 
   \begin{equation}
   	\tilde{J}:=[-\frac{\epsilon}2,\frac{\epsilon}2]\times\{0\}.
   \end{equation}
   and the points 
   \begin{equation}\label{eq74}
   	p=\big(0,-t_\epsilon\big), \qquad q^i=\Big(i\frac{\epsilon}2\frac{(1+2t_\epsilon)}{(1+t_\epsilon)},1+t_\epsilon\Big)
   \end{equation}
   where
   \begin{equation}
   	t_\epsilon=[H(\epsilon^{-1})]^{-3/2}\epsilon^{3/2},
   \end{equation}
  for some increasing function $H$ to be specified later. Our choice of the point $q^i$ is such that  the straight line starting from $p$ and terminating at $q^i$ crosses through the point $(\frac{i\epsilon}2,1)$ the midpoint of the interval $J_i$. 
   Recall  $\disjc^k_{\cone_{\text{scl}}}$ from \eqref{dijs} and Lemma \ref{lem:dis}. By definition, on the event $\{\Mdisjc_{\cone_{\text{scl}}}^{\epsilon}\geq k\}$, there exist at least $k$ ordered pairs of disjoint geodesics in $\cone_{\text{scl}}$. If $\bm{\pi}=(\pi^1,\ldots,\pi^{k+1})$ are $k+1$ disjoint geodesics in  $\cone_{\text{scl}}$,  we denote the $k$ ordered pairs of disjoint geodesics constructed from $\bm{\pi}$ in Lemma \ref{lem:dis} with $\phi=\epsilon/2$ by $\hat{\bm{\pi}}=(\hat{\pi}^1,\ldots,\hat{\pi}^k)$ where  for $i\in[k]$ $\hat{\pi}^i=(\hat{\pi}_1^i,\hat{\pi}^i_2)$, and  $\hat{\pi}^i_1$ and $\hat{\pi}^i_2$ are two disjoint geodesics in $\cone_{\text{scl}}$ s.t.\  $\hat{\pi}^i_2(1)-\hat{\pi}^i_1(1)<\epsilon/2$ and $\hat{\pi}^i_1\leq \hat{\pi}^i_2$.  Define the vectors $\bar{\bm{v}}=(\bm{v}^1,\ldots,\bm{v}^k),\bar{\bm{u}}=(\bm{u}^1,\ldots,\bm{u}^k)\in(\R^2\times\R^2)^k$ via  
   \begin{equation}\label{eq75}
   	\bm{v}^i=(v^i_1,v^i_2)=\big([\hat{\pi}_1^i(1)],[\hat{\pi}_2^i(1)]\big),\qquad  \bm{u}^i=(u^i_1,u^i_2)=\big([\hat{\pi}_1^i(0)],[\hat{\pi}_2^i(0)]\big), \qquad \forall i\in [k].
   \end{equation}
   The $i$'th element of $\bar{\bm{v}}$ (resp. $\bar{\bm{u}}$) is therefore a vector of two points on the plane that registers the endpoints of $\hat{\pi}_1^i$ and $\hat{\pi}_2^i$ in $\coneu_{\text{scl}}$ (resp. $\coned_{\text{scl}}$). Next we define
   \begin{equation}\label{eq78}
   	\hat{l}_i(\hat{\bm{\pi}})=\max\{l\in \Ind^\epsilon:[\hat{\pi}_1^i(1)]\in J_{l} \},\quad i\in[k].
   \end{equation}
   To ease the notation we shall omit the dependency on $\hat{\bm{\pi}}$ and simply write $	\hat{l}_i$. From $\hat{\pi}^i_1\leq \hat{\pi}^i_2$ and $\hat{\pi}^i_1(1)-\hat{\pi}^i_2(1)<\epsilon/2$, it must be that $	\hat{l}_i$ is the largest   index  $j\in\Ind^\epsilon$ such that  $[\hat{\pi}_1^i(1)],[\hat{\pi}_2^i(1)]\in J_j$. Finally we define the random vector  $\hat{q}=(\hat{q}^1,\ldots,\hat{q}^k)\in(\R\times \{1+t_\epsilon\})^k$ via
   \begin{equation}
   	\hat{q}^i=q^{\hat{l}_i}, \quad i\in[k].
   \end{equation}
   Note that $\hat{q}^i$ is the point from \eqref{eq74} associated with the interval  $J_{\hat{l}_i}$.  Let us now define the set of paths over a prescribed timeline in the DL i.e.\ for real numbers $s<t$
   \begin{equation}
   	\paths st=\{[\pi(r)]_{r\in[s,t]}: \text{$\pi$ is a continuous function from  $[s,t]$ to $\R$}\}.
   \end{equation}
   For $s<r<t$ reals, $\pi\in\paths sr$ and $\eta \in \paths rt$ such that $\pi(r)=\eta(r)$ we define the concatenated path $\eta\oplus \pi \in\paths st$  via
   \begin{equation}
   	[\eta \oplus \pi] (z)=
   	\begin{cases}
   		[\pi (z)] & z\in[s,r]\\
   		[\eta(z)] & z\in [r,t].
   	\end{cases} 
   \end{equation}
   Let $k\in\N$, $s<r<t$ be reals,  $\bm{\pi}=(\pi^1,\ldots,\pi^k)\in \paths sr^k$ and $\bm{\eta}=(\eta^1,\ldots,\eta^k)\in \paths rt^k$, where we used the notation $\paths sr^k=\big(\paths sr\big)^k$. Assume further  that $\pi^i(r)=\eta^i(r)$ for all $i\in[k]$, we extend the definition of $\oplus$ to multi-paths via
   \begin{equation}
   	\big(\bm{\eta}\oplus\bm{\pi}\big)(i)=\eta^i\oplus\pi^i \quad i\in[k].
   \end{equation}
   We now define the upper and lower bouquets, as was explained in Section \ref{sec:1} via the maximal optimizers in the extended landscape. Precisely, for $p=(x,t)\in\R^2$, $s>0$ and $\bm{u}=(u^1,\ldots,u^k)\in \R\times\{t+s\}$ s.t.\ $(u^1_1,\ldots,u^k_1)\in\R_{\geq}$ we define
   \begin{equation}
   	\bqtd_{p,\bm{u}}=\arg \max_{\pi_1,\ldots\pi_k}\sum_{i=1}^k\dl(\pi_i)
   \end{equation}
   where the $\sup$ is over all disjoint paths from $p$ to the points $u^i$. When there is more than one $k$-tuple of disjoint optimizers we chose one of them, thus, strictly speaking the equality in the last display is in fact containment. Similarly we define $\bqtu_{\bm{u},q}$ - for $q=(x,t)\in\R^2$, $s>0$ and $\bm{v}=(v^1,\ldots,v^k)\in \R\times\{t-s\}$ s.t.\ $(v^1_1,\ldots,v^k_1)\in\R_{\geq}$ we define
   \begin{equation}
   	\bqtu_{q,\bm{v}}=\arg \max_{\pi_1,\ldots\pi_k}\sum_{i=1}^k\dl(\pi_i)
   \end{equation}
   Next we define
   \begin{equation}
   	\gamma^{\bm{\hat{\pi}}}=\big(\gamma_1^{\bm{\hat{\pi}}},\ldots,\gamma_k^{\bm{\hat{\pi}}}\big)\in \big(\paths {-t_\epsilon}{1+t_\epsilon}^2\big)^k
   \end{equation}
   where
   \begin{equation}
   	\gamma_i^{\bm{\hat{\pi}}}=\bqtu_{	\hat{q}^i,\bm{v}^i}\oplus \hat{\bm{\pi}}(i) \oplus \bqtd_{p,\bm{u}^i} \qquad i\in[k].
   \end{equation}
   In words, for $i\in[k]$, $\gamma^{\bm{\hat{\pi}}}_i$ is a pair of disjoint paths (generally not geodesics) emanating from $p$ and terminating at $\hat{q}^i$ (see Figure \ref{fig:cons}). Informally, $\gamma^{\bm{\hat{\pi}}}_i$ is the construction of two disjoint paths with shared points from the two geodesics $(\hat{\pi}^i_1,\hat{\pi}^i_2)$, as described in Step 1 of the HDZ technique. For a trapezoid $\cone$ and $k\geq 2$ we denote by
   \begin{equation}
   	\disjcg^k_{\cone}=\{\pi_1,\pi_2,\ldots,\pi_k\in\uni: \pi_1,\pi_2,\ldots,\pi_k\text{ are disjoint in $\cone$}\},
   \end{equation}
     the set of all $k$-tuples of disjoint paths in $\cone$. Note the difference between $\disjcg^k_{\cone}$ and  $\disjc^k_{\cone}$ defined earlier in \eqref{dijs}, where the former is a random set of $k$-tuples of disjoint paths while the latter is the event where  $k$-tuples of disjoint paths exists. Nevertheless, the following holds
     \begin{equation*}
     	\{\disjcg^k_{\cone}\neq\emptyset\}=\disjc^k_{\cone}.
     \end{equation*}
    For each $i\in \Ind_\epsilon$ we define the trapezoid
   $\tilde{\cone}_i=J_i\cup \tilde{J}$.  For a path $\pi$, recall the definition of $\dl(\pi)$ from \eqref{eq136} - the action of the DL on paths. The statement and proof of the following result uses the  setup and notation leading to \eqref{eq75} with  $k=2$. The result itself can be viewed as an analogue of \eqref{eq7} in Step \ref{it1} in the HDZ. 
   \begin{figure}	
   	\centering
   \begin{tikzpicture}[scale=0.65, every node/.style={scale=0.8}]
   	\draw[dashed] (-7,7) -- (9,7);

   	\draw[purple,thick] (-1,0) -- (1,0);
   	
   	\draw[thick] (-7,6.8) -- (-7,7.2);
   	\draw[thick] (-5,6.8) -- (-5,7.2);
   	
%
   	\node [scale=1][] at (-5,5.8)
   	{$\epsilon$};
   	
   	\begin{scope}[shift={(0,-1.3)}]
   		\draw[thick] (-6,7.5) -- (-4,7.5);
   		\draw[thick] (-6,7.4) -- (-6,7.6);
   		\draw[thick] (-4,7.4) -- (-4,7.6);
   	\end{scope}
   	
   	\draw[purple, line width=0.5mm] (-1,6.8) -- (-1,7.2);
   	\draw[purple, line width=0.5mm] (1,6.8) -- (1,7.2);
   	\draw[purple, line width=0.5mm] (-1,7) -- (1,7);
   	
  	\draw[purple, line width=0.5mm] (-4,6.8) -- (-4,7.2);
   	\draw[purple, line width=0.5mm] (-6,6.8) -- (-6,7.2);
   	\draw[purple, line width=0.5mm] (-6,7) -- (-4,7);
   	
   	\draw[purple,line width=0.5mm] (7,6.8) -- (7,7.2);
   	\draw[purple, line width=0.5mm] (5,6.8) -- (5,7.2);
   	 \draw[purple, line width=0.5mm] (5,7) -- (7,7);
   	   	
   	\draw [orange,dashed, thick] plot [smooth, tension=0.8] coordinates { (0.2,0) (0.45,1) (0.5,2) (0.6,3) (0.55,4) (0.3,5) (0.4,6) (0.63,7)};
   	\draw [orange,dashed, thick] plot [smooth, tension=0.8] coordinates { (-0.6,0) (-0.4,1) (-0.45,2) (-0.63,3) (-0.4,4) (-0.25,5) (-0.6,6) (-0.7,7)};
   	
   	\draw [blue, dashed] plot [smooth, tension=0.8] coordinates { (-0.7,7) (-0.75,8) (-0.5,9) (0,10)};	
   	\draw [blue, dashed] plot [smooth, tension=0.8] coordinates { (0.63,7) (0.8,8) (0.6,9) (0,10)};
   	
   	\draw [blue, dashed] plot [smooth, tension=0.8] coordinates { (-0.6,0) (-0.75,-1) (-0.5,-2) (0,-3)};	
   	\draw [blue, dashed] plot [smooth, tension=0.8] coordinates { (0.2,0) (0.4,-1) (0.2,-2) (0,-3)};
   	
   	\draw [orange,dashed, thick] plot [smooth, tension=0.8] coordinates { (0.45,1) (1.2,2) (2.2,3) (3,4) (4.3,5) (5,6) (5.6,7)};	
   	\draw [orange,dashed,, thick] plot [smooth, tension=0.8] coordinates {(0.9,0) (1.2,1) (1.8,2) (2.7,3) (3.6,4) (4.7,5) (6,6) (6.5,7)};
   	
   	\draw [thick] plot [smooth, tension=0.8] coordinates { (-0.6,0) (-0.4,1) (-0.45,2) (-0.8,3) (-1.1,4) (-2,5) (-2.5,6) (-3,7)};
   	\draw [thick] plot [smooth, tension=0.8] coordinates {(0.9,0) (1.2,1) (2,2) (3.2,3) (5,4) (6,5) (7.5,6) (8,7)};
   	\draw [ thick] plot [smooth, tension=0.8] coordinates { (0.2,0) (0.45,1) (0.8,2) (1,3) (1.3,4) (1.5,5) (2,6) (3,7)};
   	
   	\node [scale=1][] at (-3,5)
   	{$\pi^1$};
   	\node [scale=1][] at (2,5)
   	{$\pi^2$};
   	\node [scale=1][] at (7,5)
   	{$\pi^3$};
   	
   	\node [orange, scale=1] at (-1.1,6)
   	{$\hat{\pi}^1_1$};
   	\node [orange, scale=1] at (0.9,6)
   	{$\hat{\pi}^1_2$};
   	\node [orange, scale=1] at (4.2,6)
   	{$\hat{\pi}^2_1$};
   	\node [orange, scale=1] at (6.5,6)
   	{$\hat{\pi}^2_2$};
   	
   	\draw [olive, dashed] plot [smooth, tension=0.8] coordinates { (5.6,7) (5.7,8) (6.2,9) (7,10)};	
   	\draw [olive, dashed] plot [smooth, tension=0.8] coordinates { (6.5,7) (7,8) (7.1,9) (7,10)};
   	
   	\draw [olive, dashed] plot [smooth, tension=0.8] coordinates { (0.9,0) (1,-1) (0.6,-2) (0,-3)};
   	
   	\draw[thick] (2,0) -- (2,-3);
   	\draw[thick] (1.9,0) -- (2.1,0);		
   	\draw[thick] (1.9,-3) -- (2.1,-3);
   	
   	\node [scale=1][] at (2.6,-1.5)
   	{$t_\epsilon$};
   	
   	\draw[thick] (2,7) -- (2,10);
   	\draw[thick] (1.9,7) -- (2.1,7);		
   	\draw[thick] (1.9,10) -- (2.1,10);
   	\node [scale=1][] at (2.6,8.5)
   	{$t_\epsilon$};
   	\draw[fill] (0,9.95) circle (.4ex) ;
   	\node [scale=1][] at (0.1,10.4)
   	{$\hat{q}^1$};
   	\node [scale=1][] at (7.1,10.4)
   	{$\hat{q}^2$};
   	\draw[fill] (7,9.95) circle (.4ex) ;
   	\node [scale=1][] at (0,-3.3)
   	{$p$};
   	\draw[fill] (0,-2.95) circle (.4ex) ;
   	
   	\node [purple, scale=1][] at (0,7.5)
   	{$J_{\hat{l}_1}$};
   	 \node [purple, scale=1][] at (6,7.5)
   	{$J_{\hat{l}_2}$};
   	 \node [purple, scale=1][] at (-5,7.5)
   	{$J_i$};
   	\node [purple, scale=1][] at (-1.3,0.7)
   	{$\tilde{J}$};
   	\node [scale=1][] at (-5.6,10.4)
   	{$q^i$};
   	\draw[fill] (-5.6,9.95) circle (.4ex) ;
   \end{tikzpicture}
   \caption{From $\bm{\pi}$ to $\gamma^{\hat{\bm{\pi}}}$ with $k=2$. The solid paths in black are   three geodesics $\bm{\pi}=(\pi^1,\pi^2,\pi^3)$ from the set $\disjcg^3_{\cone_{\text{scl}}}$. From $\bm{\pi}$ we construct 2 pairs of ordered disjoint geodesics $\hat{\pi}^1=(\hat{\pi}^1_1,\hat{\pi}^1_2)$ and $\hat{\pi}^2=(\hat{\pi}^2_1,\hat{\pi}^2_2)$ with close endpoints. Finally, connecting by bouquets  the upper and lower  endpoints of $\hat{\pi}^j$ to $\hat{q}^j$ and $p$ respectively,  we obtain $\gamma^{\hat{\bm{\pi}}}$}\label{fig:cons}
   \end{figure}
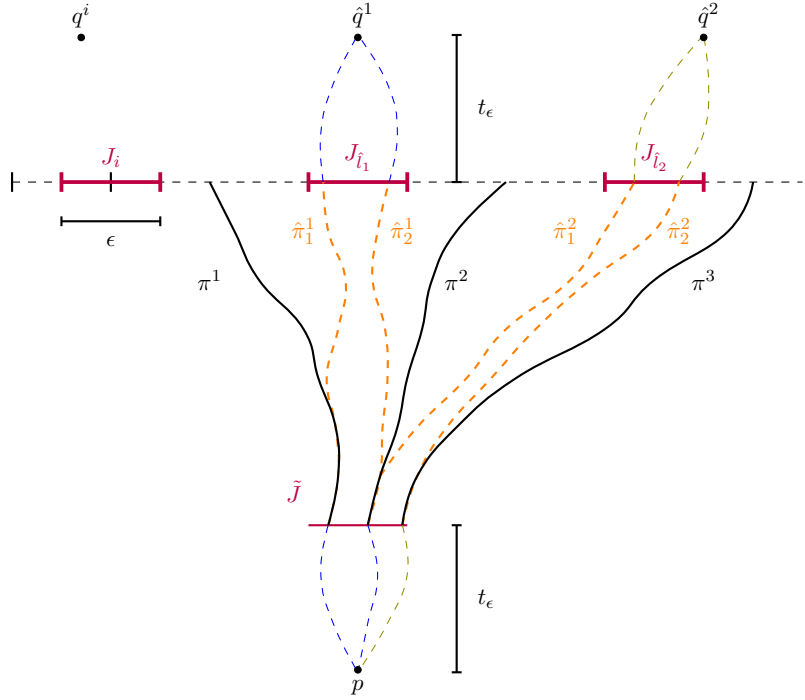
   
   \begin{lemma}\label{lm:7}
   	Fix $0<\delta<1/2$. For $0<\epsilon<1/32$ let  $ H(\epsilon^{-1}):=\tfrac1{16}\epsilon^{-2\delta/3}\vee 8$  and assume  $L(\epsilon^{-1})<\epsilon^{-1/2}/4$. Then there exist constants $c,d>0$ such that 
   	\begin{equation}
   		\P\Big(\exists \bm{\pi}\in\disjcg^3_{\cone_{\rm{scl}}}: \max_{j\in\{1,2\}}\{2\dl(p,\hat{q}^i)-\dl(\gamma_j^{\hat{\bm{\pi}}})\} >\epsilon^{1/2-\delta} \Big)\leq c\exp\Big(-d\epsilon^{-1/2}\Big).
   	\end{equation}
   \end{lemma}
   \begin{proof}
   	Recall that for $i\in\Ind_\epsilon$, the point $q^i$ was constructed such that the straight line going from the point $p$ to the point $q^i$ crosses through the point $(\epsilon i,1)$, the center of the interval $J_i$.  Define the event 
   	\begin{equation}
   		\text{GeoLoc}_\epsilon=\big \{\pi_p^{\hat{q}^j}(1) \in J_{\hat{l}_j},\pi_p^{\hat{q}^j}(0) \in \tilde{J}, \text{ for all $j\in\{1,2\}$}\big\}.
   	\end{equation}
   	Our hypothesis on $L(\cdot),\epsilon$ and  our choice of $t_\epsilon$ guarantee that any straight line going from the point $p$ to any $q^i$ for $i\in\Ind_\epsilon$ must cross the $x$-axis through the interval $(-\epsilon/4,\epsilon/4)$. Using Lemma \ref{lm:12} with $s=-t_\epsilon, t=1+t_\epsilon$, $M=H(\epsilon^{-1})/4>1$ and with $r=0$ and $r=1$ shows that there exist $c,d>0$ such that
   	\begin{equation}\label{eq15}
   		\P(\text{GeoLoc}_\epsilon)\geq 1-c\exp\big(-d[H(\epsilon^{-1})]^2\big).
   	\end{equation}
   	Recall the vector of points $\bar{\bm{u}}$ and $\bar{\bm{v}}$ from \eqref{eq75}. For   a union of geodesics  $\bm{\pi}=\cup_{i=1}^k\pi_i$, we define $\dl(\bm{\pi})=\sum_{i=1}^k\dl(\pi_i)$. By definition of $\gamma^{\hat{\bm{\pi}}}$ we can write
   	\begin{equation}\label{eq19}
   		\begin{aligned}
   			\big|\dl(\gamma^{\hat{\bm{\pi}}}_j)- 2\dl(p,\hat{q}^j)\big|&\leq \big|\dl\big(\bqtd_{p,\bm{u}^j}\big)-2\dl(\pi_{p}^{\hat{q}^j}|_{[-t_\epsilon,0]})\big|+\big|\dl(\hat{\pi}^j)-2\dl(\pi_{p}^{\hat{q}^j}|_{[0,1]})\big|\qquad \forall j\in\{1,2\},\\
   			&+\big|\dl\big(\bqtu_{\hat{q}^j,\bm{v}^j}\big)-2\dl(\pi_{p}^{\hat{q}^j}|_{[1,1+t_\epsilon]})\big|. 
   		\end{aligned}
   	\end{equation}
    Set $\tilde{u}^j=\pi_{p}^{\hat{q}^j}(0)$ and $\tilde{v}^j=\pi_{p}^{\hat{q}^j}(1)$. Recall that  $\hat{\pi}^j_1$ and $\hat{\pi}^j_2$ are geodesics and that by construction $\bm{u}^j:=([\hat{\pi}^j_1(0)],[\hat{\pi}^j_2(0)])\in \tilde{J}\times\tilde{J}$ and $\bm{v}^j:=([\hat{\pi}^j_1(1)],[\hat{\pi}^j_2(0)])=\in J_{\hat{l}_j}\times J_{\hat{l}_j}$.  Then, on the event $\text{GeoLoc}_\epsilon$ the following holds
   	\begin{equation}\label{eq76}
   		\begin{aligned}
   			&\big|\dl(\hat{\pi}^j)-2\dl(\pi_{p}^{\hat{q}^j}|_{[0,1]})\big|\leq \big|\dl(\tilde{u}^j,\tilde{v}^j)-\dl(u^j_1,v^j_1)\big|+ \big|\dl(\tilde{u}^j,\tilde{v}^j)-\dl(u^j_2,v^j_2)\big| \qquad j\in\{1,2\} \\
   			&\leq 2\sup_{\substack{u,u'\in \tilde{J}, v,v'\in J_{\hat{l}_j}}}|\ash_{u;v}-\ash_{u';v'}|+4\epsilon^2,
   		\end{aligned}
   	\end{equation}
   	where $\ash$ stands for the Airy sheet from \eqref{eq108}. The typical fluctuations of the first and third elements on the right side of \eqref{eq19} is $t_\epsilon^{1/3}$. From \eqref{eq108}, on the event $\text{GeoLoc}_\epsilon$ 
   	\begin{equation}\label{eq77}
   		\big|\dl\big(\bqtd_{p,\bm{u}^j}\big)-2\dl(\pi_{p}^{\hat{q}^j}|_{[-t_\epsilon,0]})\big|\leq \sup_{\substack{\bm{u}\in \tilde{J}\times \tilde{J}}}|\ash_{p;\bm{u}}|+2\sup_{\substack{u\in \tilde{J}}}|\ash_{p;u}|+4\frac{\epsilon^2}{t_\epsilon}\qquad j\in\{1,2\},
   	\end{equation}
   	and, from the translation invariance and flip symmetry properties of the extended DL (Lemma \ref{lm:prp}),  a similar bound holds for $|\dl\big(\bqtu_{\hat{q}^j,\bm{v}^j}\big)-2\dl(\pi_{p}^{\hat{q}^j}|_{[1,1+t_\epsilon]})|$. To conclude, we have shown that on $\text{GeoLoc}_\epsilon$ the right hand side of \eqref{eq19} is bounded by 
   	\begin{equation}\label{eq109}
   		\begin{aligned}
   			\big|\dl(\gamma^{\hat{\bm{\pi}}}_i)- 2\dl(p,\hat{q}^j)\big|&\leq 2\Big[\sup_{\substack{\bm{u}\in \tilde{J}\times \tilde{J}}}|\ash_{p;\bm{u}}|+2\sup_{\substack{u\in \tilde{J}}}|\ash_{p;u}|+4\frac{\epsilon^2}{t_\epsilon}\Big]\\
   			&+2\sup_{\substack{u,u'\in \tilde{J}, v,v'\in J_{\hat{l}_i}}}|\ash_{u;v}-\ash_{u';v'}|+4\epsilon^2.
   		\end{aligned}
   	\end{equation}
   	We now use  the hypothesis on $\epsilon$ and $H(\cdot)$ to conclude that the deterministic contribution to the right hand side of \eqref{eq76} and \eqref{eq77} is smaller than $\epsilon^{1/2-\delta}/2$ i.e.\
   	\begin{equation}\label{eq148}
   		\begin{aligned}
   			\epsilon<1/32 &\implies \epsilon^{1/2-\delta}/2>4\epsilon^2\\
   			H(\epsilon^{-1})<\epsilon^{-2\delta/3}/8&\implies \epsilon^{1/2-\delta}/2>4\frac{\epsilon^2}{t_\epsilon}.
   		\end{aligned}
   	\end{equation}
   	Next we obtain a bound on the probability of the right hand side of \eqref{eq77} to be large. Let
   	\begin{equation}
   		a=\frac{\epsilon^{1/2-\delta}}{8Gt_\epsilon^{1/3}}=\frac{\epsilon^{-\delta}\big[H(\epsilon^{-1})\big]^{1/2}}{8G}
   	\end{equation}
   	 where $G(\bm{x},\bm{y},0,t_\epsilon)$ is as in \eqref{eq124}, and $\bm{x},\bm{y}\in[-\epsilon/2,\epsilon/2]^2$. From our choice of $H$, it follows that there exists a constant $C>0$ such that $G\leq CH(\epsilon^{-1})\log(\epsilon^{-1})$ and therefore that  $a\geq C \epsilon^{-\delta/2}$ for some constant $C>0$. Lemma \ref{lm:ub} implies that  there exist $c,d>0$
   	
   	\begin{align}
   		&\P\big(\sup_{\substack{\bm{u}\in \tilde{J}^2}}|\ash_{p;\bm{u}}|>\epsilon^{1/2-\delta}/8\big)\stackrel{ \eqref{eq40}}{\leq} c\exp\big(-d\epsilon^{- \delta/2}\big)\label{eq11},\\
   		&\P\big(\sup_{\substack{u\in \tilde{J}}}|\ash_{p;u}|>\epsilon^{1/2-\delta}/8\big)\stackrel{ \eqref{eq40}}{\leq} c\exp\big(-d\epsilon^{-\delta/2}\big)\label{eq13}.
   	\end{align}
   	Using Lemma \ref{lm:ub4} with $b=L$ and $a=\epsilon^{-\delta/2}/4$ 
   	\begin{equation}
   	\begin{aligned}
   			\P\Big(\sup_{\substack{u,u'\in \tilde{J}, v,v'\in J_{\hat{l}_i}}}|\ash_{u;v}-\ash_{u';v'}|>\epsilon^{1/2-\delta}/4, \text{\,\, for all $i\in\{1,2\}$}\Big)\stackrel{ \eqref{eq41}}{\leq} cL^{10}\exp\big(-d\epsilon^{-\frac34 \delta}\big)\label{eq14}\\
   			 \leq c\exp\big(-d\epsilon^{- \delta/2}\big).
   	\end{aligned}
   	\end{equation}
   	Using  \eqref{eq148}, \eqref{eq11}--\eqref{eq14} in \eqref{eq109} along with  \eqref{eq15} we conclude the result. 
   \end{proof}
   Define the event
   \begin{equation}
   	\begin{aligned}
   		\tdis_d:=\{\text{there are three disjoint geodesics } &\gamma^1,\gamma^2,\gamma^3 \text{ in  $\cone_{\rm{scl}}$ such that}\\ &\max_{i,j\in\{1,2,3\}}|\gamma^j(1)-\gamma^i(1)|<d\}
   	\end{aligned}
   \end{equation}
   where one can find three disjoint geodesics in $\cone_{\text{scl}}$ whose endpoints on $\coneu$ are at most distance $d$ away from one another. 
   
   Our next definition involves the concept of a line ensemble. We  defer our results involving line ensembles to Section \ref{sec:Alb}. For the reader who is not familiar with the concept, it will suffice to consider Definition \ref{def:le}.
   Let $\tilde{\epsilon},\epsilon,d>0$,  and let  $\gle:\llbracket1,k\rrbracket\times I\rightarrow \R$ be a line ensemble where $[-L(\tilde{\epsilon}^{-1})/2,L(\tilde{\epsilon}^{-1})/2]\subseteq I$ and where $I$ is an interval. Recall the set $\Ind_{\epsilon}$ from \eqref{eq83}. We define
   \begin{equation}
   	\Cl(\gle,\Ind_{\tilde{\epsilon}};\epsilon,d)=\big\{\exists x_1,x_2\in \tfrac{\tilde{\epsilon}}2\Ind_{\tilde{\epsilon}}: |x_1-x_2|>d, \gle(1,x_i)-\gle(2,x_i)<\epsilon\big\}.
   \end{equation}
   On the event $\Cl(\gle,\Ind_{\tilde{\epsilon}};\epsilon,d)$, the first two lines of the the line ensemble $\gle$ are $\epsilon$-close on two points on the lattice $\tfrac{\tilde{\epsilon}}2\Ind_{\tilde{\epsilon}}$ that are at least distance $d$ away from each other. We let $\bar{\ele_2}=(\ele_1,\ele_2)$ denote the first two lines of the Airy line ensemble. 
   \begin{proposition}\label{prop:3}
   	Fix $0<\beta<1$ and $0<\delta<1/2$, and let $H(\epsilon^{-1})$ be as in Lemma \ref{lm:7}. For $0<\epsilon<2^{-[2\vee(1-\beta)^{-1}]}$ there exist $c,d>0$ such that 
   	\begin{equation}\label{eq86}
   		\P\big(\disjc^3_{\cone_{\rm{scl}}}\big) \leq  \P\big(\Cl(\bar{\ele_2},\Ind_{\tilde{\epsilon}};2\tilde{\epsilon}^{1/2-\delta},\tilde{\epsilon}^\beta)\big)+ \P\big(\tdis_{2\epsilon^{\beta}}\big)+c\exp\big(d\epsilon^{-\frac43 \delta}\big),
   	\end{equation}
   	where 
   	\begin{equation}\label{eq85}
   		\tilde{\epsilon}=\epsilon(1+t_\epsilon)^{-1}(1+2t_\epsilon)^{1/3}.
   	\end{equation}
   \end{proposition}
   \begin{proof}
   	By definition $	\disjc^3_{\cone_{\rm{scl}}}= \{\disjcg^3_{\cone_{\rm{scl}}}\neq \emptyset\}$. Recall \eqref{eq78}. From Lemma \ref{lm:7} with
   	\begin{equation}\label{eq110}
   		H(\epsilon^{-1})=\frac1{16}\epsilon^{-2\delta/3}\vee 8,
   	\end{equation}
   	we have
   	\begin{equation}\label{eq43}
   		\begin{aligned}
   			&\P(\disjc^3_{\cone_{\rm{scl}}},(\tdis_{2\epsilon^{\beta}})^c)= \P(\disjcg^3_{\cone_{\rm{scl}}}\neq \emptyset,(\tdis_{2\epsilon^{\beta}})^c)\\
   			&\leq \P\Big(\forall \bm{\pi}\in\disjcg^3_{\cone_{\rm{scl}}}\neq \emptyset,\, \max_{i\in\{1,2\}}\{2\dl(p,q^{\hat{l}_i})-\dl(\gamma^{\hat{\bm{\pi}}}_i)\} \leq \epsilon^{1/2-\delta}, \tfrac{\epsilon}2|\hat{l}_1-\hat{l}_2|>\epsilon^\beta \Big)+\\
   			&+c\exp\big(d\epsilon^{- \delta/2}\big)+\P\big(\tdis_{2\epsilon^{\beta}}\big)\\
   			&\leq \P\Big(\forall \bm{\pi}\in\disjcg^3_{\cone_{\rm{scl}}}\neq \emptyset, \,\, \max_{i\in\{1,2\}}\{2\dl(p,\hat{q}^i)-\dl(\gamma^{\hat{\bm{\pi}}}_i)\} \leq \epsilon^{1/2-\delta}, |\hat{q}^1_1-\hat{q}^2_1|>\frac{(1+2t_\epsilon)}{(1+t_\epsilon)}\epsilon^\beta \Big)\\
   			&+c\exp\big(d\epsilon^{-\frac42 \delta}\big)+\P\big(\tdis_{2\epsilon^{\beta}}\big),
   		\end{aligned}
   	\end{equation} 
   	where in the first inequality we used Lemma \ref{lm:7} and that under the hypothesis on $\epsilon$ it holds that $2\epsilon<\epsilon^{\beta}$,  in the second inequality we used  the definition of $\hat{q}^i$. From Theorem \ref{thm:DZ}, there exists a coupling between the directed landscape and the parabolic Airy line ensemble $\ele$ such that \eqref{eq42} holds. By definition, for $i\in\{1,2\}$ it holds that  $\dl(\gamma^{\hat{\bm{\pi}}}_i)\leq \dl(p^2,(\hat{q}^i)^2)$.  It therefore follows  that 
   	\begin{equation}\label{eq79}
   		\begin{aligned}
   			&\P\big(2\dl(p,\hat{q}^i)-\dl(\gamma^{\hat{\bm{\pi}}}_i) \leq \epsilon^{1/2-\delta}, \quad \forall i\in\{1,2\}\big)\leq 	\P\big(2\dl(p,\hat{q}^i)-\dl(p^2,(\hat{q}^i)^2) \leq \epsilon^{1/2-\delta}, \quad \forall i\in\{1,2\}\big)\\
   			&=\P\Big(2(1+2t_\epsilon)^{1/3}\dl(0,\tilde{q}^i)-(1+2t_\epsilon)^{1/3}\dl(0^2,(\tilde{q}^i)^2) \leq \epsilon^{1/2-\delta}, \quad \forall i\in\{1,2\}\Big)\\
   			&\stackrel{\eqref{eq42}}{=}\P\big(\ele_1(\tilde{q}^i_1)-\ele_2(\tilde{q}^i_1)\leq (1+2t_{\epsilon})^{-1/3}\epsilon^{1/2-\delta}, \quad \forall i\in\{1,2\}\big),
   		\end{aligned}
   	\end{equation}
   	where $q\mapsto\tilde{q}$ is the translation-scaling map that sends $q=(q_1,q_2)$ to $\tilde{q}=((1+2t_\epsilon)^{-2/3}q_1,(1+2t_\epsilon)^{-1}(q_2+t_\epsilon))$ and where in the penultimate equality we used the scaling property of the DL (Lemma \ref{lm:prp}).  For later reference we remark here that 
   	\begin{equation}\label{eq125}
   		q^i\stackrel{\sim}{\mapsto}\tilde{q}^i= (i\frac{\tilde{\epsilon}}{2},1),
   	\end{equation}
   	where $\tilde{\epsilon}$ was given in \eqref{eq85}. We now pick up from the penultimate line in \eqref{eq43}
   	\begin{equation}\label{eq80}
   		\begin{aligned}
   			&\P\Big(\forall \bm{\pi}\in\disjcg^3_{\cone_{\rm{scl}}}, \,\, \max_{i\in\{1,2\}}\{2\dl(p,\hat{q}^i)-\dl(\gamma^{\hat{\bm{\pi}}}_i)\} \leq \epsilon^{1/2-\delta}, |\hat{q}^1_1-\hat{q}^2_1|>\frac{(1+2t_\epsilon)}{(1+t_\epsilon)}\epsilon^\beta\Big)\\
   			&\stackrel{\eqref{eq79}+\eqref{eq125}}{\leq}  \P\Big(\ele_1(\tilde{q}^i_1)-\ele_2(\tilde{q}^i_1)\leq (1+2t_{\epsilon})^{-1/3}\epsilon^{1/2-\delta}, \text{ for some } |\tilde{q}^1_1-\tilde{q}^2_1|>\frac{(1+2t_\epsilon)^{1/3}}{(1+t_\epsilon)}\epsilon^\beta\Big)\\
   			& \leq  \P\Big(\ele_1(\tfrac{\tilde{\epsilon}}2j_i)-\ele_2(\tfrac{\tilde{\epsilon}}2j_i)\leq \epsilon^{1/2-\delta}, 
   			\text{ for some $j_1,j_2\in \Ind_{\tilde{\epsilon}}$ s.t.\ $\tfrac{\tilde{\epsilon}}2|j_1-j_2|>\frac{(1+2t_\epsilon)^{1/3}}{(1+t_\epsilon)}\epsilon^\beta$} \Big).
   		\end{aligned}
   	\end{equation} 
   	A first order approximation of the function $x\mapsto (1+x)^{1/3}$ shows that for $t_\epsilon\in(0,1)$
   	\begin{equation}\label{eq81}
   		1-t_\epsilon\leq \frac{(1+2t_\epsilon)^{1/3}}{1+t_\epsilon}\leq \frac{1+2t_\epsilon/3}{1+t_\epsilon}<1,
   	\end{equation}
   	which implies
   	\begin{equation}\label{eq82}
   		\epsilon/2<\tilde{\epsilon}<\epsilon,
   	\end{equation}
   	for $t_\epsilon<1/2$ which is satisfied under the hypothesis on $H(\epsilon^{-1})$. Using \eqref{eq81}--\eqref{eq82} in the last line of \eqref{eq80} shows that it can be bounded by 
   	\begin{equation}
   		\begin{aligned}
   			&\P\Big(\ele_1(x_i)-\ele_2(x_i)\leq 2\tilde{\epsilon}^{1/2-\delta},  
   			\text{ for some $x_1,x_2\in \tfrac{\tilde{\epsilon}}{2}\Ind_{\tilde{\epsilon}}$ s.t. $|x_1-x_2|>\tilde{\epsilon}^\beta/2$} \Big)\\
   			& =\P\Big(\Cl(\bar{\ele_2},\Ind_{\tilde{\epsilon}};2\tilde{\epsilon}^{1/2-\delta},\tilde{\epsilon}^\beta/2)\Big).
   		\end{aligned}
   	\end{equation}
   	Using the last display in \eqref{eq80} and then in \eqref{eq43} implies the result.
   \end{proof}
   Proposition \ref{prop:3} suggest that we shall need the following result.
   \begin{lemma}\label{lm:13}
   	There exists $\epsilon_0>0$ such that the following holds.  For every $0<\beta<1$, if $L(\epsilon^{-1})<\epsilon^{-\beta}[\beta\log(\epsilon^{-1})]^{-2/3}$ on $(0,\epsilon_0)$ then for $\epsilon\in (0,\epsilon_0)$
   	\begin{equation}\label{eq88}
   		\P\big(\tdis_{2\epsilon^\beta}\big)\leq  C\exp\big(c(\log\epsilon^{-\beta})^{5/6}\big)L(\epsilon^{-1})\epsilon^{3\beta} 
   	\end{equation}
   \end{lemma}
   \begin{proof}
   	Define the intervals
   	\begin{equation}
   		\begin{aligned}
   			\mathcal{J}_i:=\Big[2\epsilon^{\beta}(i-1),2\epsilon^{\beta}(i+1)\Big)&\times \{1\}, \qquad \forall i\in\Ind_{4\epsilon^\beta}.
   		\end{aligned}
   	\end{equation}
   	and recall $\tilde{J}$ from \eqref{eq83}. Let  $\tilde{\cone}_i=\mathcal{J}_i\cup \tilde{J}$ for all $i\in\Ind_{4\epsilon^\beta}$, and recall the definition of 
   	$\disjc^k_{\cone}$ from \eqref{dijs}. Observe that 
   	\begin{equation}\label{eq45}
   		\tdis_{2\epsilon^\beta}\subseteq \bigcup_{i\in\Ind_{4\epsilon^\beta}}\disjc^{3}_{\tilde{\cone}_i}
   	\end{equation}
   	From the last display it is clear that a bound on the probability of $\disjc^{3}_{\tilde{\cone}_i}$ will be useful. Before we obtain such a bound let us start by defining an event closely related to $\disjc^k_{\tilde{\cone}_i}$.
   	For a trapezoid $\cone$, define the set
   	\begin{equation}
   		\coned_{\Q}:=\coned\cap (\Q\times\R),
   	\end{equation}
   	that is, $\coned_{\Q}$ are all the points on $\coned$ whose first coordinate is rational. Similarly we define $\coneu_{\Q}$.
   	\begin{equation}
   		\begin{aligned}
   			\disjc^{k,\mathbb{Q}}_{\cone}:=&\{\text{there exist $k$ disjoint geodesics $\pi^1,\ldots,\pi^k\in\uni$ whose endpoints}\\  &\text{ are contained in $\coned_\Q\cup \coneu_\Q$ for all $i\in[k]$}\}.
   		\end{aligned}
   	\end{equation}
   	We now claim that
   	\begin{equation}\label{eq47}
   		\P(\disjc^{k}_{\cone})= \P(\disjc^{k,\mathbb{Q}}_{\cone}).
   	\end{equation}
   	We  show $\P(\disjc^{k}_{\cone})\leq  \P(\disjc^{k,\mathbb{Q}}_{\cone})$, as the other inequality is immediate.  
   	\\
   	If $\{\pi_{p^i}^{q^i}\}_{i\in \{1,\ldots,k\}}\in \disjcg^k_{\cone}$, in particular, for all $i\in[k]$, $\pi_{p^i}^{q^i}\in \uni$, i.e.\ $\pi_{p^i}^{q^i}$ is the unique geodesic going from $p^i$ to $q^i$. This implies that almost surly, each of the geodesics in $\disjcg^k_{\cone}$ can be approximated by geodesics with endpoints in $\coned_\Q$ and $\coneu_\Q$ as we shall do next. Let
   	\begin{equation}
   		d_{\text{max}}:=\max_{i,j\in[k]}d_G(\pi_{p^i}^{q^i},\pi_{p^j}^{q^j})
   	\end{equation}
   	where for $\pi,\eta\in \paths sr$ we denote $d_G(\pi,\eta)=\sup_{r\in[s,t]}|\pi(r)-\eta(r)|$. For each $i\in [k]$, let $\{a^i_m\}_{m\in\N}\in \coned_\Q$  be a sequence such that $a^i_m\rightarrow p^i$. Similarly define the sequence $\{b^i_m\}_{m\in\N}\in \coneu_{\Q}$ with respect to $q^i$. By Lemma \ref{lm:14}, with probability $1$, for each $i\in[k]$, $\pi_{a^i_m}^{b^i_m}\rightarrow \pi_{p^i}^{q^i}$ in the overlap sense. This implies that for every $\delta'>0$, there exists $\delta>0$ such that
   	\begin{equation}\label{eq44}
   		\P\Big(\disjc^{k}_{\cone},\sup_{r\in [0,1]}\big|\pi_{a^i_M}^{b^i_M}(r)-\pi_{p^i}^{q^i}(r)\big|<\delta, \forall i\in[k]\Big)>(1-\delta')\P(\disjc^{k}_{\cone}).
   	\end{equation}
   	Let $E_{\delta,M}$ denote the event inside the probability on the left hand side of \eqref{eq44}. Note that if $\omega\in \{d_{\text{max}}>2n^{-1}\}\cap E_{n^{-1},n}$ for some $n\in\N$ then $\omega\in\disjc^{k,\mathbb{Q}}_{\cone}$. This implies that $\lim _{n\rightarrow \infty}E_{n^{-1},n}\subseteq\disjc^{k,\mathbb{Q}}_{\cone}$, which along with \eqref{eq44} implies that $\P(\disjc^{k,\mathbb{Q}}_{\cone})\geq \P(\disjc^{k}_{\cone})$. This concludes the claim.
   	
   	From our hypothesis on $\epsilon$ and $L(\cdot)$, there exists $\epsilon_0>0$ small enough such that the conditions of \cite[Theorem 1.16]{BGH22} with $k=3$, $(x,s;y,t)\mapsto(0,0;y,1)$ and $\epsilon\mapsto 4\epsilon^\beta$ are satisfied. We conclude that for all $i\in \Ind_{4\epsilon^\beta}$
   	\begin{equation}\label{eq46}
   		\P(\disjc^{3,\mathbb{Q}}_{\tilde{\cone}_i})\leq C\exp\big(c(\log\epsilon^{-\beta})^{5/6}\big)\epsilon^{4\beta}.
   	\end{equation}
   	Using a union bound in \eqref{eq45} as well as the bound in \eqref{eq46} and the identity in \eqref{eq47} implies the result. 
   \end{proof}
   We are now ready to prove Proposition \ref{prop:2}.
   \begin{proof}[Proof of Proposition \ref{prop:2}]
   	Recall the definition of $\tilde{\epsilon}$ from \eqref{eq86}. By union bound on the elements of $\Ind_{\tilde{\epsilon}}$
   	\begin{equation}\label{eq84}
   		\P\big(\Cl(\bar{\ele_2},\Ind_{\tilde{\epsilon}};2\tilde{\epsilon}^{1/2-\delta},\tilde{\epsilon}^\beta)\big)\leq \sum_{x_1,x_2\in \frac12\tilde{\epsilon}\Ind_{\tilde{\epsilon}},\,|x_1-x_2|>\tilde{\epsilon}^\beta}\P\Big(\Cl\big(\bar{\ele_2},(x_1,x_2),2\tilde{\epsilon}^{1/2-\delta}\big)\Big),
   	\end{equation}
   	where for $a<b$, a line ensemble $\gle:\llbracket1,2\rrbracket\times[a,b] \rightarrow\R$, $x_1<x_2\in[a,b]$ and $\phi>0$ the event $\Cl\big(\gle,(x_1,x_2),\phi\big)$ is defined in \eqref{eq94}.  We will shortly choose $\beta,\delta>0$ s.t.\ there exists $\tilde{\epsilon}_0(\beta,\delta)$ small enough s.t.\ the conditions of Proposition  \ref{prop:1} are satisfied with $(u_1,u_2)\mapsto(x_1,x_2)$, $u_2-u_1\mapsto \tilde{\epsilon}^\beta$, $\phi \mapsto 2\tilde{\epsilon}^{1/2-\delta}$  and so each element in \eqref{eq84} is bounded by
   	\begin{equation} \label{eq107}
   			\P\Big(\Cl\big(\bar{\ele_2},(x_1,x_2),2\tilde{\epsilon}^{1/2-\delta}\big)\Big)\leq C\tilde{\epsilon}^{(1/2-\delta)6}(\log\tilde{\epsilon}^{\delta-1/2})^4\big(1+\tilde{\epsilon}^{-3\beta/2}\big).
   	\end{equation}
   	 Set $\beta=2/9$ and $\delta=1/90$. The summation in \eqref{eq84} is over at most $4\tilde{\epsilon}^{-2}[L(\tilde{\epsilon}^{-1})]$ elements. Using \eqref{eq107} in  \eqref{eq84} we have
   	\begin{equation}\label{eq87}
   		\begin{aligned}
   			\P\big(\Cl(\bar{\ele_2},\Ind_{\tilde{\epsilon}};2\tilde{\epsilon}^{1/2-\delta},\tilde{\epsilon}^\beta)\big)&\leq C(\log\tilde{\epsilon}^{\delta-1/2})^4L(\tilde{\epsilon}^{-1})\tilde{\epsilon}^{(1/2-\delta)6-3\beta/2-2}\\
   			&= C(\log\tilde{\epsilon}^{\delta-1/2})^4L(\tilde{\epsilon}^{-1})\tilde{\epsilon}^{\,9/15}\\
   			&\leq  CL(2\epsilon^{-1})\epsilon^{\,8/15}.
   		\end{aligned}
   	\end{equation}
   	where in the last inequality we used \eqref{eq82}, which holds since our choice of $H(\cdot)$ in \eqref{eq110} implies that $t_\epsilon<1/2$. Given our choice of the parameters $\beta$ and $\delta$,  consider the bounds in \eqref{eq86} along with those in \eqref{eq88} and \eqref{eq87}.  Note that each of the summands in \eqref{eq86}  is dominated  by the right hand side of \eqref{eq87}. This implies the result.
   \end{proof}
\section{Upper bound on the probability of closeness of Airy lines at two points}\label{sec:Alb}
The goal of this section is to adapt the proof of Step \ref{it:lc} in the HDZ technique for the event of two pairs of disjoint geodesics in $\cone_{\text{scl}}$. More precisely, the main result of this section is Proposition \ref{prop:1} which is one of the key ingredients  for the proof of Proposition \ref{prop:2}. 

Let us briefly explain here what is the type of  result we are after. Recall from Lemma \ref{lem:dis} that the event $\disjc^3_{\cone}$ implies the existence of two pairs of disjoint geodesics $(\pi^1,\eta^1)$ and $(\pi^2,\eta^2)$ in $\cone_{\text{scl}}$, whose endpoints are $\epsilon$-close.  From the heuristics of the HDZ technique in Section \ref{sec:1}, this implies the existence of  $x_1, x_2\in [-L(\epsilon^{-1})/2,L(\epsilon^{-1})/2]$ such that $\ele_1(x_i)-\ele_2(x_i)=O(\epsilon^{1/2})$ for $i\in\{1,2\}$. We would then like to show that the latter event is of low probability, i.e.\
\begin{equation}\label{eq25}
	\P\big(\ele_1(x_i)-\ele_2(x_i)= O(\epsilon^{1/2}) \text{ for $i\in\{1,2\}$}\big)=O\big(\epsilon^3\times (|x_2-x_1|^{-3/2}+1)\big),
\end{equation}
which can be thought of as the appropriate analogue of Step \ref{it:lc}. Clearly \eqref{eq25} would not be of help when $|x_2-x_1|$ is too small, but this is complemented by Lemma \ref{lm:13}. 

  \subsection{Preliminaries}
  We begin by introducing the main ideas and concepts behind Brownian Gibbs line ensemble and obtain some auxiliary results. 
\begin{definition}[Line ensembles]\label{def:le}
	 Let $k\in\N \cup \{\infty\}$ and let $I$ be a closed interval in $\R$. We let $\mathcal{C}^k(I)$ be the space of all $k$-tuple $\bar{f}=(f_1,\ldots,f_k)$ of continuous functions on $I$, endowed with the uniform on compact topology. We denote the associated sigma-algebra by $\els^k(I)$. A line ensemble is a random variable in $\mathcal{C}^k(I)$. If $\gle$ is a line ensemble, we often think of it as a function from $\llbracket1,k\rrbracket\times I$ to $\R$. We write $\gle(i,\cdot)$ or $\gle_i$ for the $i$'th line of  $\gle$.
\end{definition}
 \begin{definition}[Brownian bridge ensemble]
 	Let $k\in\N$, $a,b\in\R$ with $a<b$, and $\bar{x},\bar{y}\in\R^k_>$. Write $\el^{[a,b]}_{k,\bar{x},\bar{y}}$ for the law of the ensemble $\bar{B}_k:\llbracket1,k\rrbracket\times[a,b] \rightarrow\R$ whose curves $\bar{B}_k(i,\cdot):=[a,b]\rightarrow\R,i\in \llbracket1,k\rrbracket$, are independent Brownian bridges that satisfy $B(i,a)=x_i$ and $B(i,b)=y_i$.
 \end{definition}
 Let $f:[a,b]\rightarrow \R\cup\{-\infty\}$ be a measurable function such that $x_k>f(a)$ and $y_k>f(b)$. Define the non-touching event on an interval $A\subseteq[a,b]$ with lower boundary data $f$ by 
 \begin{equation}
 	\NT^A_f=\{\text{for all $x\in A, B(i,x)>B(j,x)$ whenever $1\leq i <j\leq k$, and $B(k,x)>f(x)$}\}.
 \end{equation}
 If $f=-\infty$ on the interval $[a,b]$ then we use the simplified notation $\NT^A$.
 \begin{definition}[Brownian Gibbs property]
 	Fix $n\in\N$, an interval $I\subseteq\R$, $k\in \llbracket 1,n \rrbracket$ and  $a,b\in I$ such that $a<b$. Set $f=\gle_{k+1}$ whenever $k<n$ and $f:=\infty$ when $k=n$. We denote by $D_{k;a,b}=\inti{1,k}\times (a,b)$ and $D^c_{k;a,b}=\big(\inti{1,n}\times I\big) \setminus D_{k;a,b}$. An ordered line ensemble $\gle :\inti{1,n}\times I\rightarrow \R$ has the \textit{Brownian Gibbs property} if, for all such choices of $k,a$ and $b$ 
 	\begin{equation}\label{eq58}
 		\P\Big(\gle|_{D_{k;a,b}}\in E\Big | \gle|_{D^c_{k;a,b}}\Big)=\el^{[a,b]}_{k,\bar{x},\bar{y}}\Big(E\,\,\Big| \NT_f\Big) \qquad \forall E\in\els^n(I),
 	\end{equation}
 	where $\bar{x}=(\gle(1,a),\ldots,\gle(k,a))$ and  $\bar{y}=(\gle(1,b),\ldots,\gle(k,b))$.
 \end{definition}
 The following fundamental result is key to studying the statistics of the Airy line ensemble $\ele=\{\ele_1,\ele_2,\ldots\}$.  
 \begin{proposition}[{\cite[Theorem 3.1]{CorwinHammond}}]\label{prop:4}
 	The parabolic Airy line ensemble satisfies the Brownian Gibbs property.
 \end{proposition}
 For a line ensemble $\gle:\llbracket1,k\rrbracket\times[a,b] \rightarrow\R$, a vector $\bar{u}=(u_1,u_2)$ s.t. $u_1<u_2\in(a,b)$ and  $0<\phi$ small, define the event 
 \begin{equation}\label{eq94}
 	\Cl(\gle,\bar{u},\phi)=\big \{|\gle(l,u_i)-\gle(m,u_i)|<\phi \text{ for  all $i\in\{1,2\},l,m\in\llbracket1,k\rrbracket$}\big \}
 \end{equation}
  	In order to be able to use Proposition \ref{prop:4} and in light of \eqref{eq58} we must understand  the behaviour of non-intersecting Brownian motions. This is the content of the next result. 
 \begin{lemma}[{\cite[Proposition 3.1. (2)]{Hammond1}}]\label{lm:6}
 	Let   $\rho,K>0$, and let $\bar{y}\in[-K,K]^2_>$ and  $\eta\in(0,\frac \rho4 K^{-1})$ and $\bar{x}\in \R^2_>$ such that $x_1<x_2+\eta$. We have that 
 	\begin{equation}
 		\el^{[0,\rho]}_{2;\bar{x},\bar{y}}\big(\NT^{[0,\rho]}\big)=\eta\cdot\rho^{-1}(y_1-y_2)(1+E)
 	\end{equation}
 	where 
 	\begin{equation}
 		E\leq 30\eta\rho^{-1}K.
 	\end{equation}
 \end{lemma}
 Our next result gives an upper bound on the probability of two Brownian motions on a time interval of length $\rho$ to not touch while being close to each other at two given  points.  
 \begin{lemma}\label{lm:2}
 	Fix $\rho>1$ and let  $u_1,u_2\in(0,\rho)$, so that $u_1< u_2$. Then there exists $C>0$, such that for all  $\bar{x},\bar{y}\in\R^2_>$ and $0<\phi<1/2$ s.t.\ 
 	\begin{equation}\label{eq55}
 		\phi (\log(\phi^{-1}))^{1/2}<\min\{u_1,u_2-u_1,\rho-u_2\},
 	\end{equation}
 	the following bound holds
 	\begin{equation}\label{eq27}			\el^{[0,\rho]}_{2,\bar{x},\bar{y}}\big(\NT^{[0,\rho]}_f,\Cl(\bar{B}_2,\bar{u},\phi)\big)\leq C\log(\phi^{-1}) \phi^6\big[u_1(u_2-u_1)(\rho-u_2)\big]^{-3/2}.
 	\end{equation}
 \end{lemma}
 \begin{proof}
 	The following is an adaptation of the proof in \cite[Lemma 3.7]{Hammond1} to the two point case. We consider first the case  $(x_1-x_2)\vee(y_1-y_2)\leq K$ where $K$ will be determined later. From the affine symmetry of the Brownian bridge it is enough to assume that $\bar{x},\bar{y}\in [0,K]^2_>$. Let us denote $d_{\text{min}}=\min\{u_1,u_2-u_1,\rho-u_2,1/2\}$.  Define the probabilities
 	\begin{equation}
 		\begin{aligned}
 			P_0=&	  \,\el^{[0,\rho]}_{2,\bar{x},\bar{y}}\big(\Cl(\bar{B}_2,\bar{u},\phi)\big)\\
 			P_1=\,&\el^{[0,\rho]}_{2,\bar{x},\bar{y}}\big(\NT^{[0,u_1]}\big|\Cl(\bar{B}_2,\bar{u},\phi)\big)\\
 			P_2=\,&\el^{[0,\rho]}_{2,\bar{x},\bar{y}}\big(\NT^{[u_1,u_2]}\big|\NT^{[0,u_1]}\cap \Cl(\bar{B}_2,\bar{u},\phi)\big)\\
 			P_3=\,&\el^{[0,\rho]}_{2,\bar{x},\bar{y}}\big(\NT^{[u_2,\rho]}\big|\NT^{[0,u_2]}\cap \Cl(\bar{B}_2,\bar{u},\phi)\big).
 		\end{aligned}
 	\end{equation}
 	We therefore have
 	\begin{equation}\label{eq28}
 		\el^{[0,\rho]}_{2,\bar{x},\bar{y}}\big(\NT^{[0,\rho]},\Cl(\bar{B}_2,\bar{u},\phi)\big)=P_0\cdot P_1\cdot P_2 \cdot P_3 
 	\end{equation}
 	We now bound each of the terms in \eqref{eq28}. We first assume that $\bar{x},\bar{y}\in[0,K]^2$. Assume  $\phi<\frac{d_{\text{min}}}4K^{-1}$. From Lemma \ref{lm:6} and basic properties of the Brownian motion, there exists $C>0$ such that 
 	\begin{equation}\label{eq29}
 		\begin{aligned}
 			P_0\leq& C\frac{\phi^2}{\sqrt{u_1(u_2-u_1)(\rho-u_2)}}\\
 			P_1 \leq & \sup_{\substack{\bar{x}\in[0,K]^2_>\\\bar{z}\in[0,\phi]^2_>}} \el^{[0,u_1]}_{2,\bar{x},\bar{z}}\big(\NT^{[0,u_1]}\big) \leq   \phi\cdot u_1^{-1}K(1+30\phi{u_1}^{-1}K)\\
 			P_2 \leq & \sup_{\substack{\bar{w}\in[0,\phi]^2_>\\\bar{z}\in[0,\phi]^2_>}}\el^{[u_1,u_2]}_{2,\bar{w},\bar{z}}\big(\NT^{[u_1,u_2]}\big) \leq  \phi^2(u_2-u_1)^{-1} \big(1+\frac{30\phi^2}{u_2-u_1}\big)\\
 			P_3 \leq & \sup_{\substack{\bar{y}\in[0,K]^2_>\\\bar{z}\in[0,\phi]^2_>}} \el^{[u_2,\rho]}_{2,\bar{z},\bar{y}}\big(\NT^{[u_2,\rho]}\big) \leq   \phi\cdot (\rho-u_2)^{-1}K(1+30 \phi K (\rho-u_2)^{-1}).
 		\end{aligned}
 	\end{equation}
 	Note that the first inequality above follows directly from the distribution of the Brownian bridge and that for the last three inequalities we used the  affine symmetry of the Brownian motion as well as it Markovian property. We note that for the third inequality we used \eqref{eq55}. Using \eqref{eq29} in \eqref{eq28} we obtain the bound
 	\begin{equation}\label{eq56}
 		\el^{[0,\rho]}_{2,\bar{x},\bar{y}}\big(\NT^{[0,\rho]},\Cl(\bar{B}_2,\bar{u},\phi)\big)\leq CK^2 \phi^6\big[u_1(u_2-u_1)(\rho-u_2)\big]^{-3/2}  (1+30\phi d_{\text{min}}^{-1} K)^3.
 	\end{equation}
 	Next we extend the bound for $\bar{x},\bar{y}\in\R^2_>$. Set $K=\big(200\rho\log(\phi^{-1}))^{1/2}$. If $(x_1-x_2)\vee(y_1-y_2)>K$ then, to get close to one another,  at least one of the two Brownian motions $\bar{B}_2(1,\cdot)$ and $\bar{B}_2(2,\cdot)$ must travel a distance of at least  $K/3<(K-1)/2$ over a time interval that is at most $\rho$ long i.e.\ 
 	\begin{equation}\label{eq57}
 		\begin{aligned}
 			&\el^{[0,\rho]}_{2,\bar{x},\bar{y}}(\NT^{[0,\rho]},\Cl(B,\bar{u},\phi))\leq 
 			C\exp\big(-(2\rho)^{-1}(K/3)^2\big)<\phi^{11}.
 		\end{aligned}
 	\end{equation}
 	From \eqref{eq55}, the right hand side of \eqref{eq56} dominates $\phi^{11}$. Moreover, \eqref{eq55} implies that the error term in \eqref{eq56} is bounded.  Combining  the two cases from \eqref{eq57} and \eqref{eq56} completes the proof.
 \end{proof}
 \subsection{Proof of Proposition \ref{prop:1}}
  One of the key ingredients to the proof of Proposition \ref{prop:1} is \cite[Theorem 3.8]{dauvergne2023wiener} (Theorem \ref{thm:6} below).  In a nutshell, here is the content of the result;  for a fixed  $I=(a,b)\subseteq\R$ and $k\in\N$  let $A_{k,I}=\inti{1,k}\times I$. Theorem  \ref{thm:6} states the existence of a process $\gle^{k,I}$ that on $A^c_{k,I}$ behaves like the Airy process $\ele$ and on $A_{k,I}$ essentially behaves like independent Brownian motions that are independent from $\ele$. The precise construction of the line ensemble $\gle^{k,I}$ implies that conditioned on the event  $\mathcal{O}:=\{\gle_1^{k,I}>\gle_2^{k,I}>\ldots\}$, $\gle^{k,I}$ has the same distribution as $\ele$. Importantly, the result gives a lower bound on $\P(\mathcal{O})$ that only depends on the size of the interval $I$. This point is crucial when trying to study the behaviour of $\ele$ on $I$ through the behaviour of $B^{k,I}$.  The precise result in Theorem \ref{thm:6} allows the set $I$ to be finite union of not too large disjoint intervals.
  \begin{remark}
  	In a previous version of the paper we used the 'jump ensemble method' from \cite{Hammond1} to prove Proposition \ref{prop:1}. Although the proof using the  'jump ensemble method' is valid, we opted for the use of Theorem \ref{thm:6} as it results in a much  shorter proof of Proposition \ref{prop:1}.
  \end{remark}   
 \begin{theorem}[{\cite[Theorem 3.8]{dauvergne2023wiener}}]\label{thm:6}
 	Fix $t\geq1,k,l\in\N$ and $\bm{a}\in\R^l$. When $l\geq 2$ we assume that  $a_j+t<a_{j+1}$ for all $j\in\inti{1,l-1}$. Define $U(\bm{a})=\cup_{j=1}^l(\inti{1,k}\times (a_j,a_j+t))$.
 	
 	Then there exists a random sequence of continuous functions $\gle^{\bm{a}}=\gle^{t,k,\bm{a}}=\{\gle^{\bm{a}}_i:\R\rightarrow \N\}$ such that the following points hold:
 	\begin{enumerate}
 		\item Almost surely, $\gle^{\bm{a}}$ satisfies $\gle^{\bm{a}
 		}_i(r)>\gle^{\bm{a}
 		}_{i+1}(r)$ for all pairs $(i,r)\notin U(\bm{a})$.
 		\item The ensemble $\gle^{\bm{a}}$ has the following Gibbs property. For any $m\in\N$ and $a<b$, consider the set $S=\inti{1,m}\times [a,b]$. Then, conditional on the values of $\gle^{\bm{a}}_i(r)$ for $(i,r)\notin S$ the distribution of $\gle^{\bm{a}}_i(r),(i,r)\in S$ is given by $m$ independent Brownian bridges $B_1,\ldots,B_m$ from $(a,\gle^{\bm{a}}_i(a))$ to $(b,\gle^{\bm{a}}_i(b))$ for $i\in\inti{1,m}$, conditioned on the event $B_i(r)>B_{i+1}(r)$ whenever $(i,r)\notin U(\bm{a})$.
 		\item We have 
 		\begin{equation}\label{eq146}
 			\P(\gle^{\bm{a}}_1>\gle^{\bm{a}}_2>\ldots)\geq e^{-C_k(l^5t^3)},
 		\end{equation}
 		where $C_k>0$ depends only on $k$. Moreover, conditional on the event $\big\{\gle^{\bm{a}}_1>\gle^{\bm{a}}_2>\ldots
 		\big\}$, the ensemble $\gle^{\bm{a}}$ is equal in law to the parabolic Airy line ensemble $\ele$.
 	\end{enumerate}
 \end{theorem}
 If $\gle:\inti{1,k}\times[a,b]\rightarrow \R$ is a line ensemble, then for $l\in\inti{1,k}$ we let $\bar{\gle}_k=\gle|_{\inti{1,l}\times[a,b]}$ denote the first $l$ curves. We are now ready to prove Proposition \ref{prop:1}. 
 \begin{proposition}\label{prop:1}
 	For $0<\phi<1/2$   s.t.\ 
 	\begin{equation}
 		\phi (\log(\phi^{-1}))^{1/2}<u_2-u_1,
 	\end{equation}
 	there exists $C>0$ s.t.\
 	\begin{equation} 
 			\P\big(\Cl(\bar{\ele_2},\bar{u},\phi)\big)\leq   C\phi^6(\log\phi^{-1})^4\big(1+(u_2-u_1)^{-3/2}\big).
 	\end{equation}
 \end{proposition}
 \begin{proof}
 	We consider two cases according to the distance $u_2-u_1>0$. Each case will be associated with a vector $\bm{a}\in\R^l$ for some $l\in\{1,2\}$ and the process $\gle^{\bm{a}}$  from Theorem \ref{thm:6}. In both cases we denote the event $\mathcal
 	O_{\bm{a}}=\big\{\gle^{\bm{a}}_1>\gle^{\bm{a}}_2>\ldots
 	\big\}$.		 \vspace{0.3cm}\\
 	
 		  $u_2-u_1\leq 2$.  From Theorem \ref{thm:6} with $k=2$, $l=1$, $a_1=u_1-1$ and $t=u_2-u_1+2$ we see that 
 		\begin{equation}\label{eq144}
 			\begin{aligned}
 					\P\big(\Cl(\bar{\ele_2},\bar{u},\phi)\big)&=	\P\big(\Cl(\bar{\gle^{\bm{a}}_2},\bar{u},\phi)|O_{\bm{a}}\big)= \P\big(\Cl(\bar{\gle^{\bm{a}}_2},\bar{u},\phi)\cap O_{\bm{a}}\big)[\P(O_{\bm{a}})]^{-1}\\
 					& \leq C\P\big(\Cl(\bar{\gle^{\bm{a}}_2},\bar{u},\phi)\cap O_{\bm{a}}\big),
 			\end{aligned}
 		\end{equation}
 		where $C$ equals  the RHS of \eqref{eq146}. Using the Brownian Gibbs property of $\gle^{\bm{a}}$, Lemma \ref{lm:2} with $\rho=2+u_2-u_1$ implies that 
 		\begin{equation}\label{eq142}
 			\begin{aligned}
 				\P\big(\Cl(\bar{\ele_2},\bar{u},\phi)\big)&\leq C \P\big(\Cl(\bar{\gle^{\bm{a}}_2},\bar{u},\phi)\cap O_{\bm{a}}\big)\leq C \P\big(\Cl(\bar{\gle^{\bm{a}}_2},\bar{u},\phi)\cap \{\gle^{\bm{a}}_1>\gle^{\bm{a}}_2\}\big)\\
 				&=C\el^{[u_1-1,u_2+1]}_{2,\bar{x},\bar{y}}\big(\Cl(\bar{B}_2,\bar{u},\phi),\NT^{[u_1-1,u_2+1]}\big)\\
 				&\leq C\log(\phi^{-1}) \phi^6\big(1+(u_2-u_1)^{-3/2}\big),
 			\end{aligned}
 		\end{equation}
 		where $\bar{x}=(\gle_1^{\bm{a}}(u_1-1),\gle_2^{\bm{a}}(u_1-1))$, $\bar{y}=(\gle_1^{\bm{a}}(u_2+1),\gle_2^{\bm{a}}(u_2+1))$ and where we used that $\mathcal{O}_{\bm{a}}\subseteq \{\gle^{\bm{a}}_1>\gle^{\bm{a}}_2\}$.\vspace{0.3cm}\\
 		
 		$u_2-u_1> 2$.  From Theorem \ref{thm:6} with $k=2$, $l=2$, $a_1=u_1-1,a_2=u_2-1$ and $t=2$ we see that 
 		\begin{equation}
 				\begin{aligned}
 					\P\big(\Cl(\bar{\ele_2},\bar{u},\phi)\big)&\leq  C\P\big(\Cl(\bar{\gle^{\bm{a}}_2},\bar{u},\phi)\cap O_{\bm{a}}\big)\leq C\P\big(\Cl(\bar{\gle^{\bm{a}}_2},\bar{u},\phi)\cap \{\gle^{\bm{a}}_1>\gle^{\bm{a}}_2\}\big)\\
 					&=\el^{[u_1-1,u_1+1]}_{2,\bar{x}^1,\bar{y}^1}\big(\NT^{[u_1-1,u_1+1]},\Cl(\bar{B}_2,u_1,\phi)\big)\\
 					&\times \el^{[u_2-1,u_2+1]}_{2,\bar{x}^2,\bar{y}^2}\big(\NT^{[u_2-1,u_2+1]},\Cl(\bar{B}_2,u_2,\phi)\big),
 				\end{aligned}
 		\end{equation}
 		 where $\bar{x}^1=(\gle_1^{\bm{a}}(u_1-1),\gle_2^{\bm{a}}(u_1-1))$, $\bar{y}^1=(\gle_1^{\bm{a}}(u_1+1),\gle_2^{\bm{a}}(u_1+1))$, $\bar{x}^2=(\gle_1^{\bm{a}}(u_2-1),\gle_2^{\bm{a}}(u_2-1))$, $\bar{y}^2=(\gle_1^{\bm{a}}(u_2+1),\gle_2^{\bm{a}}(u_2+1))$. Using \cite[Lemma 3.7]{Hammond1} in the last display we conclude that 
 		 \begin{equation}\label{eq143}
 		 	\P\big(\Cl(\bar{\gle_2},\bar{u},\phi)\big)\leq \big[C\phi^3(\log\phi^{-1})^2\big]^2=C\phi^6(\log\phi^{-1})^4.
 		 \end{equation}
 		 Combining \eqref{eq142}  with  \eqref{eq143} implies the result. 
 \end{proof}

\appendix

\section{Proof of \eqref{eq138}}\label{sec:Some proofs}
\begin{lemma}\label{lm:8}
	For $i\in\Ind$, let $\cone_i$ be the trapezoid defined in \eqref{eq119}. Then
	\begin{equation}
		\P\big(\disjc_{\cone_i}^3\big)=\P\big(\disjc_{\cone_0}^3\big), \qquad \forall i\in\Ind.
	\end{equation}
\end{lemma}
\begin{proof}
	For $c\in\R$, let $\cone_c=\trpz\big((0,0),\epsilon;(c,1),T\big)$. We claim that it is enough to show that for any fixed $c\in\R$
	\begin{equation}\label{eq48}
		\big\{ \pi_p^q\big\}_{p\in\coned_c,\,q\in \coneu_c}\sim \big\{ \pi_p^q+L^c\big\}_{p\in\coned_0,\,q\in \coneu_0},
	\end{equation}
	where $L^c$ is the map $r\mapsto cr$. In words, we would like to show that the distribution of all geodesics in the trapezoid $\cone_c$ is the same as that as the geodesics in the trapezoid $\cone_0$ up to an affine transformation. The equality in distribution in \eqref{eq48} is indeed enough to imply the result as the event where geodesics meet is invariant under affine transformation i.e.\ for paths $\pi,\eta:[0,1]\rightarrow \R$ it holds that 
	\begin{equation}
		\text{$\pi$ and $\eta$ are disjoint} \iff \text{$\pi+L^c$ and $\eta+L^c$ are disjoint}
	\end{equation}
	 From the definition of the geodesic, in distribution, as a process in $x,y\in \R$ and  $r\in(0,1)$
	\begin{equation}
		\begin{aligned}
			\pi_{(x,0)}^{(y+c,1)}(r)=&\argmax{z\in\R}\,\dl(x,0;z,r)+\dl(z,r;y+c,1)\\
			=&cr+\argmax{z\in\R}\,\dl(x,0;z+cr,r)+\dl(z+cr,r;y+c,1)\\
			\sim&\,cr+\argmax{z\in\R}\,[\dl(x,0;z,r)-2c(x-z)+c^2r]\\
			&+[\dl(z,r;y,1)-2c(z-y)+c^2(1-r)]\\
			=&cr+\argmax{z\in\R}\,\dl(x,0;z,r)+\dl(z,r;y,1)-2c(x-y)+c^2\\
			=&cr+\argmax{z\in\R}\,\dl(x,0;z,r)+\dl(z,r;y,1) = cr+\pi_x^y(r),
		\end{aligned}
	\end{equation}
where in the third line we used the stationary property from Lemma \ref{lm:prp}. We have proven  \eqref{eq48} and with that the required result.
\end{proof}
\bibliographystyle{alpha}
\bibliography{references_file}
\end{document}